\documentclass[11pt]{amsart}

\usepackage{amssymb,amsmath,amsthm,amscd,epsf,latexsym,verbatim,graphicx,amsfonts,hyperref,epstopdf,xcolor,wasysym}
\usepackage{enumitem}

\usepackage{tikz-cd}

\usepackage{graphicx}
\usepackage{tabularx}
\usepackage[margin=0.8in]{geometry}         
      
\usepackage{subcaption}
\usepackage{amscd}

\usepackage{hyperref}
\usepackage{tikz}
\usepackage{tkz-graph}

\usetikzlibrary{arrows}
\usetikzlibrary{shapes.misc}
\tikzset{cross/.style={cross out, draw=black, fill=none, minimum size=2*(#1-\pgflinewidth), inner sep=0pt, outer sep=0pt}, cross/.default={2pt}}

\usepackage{mathtools}
\DeclarePairedDelimiter{\ideal}{\langle}{\rangle}

\theoremstyle{theorem}
\newtheorem{theorem}{Theorem}[section]
\newtheorem{lemma}[theorem]{Lemma}
\newtheorem{proposition}[theorem]{Proposition}

\theoremstyle{definition}
\newtheorem{defn}[theorem]{Definition}

\newtheorem*{notn}{Notation}

\newdimen\arrowsize 
\pgfarrowsdeclare{arcs}{arcs} 
{ 
\arrowsize=0.2pt 
\advance\arrowsize by .5\pgflinewidth 
\pgfarrowsleftextend{-4\arrowsize-.5\pgflinewidth} 
\pgfarrowsrightextend{.5\pgflinewidth} 
} 
{ 
\arrowsize=0.2pt 
\advance\arrowsize by .3\pgflinewidth 
\pgfsetdash{}{0pt} 
\pgfsetroundjoin 
\pgfsetroundcap 
\pgfpathmoveto{\pgfpoint{-17\arrowsize}{5\arrowsize}} 
\pgfpatharc{180}{270}{17\arrowsize and 5\arrowsize} 
\pgfusepathqstroke 
\pgfpathmoveto{\pgfpointorigin} 
\pgfpatharc{90}{180}{17\arrowsize and 5\arrowsize} 
\pgfusepathqstroke 
}

\newcommand{\ZZ}{\mathbb{Z}}
\newcommand{\RR}{\mathbb{R}}

\newcommand{\CC}{\mathbb{C}}
\newcommand{\NN}{\mathbb{N}}

\newcommand{\supp}{\text{supp}}
\newcommand{\Per}{\text{Per}}
\newcommand{\AbsPer}{\operatorname{AbsPer}}
\newcommand{\Id}{\text{Id}}

\newcommand{\calH}{\mathcal{H}}
\newcommand{\modd}{\text{ mod }}
\newcommand{\TT}{\mathbb{T}}
\newcommand{\calC}{\mathcal{C}}

\newcommand{\mon}{\text{Mon}}

\newcommand{\Mon}{\text{Mon}}

\usepackage{stackengine,scalerel}
\usepackage{calc}
\newlength\shlength
\newcommand\xshlongvec[2][0]{\ThisStyle{\setlength\shlength{#1\LMpt}%
  \stackengine{-5.6\LMpt}{$\SavedStyle#2$}{\smash{$\kern\shlength%
    \stackengine{\dimexpr 1.3pt+6.25\LMpt}{$\SavedStyle\mathchar"017E$}%
      {\rule{\widthof{$\SavedStyle#2$}}{\dimexpr.1pt+.5\LMpt}\kern.4\LMpt}{O}{r}{F}{F}{L}\kern-\shlength$}}%
      {O}{c}{F}{T}{S}}}

\title{Counting formulae for square-tiled surfaces in genus two}
\author{Sunrose Shrestha}

\begin{document}
\maketitle

\begin{abstract}
Square-tiled surfaces can be classified by their number of squares and their cylinder diagrams (also called realizable separatrix diagrams).  For the case of $n$ squares and two cone points with angle $4 \pi$ each, we set up and parametrize the classification into four diagrams.  Our main result is to provide formulae for enumeration of square-tiled surfaces in these four diagrams, completing the detailed count for genus two.   The formulae are in terms of various well-studied arithmetic functions, enabling us to give asymptotics for each diagram using a new calculation for additive convolutions of divisor functions that was recently derived by the author and collaborators. Interestingly, two of the four cylinder diagrams occur with asymptotic density 1/4, but the other diagrams occur with different (and irrational) densities.
\end{abstract}

\section{Introduction}

The main result of this paper is enumeration of the number of primitive (connected) square-tiled translation surfaces in the stratum $\calH(1,1)$ by their cylinder diagrams. 

Recall that a square-tiled translation surface is a closed orientable surface built out of unit-area axis-parallel Euclidean squares glued along edges via translations. Square-tiled surfaces of genus $>1$ are ramified covers (with branching over exactly one point) of the standard square torus.  The principal stratum $\calH(1,1)$ contains genus two translation surfaces with two simple cone points. 
(The only other possibility is a single cone point with more angle excess.)
A \emph{primitive} square-tiled surface is one that covers the standard  torus with no other square-tiled surface as an intermediate cover.

Every square-tiled  surface is  built out of horizontal square-tiled cylinders. 
We can define {\em cylinder diagrams} 
 (ribbon graphs with a pairing on the boundary components) which keep track of the number of cylinders and the gluing patterns along their sides.
 In Section~\ref{sec:cyldiag} we identify the four cylinder diagrams for $\calH(1,1)$. A square-tiled surface with $n$ squares will be referred to as an \emph{$n$-square surface.}
 
We now state our main result.
\begin{theorem}\label{thm:maincounts} Let $A(n), B(n), C(n)$ and  $D(n)$ count the number of primitive $n$-square surfaces in $\calH(1,1)$ with cylinder diagram $A$, $B$, $C$ and $D$, respectively. Let $E(n)$ be the total number of primitive $n$-square surfaces in $\calH(1,1)$. Then 
$$E(n) = \frac{1}{6}(n-2)(n-3)J_2(n)$$ and this breaks down as 
$$\begin{array}{llc}
{\rm Diagram} & {\rm Formula} & {\rm Density}\\
\hline\\
\tikz{\draw [fill=gray!50] (0,0) rectangle (1,1/3);}
&A(n) =\frac{1}{24}n(n-6)J_2(n) + \frac{1}{2}nJ_1(n)& \frac{1}{4}E(n)\\

\tikz{\draw [fill=gray!50]  (0,0)--(1,0)--(1,.2)--(.3,.2)--(.3,.5)--(0,.5)--cycle;   }
&
             B(n)= \bigl((\mu \cdot \sigma_2) * (\sigma_1 \Delta \sigma_2)\bigr)(n) - \frac{1}{24}(2n+9)(n-2)J_2(n) - \frac{1}{2}nJ_1(n) &
			 \left(\frac{\zeta(2)\zeta(3)}{2\zeta(5)}-\frac{1}{2}\right)E(n) \\
\raisebox{-.2in}{\tikz{\draw [fill=gray!50]  (1/3,0)--(1,0)--(1,.5)--(.8,.5)--(.8,1)--(0,1)--(0,.5)--(1/3,.5)--cycle;   }}
&C(n) = \frac{1}{24}(n-2)(n-3)J_2(n) & \frac{1}{4} E(n)\\[15pt]
\raisebox{-.2in}{\tikz{\draw [fill=gray!50]  (1/3,0)--(1,0)--(1,2/3)--(1/3,2/3)--(1/3,1)--(0,1)--(0,1/6)--(1/3,1/6)--cycle;   }}
&D(n) = \frac{1}{6}n(n-1)J_2(n) - \bigl((\mu \cdot \sigma_2) * (\sigma_1 \Delta \sigma_2)\bigr)(n)
& \left(1- \frac{\zeta(2)\zeta(3)}{2\zeta(5)}\right)E(n)
\end{array}$$
Furthermore, while surfaces with diagram A  and diagram C are each asymptotic to one-quarter of the total, the surfaces with other diagrams have unequal (and irrational) densitites.
$$A(n)/E(n) \to  0.25 \ ; \qquad B(n)/E(n) \to 0.453... \  ; \qquad C(n)/E(n) = 0.25 \  ; \qquad D(n)/E(n) \to 0.047... $$
\end{theorem}

The statement of our main theorem uses standard notation for  arithmetic functions:  $J_1$ and $J_2$ are Jordan totient functions; $\mu$ is the M\"obius function;
$\sigma_1$ and $\sigma_2$ are divisor functions; and $\zeta$ is the Riemann zeta function.  The symbol $*$ denotes the Dirichlet convolution, and $\Delta$ is 
additive convolution.  
For more detailed definitions, see Appendix \ref{sec:arithmetic}. 

Since $\frac{6}{\pi^2}n^2 < J_2(n) \leq n^2$, the number of primitive $n$-square surfaces in $\calH(1,1)$ grows at least as fast as $\frac{1}{\pi^2} n^4$. 
The total count $E(n)$ was already known, and can be found in the work of Bloch--Okounkov \cite{blochokounkov} and Dijkgraaf \cite{dijkgraaf}. 
The novelty in our result is that we get a more detailed count, by cylinder diagrams, that allows us to obtain the individual asymptotic densities as well.

Figure~\ref{fig:examplevalues} shows the share of $E(n)$ by cylinder diagram for $4\le n\le 101$.  
We see  erratic but steady convergence in the direction of the asymptotics from Theorem \ref{thm:maincounts}.

\begin{figure}[h!!]
\includegraphics[scale=0.75]{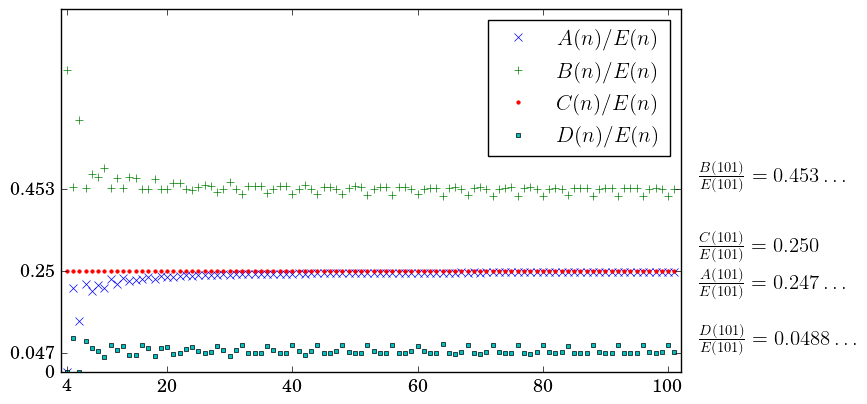}
\caption{Scatter plot for $\frac{A(n)}{E(n)}$,  $\frac{B(n)}{E(n)}$, $\frac{C(n)}{E(n)}$, $\frac{D(n)}{E(n)}$ for $4 \leq n \leq 101$}
\label{fig:examplevalues}
\end{figure}

The enumeration by cylinder diagram of primitive $n$-square surfaces in $\calH(2)$ was done in unpublished work of  Zmiaikou \cite{zmprob}. 
Let $H(n) = \frac{3}{8}(n-2)J_2(n)$ be the total number of primitive square-tiled surfaces in $\calH(2)$.
There are two cylinder diagrams for surfaces in $\calH(2)$, which 
we can denote by $F$ and $G$, and the number of primitive $n$-square surfaces with each diagram is then given as follows.
\begin{gather*}
F(n)= \frac{1}{6}n J_2(n) - \frac{1}{2}nJ_1(n) \sim \frac{4}{9}H(n);\\
G(n)= \frac{5}{24}n J_2(n) + \frac{1}{2}nJ_1(n) - \frac{3}{4}J_2(n) \sim \frac{5}{9}H(n).
\end{gather*}
The main theorem of this paper complements Zmiaikou's work, completing the enumeration of primitive square-tiled surfaces of genus 2 by cylinder diagram.

\subsection*{Relationship to other results}

This work fits into a significant body of literature on enumeration of square-tiled surfaces.  
Several papers focus on the classification by orbits in the $SL_2(\ZZ)$ action on square-tiled surfaces.
Combined work of Hubert--Leli\`evre \cite{hubertlelievre} and McMullen \cite{mcmullen} shows that for $n\ge 3$, there are either one or two $SL_2(\ZZ)$ orbits
for primitive $n$-square surfaces in $\calH(2)$.
Subsequently, Leli\`evre--Royer \cite{lelievreroyer} obtained formulae for enumerating these orbit-wise. 
They also prove that the generating functions for these countings are quasimodular forms. Dijkgraaf \cite{dijkgraaf} gave generating functions for the number of $n$-sheeted covers of genus $g$ of the square torus with simple ramification over distinct points, and Bloch--Okounkov \cite{blochokounkov} studied this problem for arbitrary ramification. 

Square-tiled surfaces (not necessarily primitive) were also counted by their cylinder diagrams by Zorich \cite{zorichsqtiled} who applied the counts to compute the Masur-Veech volume of certain small genus strata. Eskin--Okounkov \cite{eskinokounkov} generalized Zorich's work and used such counts to obtain formulae for the Masur-Veech volume of all strata. More recently Delecroix--Goujard--Zograf--Zorich \cite{delgouzogzor} refined this result, and studied the absolute and relative contributions of square-tiled surfaces with fixed number of cylinders in their cylinder diagrams to the Masur-Veech volume of the ambient strata. They got a general formulae (albeit, not closed) for any number of cylinders and any strata. Moreover, they obtained closed sharp upper and lower bounds of the absolute contributions of 1-cylinder surfaces to the volume of any stratum and in the cases where the stratum is either $\calH(2g-2)$ or $\calH(1, 1, \dots, 1)$ they obtained closed exact formulae. In the same paper, they also proved that square-tiled surfaces with a fixed cylinder diagram equidistribute in the ambient stratum.

Eskin--Masur--Schmoll \cite{eskinmasurschmoll} on the other hand used counts of primitive square-tiled surfaces in $\calH(2)$ and $\calH(1,1)$ to obtain the asymptotics for the number of closed orbits for billiards in a square table with a barrier.
\subsection*{Proof strategy and structure of paper}

The paper is organized as follows. The background in Section \ref{sec:background} covers generalities on square-tiled surfaces and introduces the {\em monodromy group}
$\Mon(S)\le S_n$ as a key tool that records the structure of gluings of the sides of $S$ in terms of permutations of squares.  
For the remaining sections, we have streamlined the presentation by creating appendices with detailed but routine calculations.
In Section \ref{sec:cyldiag} we give statements of how the four cylinder diagrams in $\calH(1,1)$ will be parametrized.  
(Full proofs that support these parametrizations are found in the Appendix~\ref{sec:cyldiagclass}.)
In Section \ref{sec:primitivity} we state number theoretic criteria on the parameters obtained in Section \ref{sec:cyldiag} which characterize primitivity. 
(Full proofs appear in Appendix~\ref{sec:primcriterion}.)
  In Section \ref{sec:enumeration} we manipulate the sums that appear from the number theoretic criteria to deduce the counting formulae in the main theorem.
 (Full work showing simplifications of intermediate sums appears in Appendix~\ref{sec:intermediatesums}.)
Finally, in Section \ref{sec:proportions} we complete the proof of the main theorem by computing asymptotic densities for each cylinder diagram. (Appendix \ref{sec:arithmetic} contains background on some arithmetic functions, and number theoretic identities used during the enumerations.) 

The main idea in the proof is to take advantage of the key fact that 
a connected $n$-square surface is primitive if and only if the associated monodromy group, $\Mon(S)$, satisfies an algebraic condition also called primitivity, which is described 
in terms of the orbits of its action on the $n$ squares.  (See Section~\ref{sec:background}.)
The number theoretic conditions result directly from this algebraic characterization, and much of the rest of the work is in cleverly handling the sums involving arithmetic functions.

Our methods can be extended directly to enumeration of primitive square-tiled surfaces in certain strata of higher genus, with asymptotic proportions by cylinder diagrams. 
However, the complexity becomes forbidding.  Even $\calH(4)$, which is the smallest stratum in genus 3, already has 22 different cylinder diagrams
(as shown by S. Leli\`evre in \cite{lelievre} as an Appendix to \cite{matheusmolleryoccoz}.)
A more tractable starting point might be to consider the hyperelliptic component of $\calH(4)$, which has just 5 cylinder diagrams. 
On the other hand, there are clear limitations on the generality of this counting method:  it is known that in $\calH(1,1,1,1)$, the principal stratum of genus 3, 
the lattice of absolute periods does not pick up primitivity, so new ideas would be needed.

\subsection*{Acknowledgements} We are very grateful for Moon Duchin for suggesting this problem and advising us throughout the project. We are also very grateful to Samuel Leli\`evre for many insightful conversations, comments, and suggestions. We also acknowledge David Zmiaikou, Robert Lemke-Oliver and Frank Thorne for helpful conversations.


\section{Background}\label{sec:background}

\subsection{Square-tiled Surfaces}

We begin this section by defining square-tiled surfaces more rigorously.
\begin{defn}[Square-tiled surface] \label{defn:sqtiled} A \emph{square-tiled surface} is a closed orientable surface obtained from the union of finitely many Euclidean axis parallel unit area squares $\{\Delta_1, \dots, \Delta_n\}$ such that
\begin{itemize}
\item the embedding of the squares in $\RR^2$ is fixed only up to translation,

\item after orienting the boundary of every square counterclockwise, for every $1 \leq j \leq n$, for every oriented side $s_j$ of $\Delta_j$, there exists a $1 \leq k \leq n$, and an oriented side $s_k$ of $\Delta_k$ so that $s_j$ and $s_k$ are parallel, and of opposite orientation. The sides $s_j$ and $s_k$ are glued together by a parallel translation. 
\end{itemize}
\end{defn}

A few key things that follow from the definition:
\begin{itemize}
\item The orientations of the glued edges $s_j$ and $s_k$ are opposite so that as one moves along the glued side, $\Delta_j$ appears to the left and $\Delta_k$ appears to the right (or vice versa). 
\item The total angle around a vertex is $2 \pi c$ for some positive integer $c$. When $c > 1$, we call the point a cone point. \item Since the squares are embedded in $\RR^2$ up to translation, we distinguish between two squares if one is obtained from the other by a nontrivial rotation. However, two squares are ``cut, parallel transport, and paste" equivalent. Hence, square-tiled surfaces come with a well defined vertical direction.
\end{itemize}

Considering the squares as embedded in $\CC$, one can give a square-tiled surface a complex structure as well. Moreover, since the gluing of the sides is by translation, the transition functions are translations:
$$ z \rightarrow z' + C$$

Hence, the 1-form $dz$ on $\CC$ gives rise to a holomorphic 1-form $\omega$ on $S$ so that locally, $\omega = dz$. This is well defined since at a chart where the coordinate function is $z'$, $\omega$ takes the form $\omega = dz'$. Around the cone points, up to a change of coordinates, the coordinate functions are $z^{k+1}$ so that $\omega$ takes the form $\omega = z^{k} dz$. Hence, the cone points are zeros of $\omega$, and if the angle around the cone point is $2\pi (k+1)$, then the degree of the zero is $k$.

Given the angles around each cone point, and the number of cone points, one can also recover the genus of the square-tiled surface. Note that since the squares are Euclidean, square-tiled surfaces are flat everywhere except at the cone points. Hence, the classical Gauss-Bonnet theorem takes a relatively simple form to give: 
$$\int \kappa dA = 2 \pi \chi(S) \implies - 2\pi \sum_{i=1}^m k_i = 2 \pi (2-2g) \implies \sum_{i=1}^m k_i = 2g - 2$$
where $\kappa$ is Gaussian curvature, $\chi(S)$ is the Euler characteristic of $S$, and the sum is over the cone points with angles $2\pi(k_i+1)$. 

More generally, if we allow arbitrary Euclidean polygons in definition \ref{defn:sqtiled}, we get \emph{translation surfaces} which are a class of surfaces of which square-tiled surfaces are a particular subset. 

For translation surfaces, let $\alpha = (k_1, \dots, k_m)$ be the integer vector that records the angle data of the cone points so that there are $m$ cone points and each cone point has angle $2\pi(k_i+1)$. Since the genus of the surface is recovered using this data, the space of genus $g$ translation surfaces is stratified with surfaces sharing the same angle data $\alpha =(k_1, \dots, k_m)$ for the various integer partitions of $2g-2$. These are called \emph{strata} and are denoted $\calH(\alpha)$.

 Let $\TT = \CC/(\ZZ +i\ZZ)$ be the standard torus. Given a square-tiled surface $S$ (with a holomorphic 1-form $\omega$), we know that the cone points are in the integer lattice, and hence we get a map,
$$ \pi :S \rightarrow \TT$$ by 
$$ p \rightarrow \int_{P_1}^p \omega \mod \ZZ+i\ZZ$$
where $\{P_1, \dots, P_m\}$ is the set of cone points of $S$. $\pi$ is holomorphic and onto, and hence it is a ramified covering where the ramification points are exactly the zeros of $\omega$ (or the cone points) which project to $0 \in \TT$.

Given an element $\rho$ in the relative homology group $H_1(S, \{P_1, \dots, P_m\};\ZZ)$ of a square-tiled surface $S$, we call $\int_\rho \omega$ a \emph{period}. Since the cone points are in the integer lattice, all periods are in $\ZZ+i\ZZ \simeq \ZZ^2$.

Since square-tiled surfaces are cut, parallel transport and paste equivalent, their representation in the plane is not unique. In particular, square-tiled surfaces can be represented by parallelograms as well. We call such representations \emph{unfolded representations}. See Figure \ref{fig:unfoldedrepresentation} for a basic example. 

\begin{figure}[h!]
\includegraphics[scale=0.5]{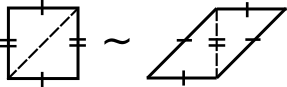}
\caption{An unfolded representation for the standard torus}
\label{fig:unfoldedrepresentation}
\end{figure}
Next we define some geometrical objects of interest on these surfaces.
\begin{defn}[Saddle connection] A \emph{saddle connection} in a translation surface $S$ is a curve $\gamma: [0,1] \rightarrow S$ such that $\gamma(0)$ and $\gamma(1)$ are cone points, but $\gamma(s)$ is not a cone point for any $0 < s < 1$.
\end{defn}

\begin{defn}[Holonomy vector] The \emph{holonomy vector} associated to a saddle connection $\gamma$ in a translation surface is the relative period $v = \int_\gamma \omega$ viewed as a vector in  $\RR^2$.
\end{defn}

Geometrically, a holonomy vector of a saddle connection $\gamma$ records the Euclidean horizontal and vertical displacement of a geodesic representative of $\gamma$. If $(v_1 ,v_2)$ is a holonomy vector of a saddle connection $\gamma$, then $v_1$ will be called the \emph{horizontal holonomy} of $\gamma$ and $v_2$ will be called the \emph{vertical holonomy} of $\gamma$.

\begin{defn}[Lattice of periods] The \emph{lattice of periods} of a square-tiled surface $S$, denoted as $\Per(S)$, is the rank 2 sublattice of $\ZZ^2$ holonomy vectors of $S$.
\end{defn}

\begin{defn}[Lattice of absolute periods] The \emph{lattice of absolute periods} of a square-tiled surface $S$, denoted as $\AbsPer(S)$, is the rank 2 sublattice of $\ZZ^2$ generated by the holonomy vectors of closed saddle connections of $S$.
\end{defn}

Note that $\AbsPer(S) \subset \Per(S)$.

We say that a square-tiled surface $S$ covers $S'$ if the following diagram commutes for a ramified covering $\pi$. $S$ will be called a \emph{proper ramified covering} of $S'$ if $\pi$ and $h'$ have covering degree $> 1$. 

$$\begin{tikzcd}[column sep=small] 
S  \arrow{dr}[swap]{h} \arrow{rr}{\pi} 
&& S' \arrow{dl}{h'} \\
&
\TT
\end{tikzcd}
$$

This motivates the following definition:

\begin{defn}[Primitive square-tiled surface] We call a square-tiled surface \emph{primitive} it is not a proper ramified covering of any other square-tiled surface.
\end{defn}

\subsection{The Monodromy Group}
Given an $n$-square surface $S$, first fix a labelling of the squares by $\{1, \dots, n \}$. Square-tiled surfaces, come with a well defined vertical direction, and the squares used to make the surface are axis parallel. Hence, for any square, there is a well defined notion of top, right, bottom and left neighboring squares. We associate two permutations, $\sigma$ and $\tau$, to $S$ defined by 
$$ \sigma(i) = j \iff \text{ right side of square }i\text{ is glued to the left side of square }j$$
$$ \tau(i) = j \iff \text{ top side of square }i\text{ is glued to the bottom side of square }j$$

The permutations $\sigma$ and $\tau$ describe completely how to glue the squares to form $S$. $\sigma$ is referred to as the \emph{right permutation} associated to $S$ and $\tau$ is referred to as the \emph{top permutation} associated to $S$. 

If the labelling on $S$ is changed by a permutation $\gamma \in S_n$ (the symmetric group on $n$ objects) so that square $i$ is now labelled $\gamma(i)$, then one checks that the associated right and top permutations we get for the newly labelled $S$, are $\gamma \sigma \gamma^{-1}$ and $\gamma \tau \gamma^{-1}$.

Hence, given an unlabelled square-tiled surface $S$ we can obtain a simultaneous conjugacy class of a pair of permutations in $S_n$ as described above. One checks that the converse is true: given a simultaneous conjugacy class in $S_n \times S_n$, one can associate uniquely, a (possibly disconnected) square-tiled surface. 

Notationally, given a pair of permutations $(\sigma, \tau)$ we will denote $S(\sigma, \tau)$ as the square-tiled surface that has $\sigma$ and $\tau$ as its right and top permutations. Note that fixing the conjugacy class representative $\sigma$, fixes a labelling of $S$, so that $S(\sigma, \tau)$ comes with a labelling. 

The commutator $[\sigma,\tau]$ is also of interest, since its cycle type defines the topological type of $S$. We will state the following known propositions, to this effect. 

\begin{proposition}\label{prop:lowerleft}
Let $\sigma, \tau \in S_n$ such that $[\sigma, \tau]$ is a product of disjoint nontrivial cycles $u_1, \dots, u_m$ of lengths $(k_i+1), \dots, (k_m+1)$. Consider $S(\sigma, \tau)$ endowed with a labelling prescribed by $\sigma$. Then, for labelled squares $x \neq y$, their lower left corners are identified if and only if and each $x,y \in \supp(u_j)$ for some $j$.
\end{proposition}
Recall that the support of a cycle $u = (a_1\dots a_j)$, denoted $\supp(u)$, is the set $\{a_1, \dots, a_j\}$ of elements that are nontrivially moved by $u$. 
Knowing Proposition \ref{prop:lowerleft}, it follows that the cycle type of $[\sigma, \tau]$ determines the stratum:
\begin{proposition} $S(\sigma, \tau) \in \calH(k_1, \dots, k_m)$
\end{proposition}

The group $\ideal{\sigma, \tau}$ is also of interest, and has a name.

\begin{defn}[Monodromy Group] Let $S$ be a labelled square-tiled surface with $\sigma, \tau \in S_n$ as the top and right permutations. We define the \emph{monodromy group}, denoted $\Mon(S)$, as the subgroup of $S_n$ generated by $\sigma$ and $\tau$ (upto conjugacy). 
\end{defn}

Note that relabelling $S$ results in an isomorphic monodromy group. Note also that $S$ is connected if and only if $\Mon(S)$ acts transitively on the set $\{1, \dots, n\}$ with the natural permutation action. 

Next we define the notion of primitivity for a subgroup of $S_n$. 

\begin{defn}[Blocks] A non-empty subset $\Delta \in \{1, \dots, n\}$ is a \emph{block} for a subgroup $H \subset S_n$ if for all $h \in H$, either $h(\Delta) = \Delta$ or $h(\Delta) \cap \Delta = \emptyset$.  
\end{defn}
 
Note that singletons and the whole set $\{1, \dots n\}$ are always blocks. These will be called \emph{trivial blocks}.

\begin{defn}[Primitive subgroup] A subgroup $H \subset S_n$ is \emph{primitive} if $H$ has no nontrivial blocks.
\end{defn}

We then have the following proposition, that bridges the two notions of primitivity we have seen so far:

\begin{proposition}[\cite{zmthesis}]\label{prop:primsurfprimgroup} A connected $n$-square surface $S$ is primitive if and only if $\mon(S) \subset S_n$ is primitive.
\end{proposition}

We will use this proposition in Section \ref{sec:primitivity} to characterize primitivity of square-tiled surfaces in $\calH(1,1)$ in terms of geometric parameters that define them.

\section{Cylinder Diagrams}\label{sec:cyldiag}
In this section, we study the geometry of square-tiled surfaces by studying their system of horizontal saddle connections. Let $S$ be a square-tiled surface and consider a graph $\Gamma$ on $S$ with the vertex set as the cone points of $S$, and horizontal saddle connections as the edge set. Since every square-tiled surface has a complete cylinder decomposition in the horizontal direction, the complement $S \setminus \Gamma$ is a collection of horizontal cylinders of $S$. 

Since $S$ is an orientable surface, we take an orientation on $S$ and endow the horizontal foliation of $S$ with a compatible orientation. This orientation induces an orientation on the edges of $\Gamma$, and hence we obtain an oriented graph. Moreover, given any vertex $v$ of $\Gamma$, as $S$ is oriented, we get a cyclic order on the edges incident to $v$. Since $S$ is oriented, the orientation of the edges incident to $v$ alternate between orientation towards and away from $v$ as we move counterclockwise.

Taking an $\epsilon$ neighborhood of the edges of $\Gamma$, we obtain a \emph{ribbon graph} $R(\Gamma)$ which is a collection of oriented strips glued as per the cyclic ordering on the vertices of $\Gamma$. $R(\Gamma)$ is an orientable surface with boundary. There are two orientations that are induced on the boundary components of $R(\Gamma)$. The first orientation is the canonical orientation on the boundary components coming from the orientation of $R(\Gamma)$ as an oriented surface. The second orientation is induced by the oriented edges of $\Gamma$. Note that these orientations do not necessarily match on all boundary components. We say that a boundary component is \emph{positively oriented} if the two notions give the same orientation, and \emph{negatively oriented} if the two notions do not match.

The complement, $S \setminus R(\Gamma)$ is then a union of flat cylinders, each of whose two boundaries are glued to boundary components of $R(\Gamma)$, one positively oriented and one negatively oriented. Hence, the boundary components of $R(\Gamma)$ decompose into pairs so that each of the components in a pair bound the same cylinder in $S$, and have opposite signs of orientation. This motivates the following definition:

\begin{defn}[Separatrix Diagram] A \emph{separatrix diagram} is a pair $(R(\Gamma), P)$, where $\Gamma$ is a finite directed graph and $P$ is a pairing on the boundary components of the associated ribbon graph $R(\Gamma)$ such that 
\begin{enumerate}
\item edges incident to each vertex are cyclically ordered, with orientations alternating between to and from the vertex.
\item the boundary components of $R(\Gamma)$, in each pair defined by $P$, are of opposite orientation.
\end{enumerate}
\end{defn}
 
Note that from the process described above, to each square-tiled surface one can associate a separatrix diagram. Conversely, from each separatrix diagram, one obtains a closed orientable surface by gluing in topological cylinders between the paired boundary components. 

However, in order for the resulting surface to be a square-tiled surface, we first assign real variables representing lengths to the edges of $\Gamma$. Then each boundary component of $R(\Gamma)$ has as its length, the sum of the lengths of saddle connections that run parallel to it. To obtain a square-tiled surface from a separatrix diagram $(R(\Gamma), P)$, in each pair of boundary components determined by $P$, the lengths of the boundary components in the pair must be equal (so that one can glue in a metric cylinder with these boundary components). Imposing this condition, we obtain a system of linear equations with variables as the lengths of saddle connections. Hence, a square-tiled surface is obtained if and only if there exists a solution with positive integer lengths. 

\begin{defn}[Cylinder Diagram] A cylinder diagram is a realizable separatrix diagram.
\end{defn}

For more details, see \cite{zorich}.

We next classify the cylinder diagrams of surfaces in $\calH(1,1)$, the proof of which we present in Appendix \ref{sec:cyldiagclass}

\begin{proposition}\label{prop:cyldiagclass} There are 4 distinct cylinder diagrams in $\calH(1,1)$. 
\end{proposition}

Using Proposition \ref{prop:cyldiagclass} we make Figure \ref{fig:cylindertypes} that show the cylinder diagrams in $\calH(1,1)$ and prototypical surfaces arising from them and the parameters associated to these surfaces. The gluings are indicated by the dotted lines, and if dotted lines are missing then the gluing is by the obvious opposite side horizontal or vertical translation. We parametrize each of these prototypes by the lengths and heights of the cylinders, the lengths of the horizontal saddle connections, and the amount of shear (twist) on the cylinders. In all of the parametrization, $p, q, r$ are heights of cylinders, $j, k, l, m$ are lengths of horizontal saddle connections, and $\alpha, \beta, \gamma$ are shears in the cylinders.

As shown in the Figure \ref{fig:cylindertypes}:
\begin{itemize}
\item Cylinder diagram A is parametrized by $(p, j, k, l, m, \alpha)$ 
\item Cylinder diagram B is parametrized by $(p, q, k, l, m, \alpha, \beta)$ where $\alpha$ is the shear in the longer cylinder and $\beta$ the shear in the shorter cylinder,
\item Cylinder diagram C is parametrized by $(p, q, k, l, m, \alpha, \beta)$ where $\alpha$ is the shear in the cylinder of width $p$ and $\beta$ is the shear in the cylinder of width $q$.
\item Cylinder diagram D is parametrized by $(p, q, r, k, l, \alpha, \beta, \gamma)$ where $\alpha$, $\beta$, $\gamma$ are the shears in the cylinders with width $p, q, r$ respectively.
\end{itemize}

We note that these parameters are not unique as stated. In order to use them to count the surfaces, we first need a unique set of parameters for each cylinder diagram. The following propositions give unique parametrizations of the these surfaces:

\begin{proposition}[Uniqueness of cylinder diagram A parameters]\label{prop:typeIparam}
Let $(p,j,k,l,m, \alpha) \in \NN^6$ such that $p (j+k+l+m) = n$. The set of such tuples, $ \Sigma_A = S_1 \cup S_2 \cup S_3 \cup S_4 \cup S_5$, where
\begin{gather*}
S_1 := \{ (p, j, k, l, m, \alpha)\mid j< k, l,  m, 0 \leq \alpha < n/p\} \hspace{1cm} S_2 := \{ (p, j, k, l, m, \alpha) \mid j =l < k \leq m,  0 \leq \alpha < n/p\}\\
S_3 := \{ (p, j, k, l, m, \alpha) \mid j =k < l, m,  0 \leq \alpha < n/p\} \hspace{1cm} S_4 := \{ (p, j, k, l, m, \alpha) \mid  j=k=l < m,  0 \leq \alpha < n/p\}\\
S_5 = \{(p, j, k, l, m, \alpha) \mid j =k=l=m, 0 \leq \alpha < n/(2p)\}
\end{gather*}
uniquely parametrizes $n$-square surfaces in $\calH(1,1)$ with cylinder diagram A.
\end{proposition}

\begin{proposition}[Uniqueness of cylinder diagram B parameters]\label{prop:typeIIparam}
The set
\begin{gather*}
\Sigma_B:=\{(p, q, k, l, m, \alpha, \beta) \in \NN^7| 0 \leq \alpha < k+l+m, 0 \leq \beta < m, p(k+l+m) + qm = n  \}
\end{gather*}
uniquely parametrizes $n$-square surfaces in $\calH(1,1)$ with cylinder diagram B.
\end{proposition}

\begin{proposition}[Uniqueness of cylinder diagram C parameters]\label{prop:typeIIIparam}
Let $(p,q,k,l,m, \alpha, \beta) \in \NN^7$ such that $p(k+l) + q(l+m) = n$, $0 \leq \alpha < k+l$ and $0 \leq \beta < l+m$. The set of such tuples $\Sigma_C = S_1 \cup S_2 \cup S_3$ where 
\begin{gather*}
S_1 := \{(p, q, k, l, m, \alpha, \beta)\mid k < m \} \hspace{1cm} S_2 := \{(p, q, k, l, m, \alpha, \beta) \mid k = m, p < q\} \\ 
S_3 := \{\{(p, q, k, l, m, \alpha, \beta)\mid k = m, p = q, \alpha \leq \beta \}
\end{gather*}
uniquely parametrizes $n$-square surfaces in $\calH(1,1)$  with cylinder diagram C.
\end{proposition}

\begin{proposition}[Uniqueness of cylinder diagram D parameters]\label{prop:typeIVparam}
Let $(p,q,r,k,l,\alpha, \beta, \gamma) \in \NN^8$ such that $(p+q)k + (r+q)l = n$, $0 \leq \alpha < k, 0 \leq \beta < l+k$ and $ 0 \leq \gamma < l$. The set of such tuples $\Sigma_D = S_1 \cup S_2 \cup S_3$ where
\begin{gather*}
S_1 := \{(p, q, r, k, l, \alpha, \beta, \gamma) \mid  k < l  \}  \hspace{1cm} S_2 := \{(p, q, r, k, l, \alpha, \beta, \gamma) \in \NN^8| k = l, p \leq r  \}\\
S_3 := \{(p, q, r, k, l, \alpha, \beta, \gamma) \in \NN^8| k = l, p = r, \alpha \leq \gamma  \}
\end{gather*}
uniquely parametrizes $n$-square surfaces in $\calH(1,1)$ with cylinder diagram D. 
\end{proposition}

The non-uniqueness of parameters stems from cut and paste equivalence. For instance, given a surface with cylinder diagram D, if the bottom cylinder is longer than the top cylinder, (i.e. $k > l$ in the parametrization), one can interchange them via a cut and paste move so that the shorter cylinder is at the bottom. In other words, parameters $(p, q, r, k, l, \alpha, \beta, \gamma)$ and $(r, q, p, l, k, \gamma, \beta, \alpha)$ define the same surface. 

Hence we need to account for such duplicate parametrization. We detail the proofs in Appendix \ref{sec:cyldiagclass}.

\begin{figure}[h!]
\centering
\begin{tabular}{cc|cc}
\includegraphics[scale=0.33]{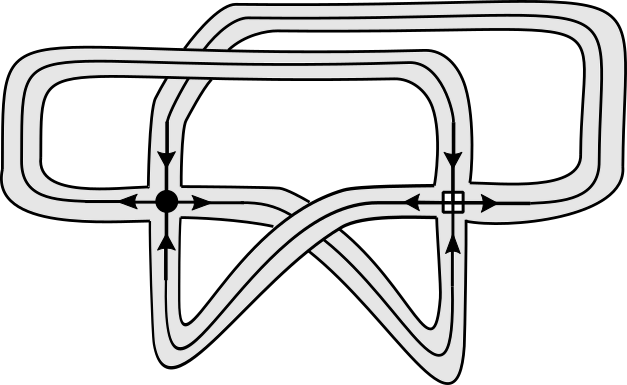}&\hspace{0.4cm}&\hspace{0.4cm}&\includegraphics[scale=0.33]{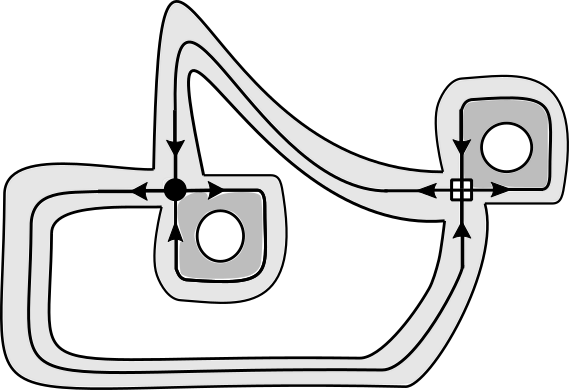}\\[10pt]

\includegraphics[scale=0.25]{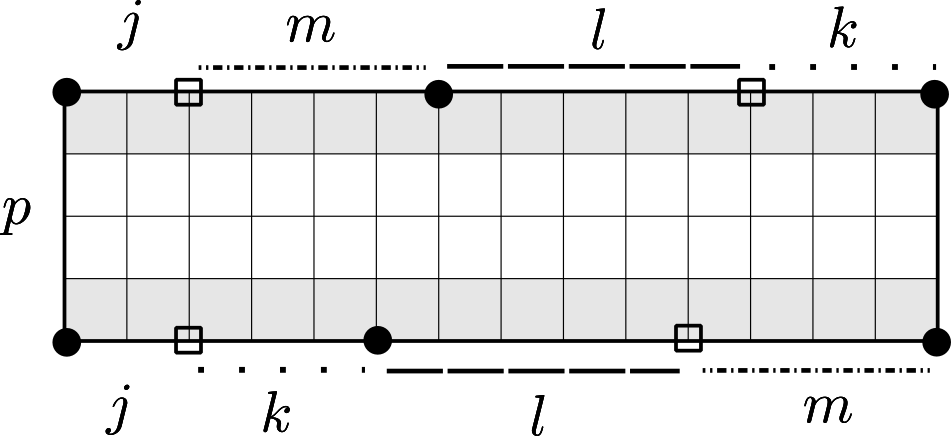}&\hspace{0.4cm}&\hspace{0.4cm}&\includegraphics[scale=0.25]{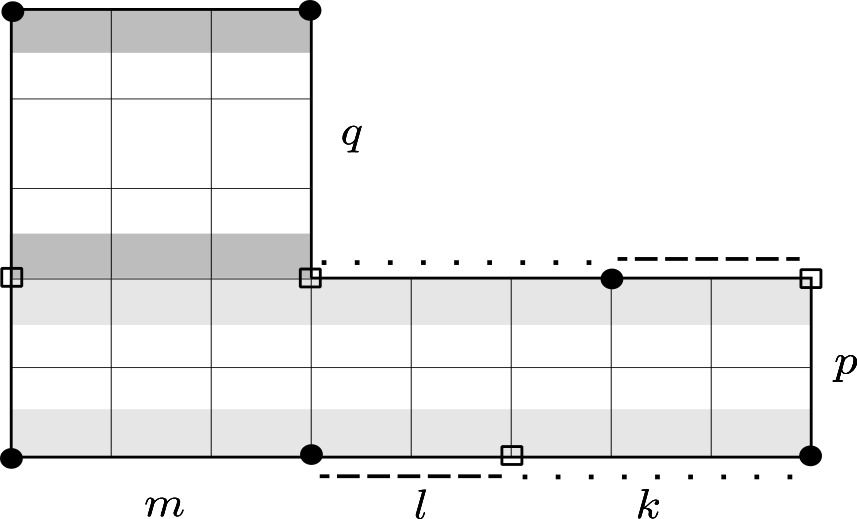}\\ [10pt]
  
   A &\hspace{0.4cm}&\hspace{0.4cm}& B \\[10pt]
   \hline
 & & &\\[10pt]
  
   \includegraphics[scale=0.33]{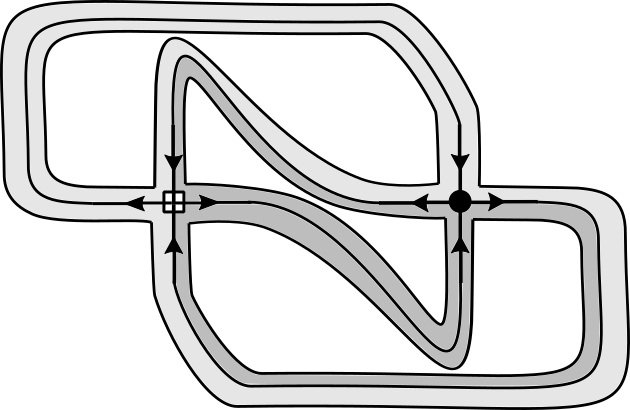}&\hspace{0.4cm}&\hspace{0.4cm}&\includegraphics[scale=0.33]{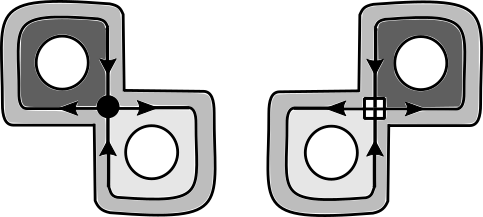}\\[10pt]
   \includegraphics[scale=0.25]{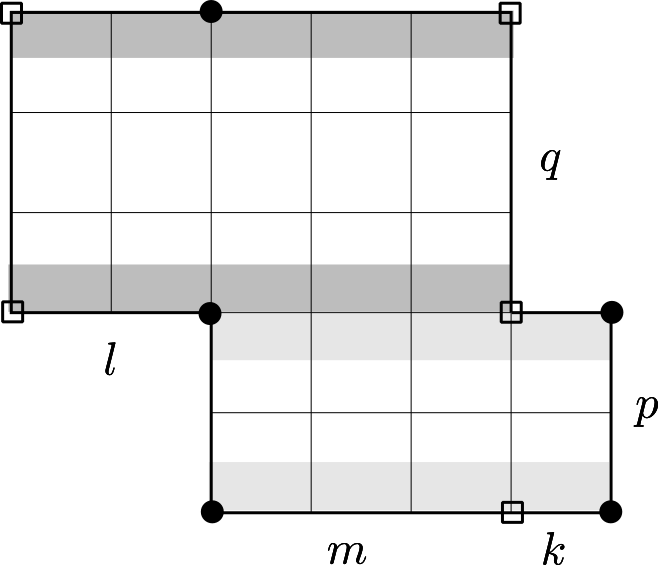}&\hspace{0.4cm}&\hspace{0.4cm}&\includegraphics[scale=0.25]{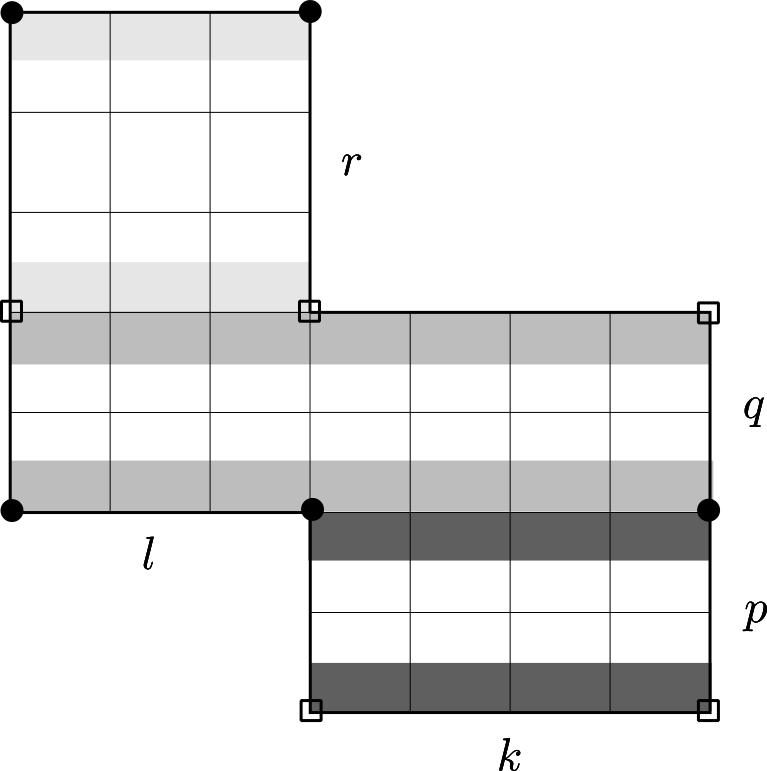} \\[10pt]
   C&\hspace{0.4cm}&\hspace{0.4cm}& D\\[10pt]
   \hline
\end{tabular}

 \caption{The 4 distinct cylinder diagrams for square-tiled surfaces of $\calH(1,1)$, and the prototypical square tiled surfaces (with parameters) that have those cylinder diagrams. Regions with the same shading have boundaries that are paired. The shears on the cylinders are omitted to avoid clutter of the pictures.}
   \label{fig:cylindertypes}
   \vspace{-0.3cm}
\end{figure}

\section{Primitivity criteria}\label{sec:primitivity}
In this section we will develop a characterization of primitivity for square-tiled surfaces in $\calH(1,1)$ according to the parameters used to describe their cylinder diagrams. 
\begin{notn} $m \wedge n$ will denote the greatest common divisor of the integers $m$ and $n$. 
\end{notn}
We start with a lemma, presented in \cite{zmprob} by Zmiaikou, that characterizes when certain integer vectors $\ZZ^2$. 
\begin{lemma}\label{lem:lattice1}
For integers $k, l, m, \alpha, \beta, p, q$, the vectors $(k, 0), (l, 0), (\alpha, p), (\beta, q)$ generate the lattice $\ZZ^2$ if and only if 
$$ p \wedge q =1 \hspace{1cm} \text{ and } \hspace{1cm} k \wedge l \wedge (p\beta - q \alpha) = 1$$
\end{lemma}

Next, we obtain a sufficient condition for primitivity of a square-tiled surface in $\calH(1,1)$. We will use this Lemma to get sufficient number theoretic conditions for primitivity in each cylinder diagram.

\begin{lemma}\label{lem:latticeprim}
A square-tiled surface $S$ in $\calH(1,1)$ with $\AbsPer(S) = \ZZ^2$ is primitive.

\end{lemma}
\begin{proof}
Assume $S$ covers another square-tiled surface $T$. Since $S$ is of genus 2, $T$ has to be of genus 1. If $(x,y) \in \ZZ^2$ is in $\AbsPer(S)$, then, $(x/m, y/m) \in \AbsPer(S)$ for some $m \in \NN$. So, $\AbsPer(S) = \ZZ^2$ implies that $\AbsPer(T) = \ZZ^2$ as well. Hence, $T$ is a genus 1 surface with $\ZZ^2$ as its lattice of absolute periods, which implies that $T$ is the standard torus $\TT$. Therefore, as $\TT$ is the only square-tiled surface covered by $S$, we conclude that $S$ is primitive.
\end{proof}

The next Lemma states necessary and sufficient conditions for primitivity in each of the four cylinder diagrams, in terms of their parameters. 
\begin{lemma}[Primitivity criterion by cylinder diagram]\label{lem:allprimitivity} Let $S \in \calH(1,1)$ be an $n$-square surface. Then 
\begin{enumerate}
\item if $S$ has cylinder diagram A, and is parametrized by $(p, j, k, l, m, \alpha)$, it is primitive if and only if 
$$p=1\hspace{1cm} \text{ and } \hspace{1cm}(j+k) \wedge (k+l) \wedge n= 1.$$
\item if S has cylinder diagram B, and is parametrized by $(p, q, k, l, m, \alpha, \beta)$, it is primitive if and only if 
$$ p \wedge q =1 \hspace{1cm} \text{ and } \hspace{1cm} (k +l) \wedge m \wedge (p\beta - q \alpha + (p+q) l) = 1.$$
\item if S has cylinder diagram C, and is parametrized by $(p, q, k, l, m, \alpha, \beta)$, it is primitive if and only if 
$$ p \wedge q =1 \hspace{1cm} \text{ and } \hspace{1cm} (k +l) \wedge (l+ m) \wedge (p\beta - q \alpha) = 1.$$
\item if S has cylinder diagram D, and is parametrized by $(p, q, r, k, l, \alpha, \beta, \gamma)$, it is primitive if and only if 
$$ (p+q) \wedge (r + q) =1 \hspace{1cm} \text{ and } \hspace{1cm} k \wedge l \wedge ((p-r)\beta + (p+q) \gamma - (r+q) \alpha) = 1.$$
\end{enumerate}
\end{lemma}
The proof is postponed to Appendix \ref{sec:primcriterion}. The general strategy to prove necessity of the number theoretic conditions is to first assume they are not satisfied, then show that this implies the monodromy group is not primitive, and hence the surface is not primitive.

We next present an alternate parametrization of the primitive $n$-square surfaces with cylinder diagram A. Even though the previous parametrization, in terms of the lengths of the horizontal saddle connections and heights and shears of the cylinders, is in lieu with the parametrization of the other cylinder diagrams, this new parametrization lends itself better to enumeration. 
To do this new parametrization, we set up a bijection between the previous parametrization in Proposition \ref{prop:typeIparam} under the primitivity conditions imposed by Lemma \ref{lem:allprimitivity} and the new set of parameters. We detail this bijection in Appendix \ref{sec:cyldiagclass}. 
\begin{lemma}\label{prop:typeIreparam}
The set of primitive $n$-square surfaces of $\calH(1,1)$ with cylinder diagram A is parametrized uniquely by the set
$$ \Omega := \{ (x, y, z, t) \in \NN^4| 1 \leq x < y < z < t \leq n, (z - x) \wedge (t-y) \wedge n = 1\}$$
\end{lemma}

Putting together the various cylinder wise primitivity criterion obtained in Lemma \ref{lem:allprimitivity} we get a converse of Lemma \ref{lem:latticeprim}. Then noting that $\AbsPer(S) = \Per(S)$ for square-tiled surfaces in $\calH(2)$, we recover the combined statement, which is stated in \cite{eskinmasurschmoll}.

\begin{proposition}\label{prop:absper}
A genus two square-tiled surface $S$ is primitive if and only if $\AbsPer(S) = \ZZ^2$.
\end{proposition}

We note that an analogous statement for genus 3 is not true, and provide an example of a square-tiled surface with $\AbsPer(S) = \ZZ^2$ but is not primitive, in Figure \ref{fig:nonprimexample}.

\begin{figure}[!h]
\begin{tabular}{ccc}
\includegraphics[scale=0.3]{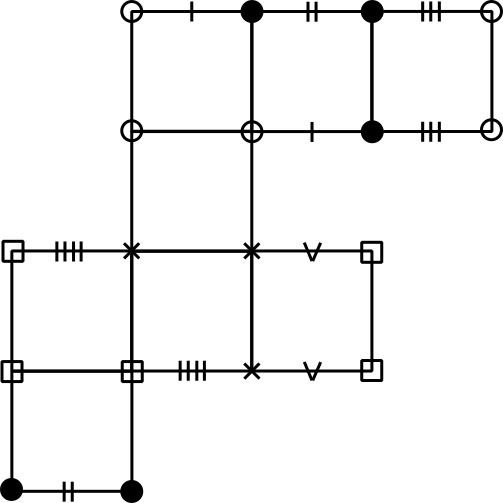} & \hspace{1.5cm} & \includegraphics[scale=0.3]{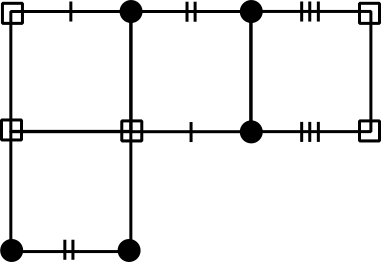}\\
(i) & \hspace{1.5cm} & (ii)
\end{tabular}
\caption{(i)A genus 3 non-primitive square-tiled surface $S$ with $\AbsPer(S) = \ZZ^2$. (ii)A genus 2 square-tiled surface that $S$ covers.}
\label{fig:nonprimexample}
\end{figure}

\section{Enumeration of Primitive Square-tiled Surfaces}\label{sec:enumeration}

Now that we have number theoretic conditions ensuring primitivity of the square-tiled surfaces with different cylinder diagrams, we can count primitive square-tiled surfaces.  
We begin with the following lemma about enumerating integers that are relatively prime to $d$ in a given interval of length $d$ with integer endpoints:
\begin{lemma}\label{lem:dprimeintegers}
Let $\beta \in \ZZ$. For any positive integer $d$,  the number of integers in the interval $[\beta, \beta+d)$ that are relatively prime to $d$ is given by $\phi(d)$. 
\end{lemma}

\begin{proof}
Given an integer in $[\beta, \beta+d)$, note that it is relatively prime to $d$ if and only if its residue class mod $d$ is a unit in $\ZZ/d\ZZ$. Since each residue class mod $d$ has one and only representative in $[\beta, \beta+d)$, the number of relatively prime integers to $d$ in $[\beta, \beta+d)$ is the number of units in $\ZZ/d\ZZ$, which is $\phi(d)$. 
\end{proof}

Next, we prove a generalization of  Lemma 8 presented in \cite{zmprob} by Zmiaikou that will be used to count the contribution of the shear parameters to our enumeration.
\begin{lemma}\label{zmiakoulem}
Let $p, q, k, l \in \NN$ such that $p \wedge q = 1$ and $\beta_1,\beta_2 \in \ZZ$. The number of distinct pairs $(\alpha, \gamma) \in \ZZ^2$ such that 
$$ \beta_1 \leq \alpha < k+\beta_1, \beta_2 \leq \gamma < l + \beta_2 \text{ and } k \wedge l \wedge (p \gamma - q \alpha) = 1$$ is equal to $kl \cdot\frac{\phi(k\wedge l)}{k \wedge l}$.
\end{lemma}

\begin{proof}
Let $d := k \wedge l$. Let $(\alpha, \gamma)$ satisfy the condition that
$ d \wedge (p \gamma - q \alpha) = 1$

Assume now that you have $\alpha', \gamma' \in \ZZ$ such that $d | (\alpha' - \alpha)$ and $d| (\gamma' - \gamma)$. Then,
$$ d \wedge (p \gamma' - q \alpha') = d \wedge (p (\gamma'-\gamma) - q (\alpha' -\alpha) + p \gamma - q \alpha) = d \wedge (p \gamma - q \alpha) = 1$$
So, to count the number of solutions $(\alpha, \gamma)$ in the box $B:= [ \beta_1, k+\beta_1) \times [\beta_2, l+\beta_2)$, it suffices to count the number of solutions in the size $d \times d$ smaller box $K:=[\beta_1, \beta_1+d) \times [\beta_2, \beta_2+d)$ then multiply the number of such solutions by $\frac{kl}{d^2}$, the number of copies of $K$ required to tile $B$, to obtain the total number of solutions in $B$. 

Now, we find the number of distinct pairs $(\alpha, \gamma) \in K$ such that $d \wedge (p \gamma - q \alpha) =1$. 

Since $p \wedge q = 1$, there exists $a, b \in \ZZ$ such that $A := \begin{bmatrix}-q & p \\ a & b\end{bmatrix} \in SL_2(\ZZ)$. Call a point $(x,y) \in \ZZ^2$ $d$-prime if $x \wedge d = 1$. Then, counting $(\alpha, \gamma) \in K$ such that $d \wedge (p \gamma - q\alpha) = 1$ is equivalent to counting $d$-prime points in $A(K)$, since
$$ A \cdot \begin{bmatrix} \alpha \\ \gamma  \end{bmatrix} =  \begin{bmatrix}-q & p \\ a & b\end{bmatrix}  \begin{bmatrix} \alpha \\\gamma  \end{bmatrix} =  \begin{bmatrix} p \gamma - q\alpha \\ a \alpha + b \gamma \end{bmatrix} \in A(K)$$

By Lemma \ref{lem:dprimeintegers}, we know there are $\phi(d)$ integers in $[\beta_1, \beta_1+d)$ that are relatively prime to $d$. Hence, there are $d\phi(d)$ $d$-prime points in $K$. 

Next, we show that the number of $d$-prime pairs in $K$ and $A(K)$ is the same.

Let $T_1 = \begin{bmatrix} 1 & 1 \\ 0 & 1 \end{bmatrix}$ and $T_2  \begin{bmatrix} 1 & 0 \\ 1 & 1 \end{bmatrix}$. 
We will argue that $T_1(K)$ and $T_2(K)$ both have $d\phi(d)$ $d$-prime pairs. 

$T_1(K)$ is the half open parallelogram with vertices $(\beta_1+\beta_2, \beta_2)$, $(\beta_1+\beta_2 + d, \beta_2)$, $(\beta_1 + \beta_2 + 2d, \beta_2+d)$, $(\beta_1+\beta_2 + d, \beta_2+d)$ as shown in Figure \ref{zmiakoulempic}.

\begin{figure}[!ht!]
\centering
\includegraphics[scale=0.07]{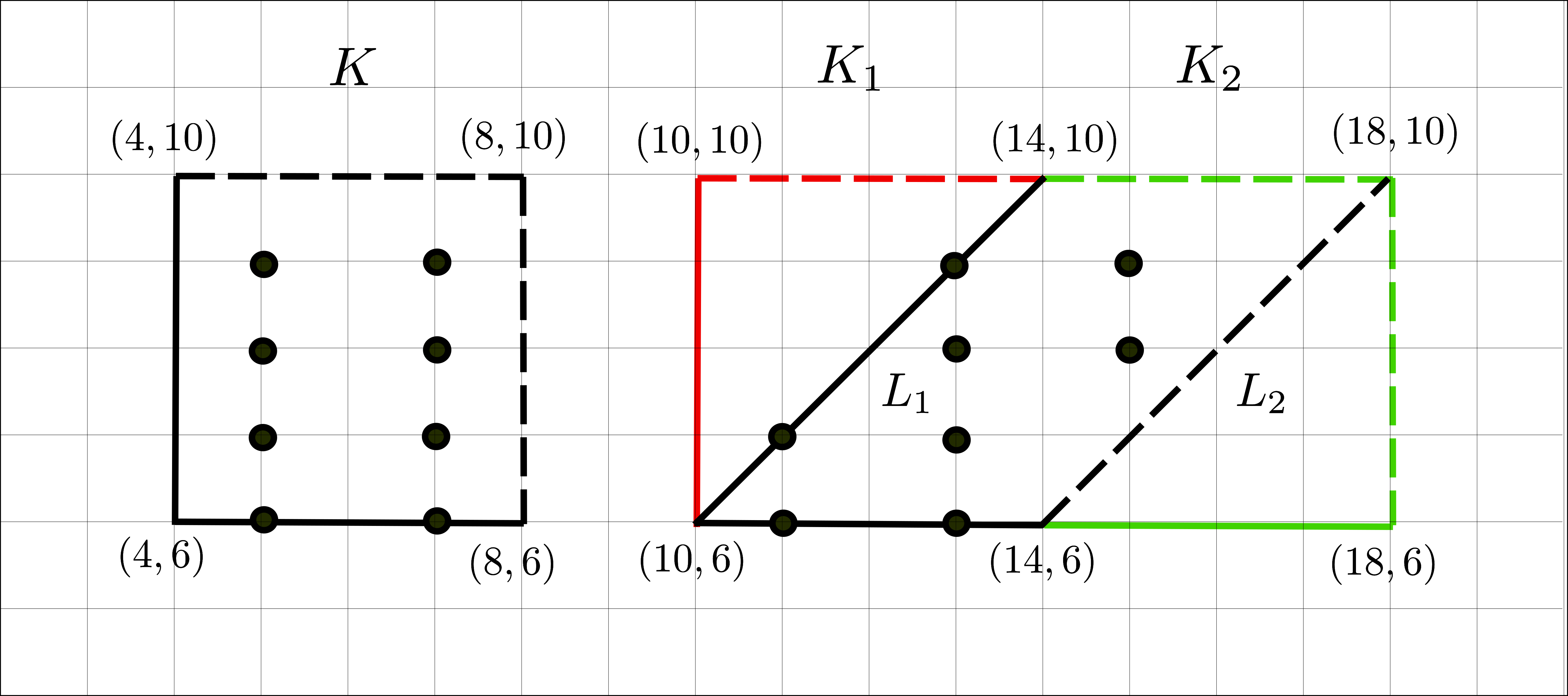}
\caption{Example case when $K$ is a $4 \times 4$ box based at $(\beta_1, \beta_2) = (4,6)$ and $d = 4$.}
\label{zmiakoulempic}
\end{figure}

First consider $K_1 := [ \beta_1+\beta_2, \beta_1+\beta_2 +d) \times [\beta_2, \beta_2+d)$. By Lemma \ref{lem:dprimeintegers}, we know that $[\beta_1+\beta_2, \beta_1+\beta_2 +d)$ has $\phi(d)$ $d$-prime integers so that $K_1$ has $d \phi(d)$ $d$-prime points. Similarly, $K_2 := [\beta_1+\beta_2+d, \beta_1+\beta_2+2d) \times [\beta_2, \beta_2+d)$ also has $d\phi(d)$ $d$-prime points. In fact, we note that the $d$-prime points of $K_2$ are exactly the $d$-prime points of $K_1$ translated by a horizontal $d$ displacement. 

Now, let $L_1$ be the half open triangle bounded by points $(\beta_1+\beta_2, \beta_2)$, $(\beta_1+\beta_2 + d, \beta_2)$, $(\beta_1+\beta_2 +d, \beta_2+d)$ where the vertical side is not included. Similarly, let $L_2$, be the half open triangle bounded by points $(\beta_1+\beta_2+d, \beta_2)$, $(\beta_1+\beta_2 + 2d, \beta_2)$, $(\beta_1+\beta_2 +2d, \beta_2+d)$ where the vertical side is not included. Note again that $L_1$ has the same number of $d$-prime points as $L_2$, since the $d$-prime points of $L_2$ can be obtained by just horizontally translating the $d$-prime points of $L_1$ by $d$.  

So, far we have,
\begin{gather*}
 \# d\text{-prime points in } K_1 = \# d\text{-prime points in } K_2\\
 \# d\text{-prime points in } L_1 = \# d\text{-prime points in } L_2
\end{gather*}

Hence, 
$$  \# d\text{-prime points in } K_1 - L_1 = \# d\text{-prime points in } K_2 - L_2$$

But note that $(K_2 - L_2)\cup  L_1 = T_1(K)$. Hence, $T_1(K)$ also has $d \phi(d)$ $d$-prime points. 

A similar geometric argument yields the fact that $T_2(K)$ has $d\phi(d)$ $d$-prime points. 

However, note that $T_1$ and $T_2$ generate $SL_2(\ZZ)$. So, in fact, $M(K)$ for any $M \in SL_2(\ZZ)$ has $d\phi(d)$ $d$-prime points and in particular, $A(K)$ has $d\phi(d)$ $d$-prime points. 
Equivalently, there exists $d\phi(d)$ distinct integer $(\alpha, \gamma)$ points in $K$ such that $$ k \wedge l \wedge (p \gamma - q \alpha) = 1$$
Hence, there are $\frac{kl}{d^2}d\phi(d) = kl \cdot\frac{\phi(k\wedge l)}{k \wedge l}$ such points in $B$. \end{proof}

\begin{notn}Since the function $\frac{\phi(m)}{m}$ appears frequently in our computations, we define $\phi'(m):= \frac{\phi(m)}{m}$
\end{notn}

\subsection{Enumeration for Cylinder Diagram A}\label{typeI}
In this section we count the number of primitive $n$-square surfaces in $\calH(1,1)$ with cylinder diagram A. 

\begin{proposition}\label{prop:typeIcount}
The number of primitive $n$-square surfaces in $\calH(1,1)$ with cylinder diagram A is given by
$$A(n) = \frac{1}{2}nJ_1(n) + \left(\frac{1}{24}n^2 - \frac{1}{4}n\right)J_2(n)$$
\end{proposition}
Before we prove this, we need the following Lemma which counts the number of positive integer quadruples under certain number theoretic conditions.

\begin{lemma}\label{dquadruple}
Let $n > 3$. For $d|n$, the number of quadruples $(x, y, z, t) \in \NN^4$ such that
$$ x < y < z < t \in [1,n] \hspace{1cm} \text{ and } d| (z-x) \text{ and } d|(t-y)$$
is given by 
$$ \binom{d}{2}\binom{n/d}{2} + 3 \binom{d}{2}\binom{n/d}{3} + d^2 \binom{n/d}{4}$$
\end{lemma}
\begin{proof}
First we write 
$$x = \alpha_1d + r_1 \hspace{0.5cm} y = \alpha_2d + r_2  \hspace{0.5cm} z = \alpha_3d + r_3  \hspace{0.5cm} t = \alpha_4d + r_4$$

Where $\alpha_i \in \{0, \dots, \frac{n}{d} -1\}$ and $r_i \in \{1, \dots, d\}$. Since $d | z-x$ and $d| t-y$, we get that $r_1 = r_3$ and $r_2 = r_4$.  Then as $x < y < z < t \in [1,n]$, there can be 5 different cases:
\begin{itemize}
\item  $\alpha_1 = \alpha_2 < \alpha_3 = \alpha_4$ and $r_1 = r_3 < r_2 = r_4$. In this case, $\binom{n/d}{2}$ choices of the $\alpha_i$ and $\binom{d}{2}$ choices for the $r_i$. Hence, we have $\binom{n/d}{2}\cdot \binom{d}{2}$ different choices for this case.
\item $\alpha_1 = \alpha_2 < \alpha_3 < \alpha_4$ and $r_1 = r_3 < r_2 = r_4$. In this case, $\binom{n/d}{3}$ choices of the $\alpha_i$ and $\binom{d}{2}$ choices for the $r_i$. Hence, we have $\binom{n/d}{3}\cdot \binom{d}{2}$ different choices for this case.
\item $\alpha_1 < \alpha_2 = \alpha_3 < \alpha_4$ and $r_1 = r_3 > r_2 = r_4$. In this case $\binom{n/d}{3}$ choice for $\alpha_i$ and $\binom{d}{2}$ choices for the $r_i$. Hence, we have $\binom{n/d}{3}\cdot \binom{d}{2}$ different choices for this case.
\item $\alpha_1 < \alpha_2 < \alpha_3 = \alpha_4$ and $r_1 = r_3 <  r_2 = r_4$. In this case, $\binom{n/d}{3}$ choices of the $\alpha_i$ and $\binom{d}{2}$ choices for the $r_i$. Hence, we have $\binom{n/d}{3}\cdot \binom{d}{2}$ different choices for this case.
\item $\alpha_1 < \alpha_2 < \alpha_3 < \alpha_4$ and $r_1 = r_3$, $r_2 = r_4$. In this case, $\binom{n/d}{4}$ choices of the $\alpha_i$ and $d^2$ choices for the $r_i$. Hence, we have $\binom{n/d}{4}\cdot d^2$ different choices for this case.
So in total we have 
$$ \binom{d}{2}\binom{n/d}{2} + 3 \binom{d}{2}\binom{n/d}{3} + d^2 \binom{n/d}{4}$$
choices for $x,y, z, t$ with the given restrictions.
\end{itemize}\end{proof}

We are now ready to give the proof of Proposition \ref{prop:typeIcount}.

\begin{proof}[Proof of Proposition \ref{prop:typeIcount}]

By Lemma $\ref{prop:typeIreparam}$, to count the number of primitive $n$-square surfaces in $\calH(1,1)$ with cylinder diagram A, it suffices to count integers,
$$x < y < z < t \in [1, n]\hspace{1cm} \text{ such that }\hspace{1cm} (z-x) \wedge (t-y) \wedge n = 1$$
To enumerate the number of ways we can pick these numbers, we start with the ${n \choose 4}$ quadruples $(x,y,z,t) \in \NN^4$ such that $x < y<z<t$. We want to delete from this number all the quadruples such that $z-x$ and $t-y$ share factors with $n$. So, we begin by subtracting the quadruples such that $z-x$ and $t-y$ are simultaneously divisible by a prime divisor of $n$. We then add the number of quadruples such that $z-x$ and $t-y$ are simultaneously divisible by two distinct primes divisors of $n$. Then we subtract the number of quadruples such that $z-x$ and $t-y$ are simultaneously divisible by three distinct prime divisors of $n$ and so on as per the inclusion exclusion principle.

\begin{align*}
&A(n)\\
&=\#\{ (x,y,z,t) \in \NN^4| x < y< z < t \in [1,n], n \wedge (z-x) \wedge (t-y) = 1\} \\
&=  {n \choose 4} -\sum_{p|n, p \text{ prime }} \# \{ (x,y,z,t) \in \NN^4| x < y< z < t \in [1,n], n \wedge (z-x) \wedge (t-y) = 1\} \\
& + \sum_{p_1|n, p_2|n,  p_1 \neq p_2 \text{ primes }}\# \{ (x,y,z,t) \in \NN^4| x < y< z < t \in [1,n], p_1, p_2 \text{ divide } (z-x) \&  (t-y) \} \\
& \vdots\\
&= \sum_{d|n} \mu(d) \# \{(x,y,z,t) \in \NN^4| x < y < z < t \in [1,n], d \text{ divides } (z-x) \& (t-y) \}
\end{align*}
Then by Lemma \ref{dquadruple} this sum is equal to 
\begin{align*}
&\sum_{d|n}\mu(d)\left({d \choose 2}{n/d \choose 2} + 3{d \choose 2}{n/d \choose 3} + d^2 {n/d \choose 4}\right)\\
=& \sum_{d|n}\mu(d)\left(\frac{d(d-1)}{2}\left( \frac{\frac{n}{d}(\frac{n}{d}-1)}{2} + 3 \frac{\frac{n}{d}(\frac{n}{d}-1)(\frac{n}{d}-2)}{6}\right)+d^2\left(\frac{\frac{n}{d}(\frac{n}{d}-1)(\frac{n}{d}-2)(\frac{n}{d}-3)}{24}\right)\right)\\
=& \sum_{d|n} \mu(d) \left(  \left(-\frac{1}{24}n^2-\frac{1}{4}n\right) + \left(\frac{1}{2}n^2\right)\frac{1}{d} + \left(\frac{1}{24}n^4 - \frac{1}{4}n^3\right)\frac{1}{d^2}\right) \\
=&  \left(-\frac{1}{24}n^2-\frac{1}{4}n\right) \sum_{d|n} \mu(d) + \left(\frac{1}{2}n^2\right) \sum_{d|n} \mu(d)\frac{1}{d} + \left(\frac{1}{24}n^4 - \frac{1}{4}n^3\right)\sum_{d|n} \mu(d)\frac{1}{d^2}
\end{align*}

Using Propositions \ref{prop:mobiussum} and \ref{prop:mobiusproduct}, the sum simplifies to,
$$\left(\frac{1}{2}n\right)n \prod_{p|n}\left(1 - p^{-1}\right) + \left(\frac{1}{24}n^2 - \frac{1}{4}n\right)n^2  \prod_{p|n}\left(1 - p^{-2}\right) = \frac{1}{2}nJ_1(n) + \left(\frac{1}{24}n^2 - \frac{1}{4}n\right)J_2(n)$$

\end{proof}

\subsection{Enumeration for Cylinder Diagram B}
In this section we count the number of primitive $n$-square surfaces in $\calH(1,1)$ with cylinder diagram B. 

\begin{proposition}\label{typeIIcount}
The number of primitive $n$-square surfaces in $\calH(1,1)$ with cylinder diagram B is given by,
$$B(n)= \bigl((\mu \cdot \sigma_2) * (\sigma_1 \Delta \sigma_2)\bigr)(n) - \left(\frac{1}{12}n^2 + \frac{5}{24}n - \frac{3}{4}\right)J_2(n) - \frac{1}{2}nJ_1(n)$$
\end{proposition}

We first count the contribution of the shear parameters to our count, in the following lemma:
\begin{lemma}\label{typeIIalphabeta}
 Let $p,q,k,l,m \in \NN$ and $p \wedge q = 1$. The number of $(\alpha, \beta) \in \ZZ^2$ such that 
$$ 0 \leq \alpha< k+l+m, 0 \leq \beta < m,  \hspace{1cm}\text{ and } \hspace{1cm} (k+l+m) \wedge m \wedge (p\beta - q\alpha + (p+q)l) =1$$
is given by,
$$ (k+l+m)m \cdot \phi'((k+l) \wedge m)$$ 
\end{lemma}

\begin{proof}

Rewrite $p\beta - q\alpha + (p+q)l = p(\beta +l) -q(\alpha -l)$. Set $\beta' = \beta + l$ and $\alpha' = \alpha - l$. Then, we want to find $(\alpha', \beta') \in \ZZ^2$ such that
$$ -l \leq \alpha'< (k+l+m)-l, l \leq \beta' < m+l,  \hspace{1cm}\text{ and } \hspace{1cm} (k+l+m) \wedge m \wedge (p\beta' - q\alpha') =1$$

Applying now, Lemma \ref{zmiakoulem} with $\beta_1 = -l$ and $\beta_2 = l$, we get that the number of required $(\alpha, \beta)$ is 
$(k+l+m)m \cdot \phi'((k+l+m) \wedge m)$\end{proof}

We are now ready the prove Proposition \ref{typeIIcount}
\begin{proof}[Proof of Proposition \ref{typeIIcount}]
First note that we need to count the parameters stated in Proposition \ref{prop:typeIIparam} under the conditions stated in Lemma \ref{lem:allprimitivity}. Since Lemma \ref{typeIIalphabeta} gives the contribution of the shear parameters, we must evaluate the sum
$$\sum_{\substack{p, q, k, l, m \in \NN  \\ p(k+l+m) + qm = n \\ p \wedge q = 1}} (k+l +m)m \cdot \phi'((k+l+m) \wedge m) = \sum_{\substack{p, q, k, l, m \in \NN  \\ k-m > l \\ pk + qm = n \\ p \wedge q = 1}} km\cdot \phi'(k\wedge m) $$ 
after a reparametrization.
Then,
\begin{align*}
B(n) &= \sum_{\substack{p, q, k, l, m \in \NN  \\ k-m > l \\ pk + qm = n \\ p \wedge q = 1}} km\cdot \phi'(k\wedge m)
=\sum_{\substack{p, q, k, l, m \in \NN  \\ k-m = l \\ pk + qm = n \\ p \wedge q = 1}} km\cdot \phi'(k\wedge m)(k-m-1)\\
&= \sum_{\substack{p, q, l, m \in \NN  \\ k-m=l\\pk + qm = n \\ p \wedge q = 1}} km(k-m) \cdot \phi'(k\wedge m) - \sum_{\substack{p, q, k, l, m \in \NN  \\ k-m = l \\ pk + qm = n \\ p \wedge q = 1}} km\cdot \phi'(k\wedge m)\\
&=  \sum_{\substack{p, q, l, m \in \NN  \\ q > p \\pl + qm = n \\ p \wedge q = 1}} (l+m)ml \cdot \phi'(l\wedge m) - \sum_{\substack{p, q, k, m \in \NN  \\ k > m \\ pk + qm = n \\ p \wedge q = 1}} km\cdot \phi'(k\wedge m)
\end{align*}
Define
$$B_1(n) :=  \sum_{\substack{p, q, l, m \in \NN  \\ q > p \\pl + qm = n \\ p \wedge q = 1}} (l+m)ml \cdot \phi'(l\wedge m)  \hspace{1cm}\text{ and }\hspace{1cm}B_2(n):= \sum_{\substack{p, q, k, m \in \NN  \\ k > m \\ pk + qm = n \\ p \wedge q = 1}} km\cdot \phi'(k\wedge m)$$

$B_1$ is simplified in Lemma \ref{lem:X} and $B_2$ is simplified in Lemma \ref{lem:Y}.

Finally, putting them together, we get, 
\begin{align*}
(B_1 - B_2)(n) &= \bigl(( \mu \cdot \sigma_2 )*( \sigma_1 \Delta \sigma_2)\bigr)(n) - \frac{1}{12}n^2 J_2(n) - \frac{5}{24} nJ_2(n) - \frac{1}{2} n J_1(n) +\frac{3}{4}J_2(n)\\
&= \bigl((\mu \cdot \sigma_2) *( \sigma_1 \Delta \sigma_2)\bigr)(n) - \left(\frac{1}{12}n^2 + \frac{5}{24}n - \frac{3}{4}\right)J_2(n) - \frac{1}{2}nJ_1(n)
\end{align*}
\end{proof}

\subsection{Enumeration for Cylinder Diagram C}
In this section we count the number of primitive $n$-square surfaces in $\calH(1,1)$ with cylinder diagram C.

\begin{proposition}\label{prop:typeIIIcount}
 The number of primitive $n$-square surfaces in $\calH(1,1)$ with cylinder diagram C is given by,
$$ C(n)= \frac{1}{24}(n-2)(n-3) J_2(n) $$

\end{proposition}

We first count the contribution of the shear parameters to our count, in the following lemma:
\begin{lemma} \label{typeIIIalphabeta}
Let $p,q,k,l,m \in \NN$ and $p \wedge q = 1$. The number of $(\alpha, \beta) \in \ZZ^2$ such that 
$$ 0 \leq \alpha< k+l, 0 \leq \beta < m+l,  \hspace{1cm}\text{ and } \hspace{1cm} (k+l) \wedge (l+m) \wedge (p\beta - q\alpha) =1$$
is given by,
$$ (k+l)(l+m) \phi'((k+l) \wedge (l+m))$$
\end{lemma}

\begin{proof}
Apply Lemma \ref{zmiakoulem} with $k$ and $l$ as $k+l$ and $l+m$ and $\beta_i = 0$. \end{proof}

We are now ready to prove Proposition  \ref{prop:typeIIIcount}
\begin{proof}[Proof of Proposition \ref{prop:typeIIIcount}]
We first note that we need to count the parameters stated in Proposition \ref{prop:typeIIIparam} under the conditions stated in  Lemma \ref{lem:allprimitivity}. Hence, using Lemma \ref{typeIIIalphabeta}, which counts the contribution of the shear parameters,  to count the number of primitive $n$-square surfaces in $\calH(1,1)$ with cylinder diagram C, we must evaluate the sum:
$$\sum_{\substack{p, q, k, l, m \in \NN  \\ k < m \\ p(k+l) + q(l+m) = n \\ p \wedge q = 1}} (k+l )(l+m) \phi'((k+l) \wedge (l+m)) + \sum_{\substack{p, q, k, l \in \NN   \\ p < q \\ (p+q)(k+l)  = n \\ p \wedge q = 1}} (k+l )\phi(k+l) +   \frac{1}{2}\sum_{\substack{ k, l \in \NN   \\ 2(k+l)  = n }} (k+l)\phi(k+l) $$

Let $C_1(n), C_2(n), C_3(n)$ be the first, second and third summation terms in the above expression. 
We start with $C_1(n)$.
Using symmetry between $k$ and $m$,  we have that 
\begin{align*}
C_1(n) &=  \sum_{\substack{p, q, k, l, m \in \NN  \\ k < m \\ p(k+l) + q(l+m) = n \\ p \wedge q = 1}} (k+l )(l+m) \phi'((k+l) \wedge (l+m))\\&= \frac{1}{2}  \sum_{\substack{p, q, k, l, m \in \NN  \\ p(k+l) + q(l+m) = n \\ p \wedge q = 1 }} (k+l )(l+m) \phi'((k+l) \wedge (l+m)) -  \frac{1}{2} \sum_{\substack{p, q, k, l\in \NN  \\ (p+q)(l+k) = n \\ p \wedge q = 1}} (k+l ) \phi(k+l) 
\end{align*}

Next, due to symmetry between $p$ and $q$, 
\begin{align*}
C_2(n) + C_3(n) =  \sum_{\substack{p, q, k, l \in \NN   \\ p < q \\ (p+q)(k+l)  = n \\ p \wedge q = 1}} (k+l )\phi(k+l) +  \frac{1}{2}\sum_{\substack{ k, l \in \NN   \\ 2(k+l)  = n }} (k+l)\phi(k+l) = \frac{1}{2}  \sum_{\substack{p, q, k, l\in \NN  \\ (p+q)(l+k) = n \\ p \wedge q = 1}} (k+l ) \phi(k+l)
\end{align*}

Hence, 
\begin{align*}
C(n) = C_1(n)+C_2(n) + C_3(n) = \frac{1}{2}  \sum_{\substack{p, q, k, l, m \in \NN  \\ p(k+l) + q(l+m) = n \\ p \wedge q = 1 }} (k+l )(l+m) \phi'((k+l) \wedge (l+m))
\end{align*}
Now, we write $C \equiv \frac{1}{2} (\mu * \tilde{C})$ where 
$$\tilde{C}(n) =  \sum_{\substack{p, q, k, l, m \in \NN  \\ p(k+l) + q(l+m) = n }} (k+l )(l+m) \phi'((k+l) \wedge (l+m))$$
Next, we simplify $\tilde{C}(n)$ as, 
\begin{align*}
 \sum_{\substack{p, q, k, l, m \in \NN  \\ k > l, m > l\\ pk + qm = n }} km\cdot\phi'(k\wedge m) =  \sum_{\substack{p, q, k, l \in \NN  \\ k > l\\ pk + ql = n }} kl(l-1)\cdot\phi'(k\wedge l) + \sum_{\substack{p, q, k, l \in \NN  \\ k > l\\ pl + qk = n }} kl(l-1)\cdot\phi'(k\wedge l) + \sum_{\substack{p, q, l \in \NN  \\  (p + q)l = n }} l(l-1)\phi(l) 
\end{align*}
Due to symmetry between $p$ and $q$ the first two summation terms are equal so we continue as,
\begin{align*}
\tilde{C}(n) & = 2 \sum_{\substack{p, q, k, l \in \NN  \\ k > l\\ pk + ql = n }} kl(l-1)\cdot\phi'(k\wedge l) +  \sum_{\substack{p, q, l \in \NN  \\  (p + q)l = n }} l(l-1)\phi(l)\\
&= 2 \sum_{\substack{p, q, k, l \in \NN  \\ k > l\\ pk + ql = n }} kl^2\cdot\phi'(k\wedge l)  - 2 \sum_{\substack{p, q, k, l \in \NN  \\ k > l\\ pk + ql = n }} kl\cdot\phi'(k\wedge l) + \sum_{\substack{p, q, l \in \NN  \\  (p + q)l = n }} l^2\phi(l) - \sum_{\substack{p, q, l \in \NN  \\  (p + q)l = n }} l\phi(l)
\end{align*}
But, using symmetry between $k$ and $l$,
$$  \sum_{\substack{p, q, k, l \in \NN  \\ k > l\\ pk + ql = n }} kl\cdot\phi'(k\wedge l)  = \frac{1}{2} \sum_{\substack{p, q, k, l \in \NN  \\ pk + ql = n }} kl\cdot\phi'(k\wedge l)  - \frac{1}{2} \sum_{\substack{p, q, l \in \NN  \\  (p + q)l = n }} l\phi(l)$$
So,
$$\tilde{C}(n)= 2 \sum_{\substack{p, q, k, l \in \NN  \\ k > l\\ pk + ql = n }} kl^2\cdot\phi'(k\wedge l) +  \sum_{\substack{p, q, l \in \NN  \\  (p + q)l = n }} l^2\phi(l) - \sum_{\substack{p, q, k, l \in \NN  \\ pk + ql = n }} kl\cdot\phi'(k\wedge l) $$

We simplify the first summation term in Lemma \ref{lem:U} and third summation term in Lemma \ref{lem:V}.  The second summation term simplifies as:
 $$\sum_{\substack{p, q, l \in \NN  \\  (p + q)l = n }} l^2\phi(l) = \sum_{l |n} l^2 \phi(l)\left(\frac{n}{l} -1\right)= \sum_{l |n} l^2 \phi(l)\frac{n}{l}  - \sum_{l |n} l^2 \phi(l)= (\Id_2\cdot \phi * \Id_1) - \Id_2 \cdot \phi * 1)(n)$$

Putting all of the terms together and simplifying using the identities in Proposition \ref{prop:arithident},
\begin{align*}
C &\equiv \frac{1}{2}(\mu * \tilde{C})\\
&\equiv \frac{1}{2}\mu* \left(2(\Id_2 \cdot \mu) *  \left(\frac{1}{24}\sigma_4 + \frac{1}{2} \sigma_3 - \frac{1}{24}(12\Id_1+1)\sigma_2\right) + \Id_2 \cdot \phi * \Id_1 - \Id_2 \cdot \phi *1\right.\\
&-    \left.(\Id_1 \cdot \mu) * \left(\frac{5}{12} \sigma_3 + \frac{1}{12}\sigma_1 - \frac{1}{2}\Id_1 \sigma_1\right)\right)\\
&\equiv (\Id_2 \cdot \mu) *\left(\frac{1}{24} \Id_4 + \frac{1}{2} \Id_3 - \frac{1}{2}(\Id_1\cdot \sigma_2*\mu) - \frac{1}{24}\Id_2\right) + \frac{1}{2}\left(\mu * \Id_2 \cdot \phi * \Id_1 - \mu* \Id_2 \cdot \phi * 1\right)\\
& - (\Id_1 \cdot \mu)*\left(\frac{5}{24} \Id_3 + \frac{1}{24}\Id_1 - \frac{1}{4}(\Id_1 \cdot \sigma_1 * \mu)\right)\\
&\equiv \frac{1}{24} \Id_2\cdot J_2 - \frac{5}{24} \Id_1\cdot J_2 + \frac{1}{4} J_2
\end{align*}

\end{proof}
\subsection{Enumeration for Cylinder Diagram D}
In this section we count the number of primitive $n$-square surfaces in $\calH(1,1)$ with cylinder diagram D. 

\begin{proposition}\label{prop:typeIVcount}
The number of primitive $n$-square surfaces in $\calH(1,1)$ with cylinder diagram $D$ is given by,
$$ D(n)= \left(\frac{1}{6}n^2 - \frac{1}{6}n\right)J_2(n) - \bigl((\mu \cdot \sigma_2) * (\sigma_1 \Delta \sigma_2)\bigr)(n)$$

\end{proposition}
We first count the contribution of the shear parameters to our count, in the following lemma:
\begin{lemma}\label{typeIValphabeta}
Let $p, q, r, k, l \in \NN$ such that $(p+q) \wedge (r+q) = 1$. The number of $(\alpha, \beta, \gamma) \in \ZZ^3$ such that 
$$ 0 \leq \alpha< k, 0 \leq \beta < k+l, 0 \leq \gamma < l \hspace{1cm}\text{ and } \hspace{1cm} k \wedge l \wedge ((p+q)(\beta+ \gamma) -(r+ q)(\alpha + \beta)) =1$$ is given by,
$$ (k+l)kl \cdot\phi'(k\wedge l)$$
\end{lemma}
\begin{proof}
We will show that, given $p, q, r, k, l \in \NN$ such that $(p+q) \wedge (r+q) = 1$ and a fixed $0 \leq \beta < k+l$, the number of $(\alpha, \gamma)$ pairs satisfying 
\begin{gather}\label{cond1}
0 \leq \alpha < k, 0 \leq \gamma < l \hspace{1cm} \text{ and } \hspace{1cm} k \wedge l \wedge ((p+q)(\beta+ \gamma) -(r+ q)(\alpha + \beta)) =1 
\end{gather}
is given by
$kl \cdot\phi'(k\wedge l)$. Then, since there are $k+l$ possible $\beta$ values, this will imply the Lemma. 

Let $p' = p+q$ and $r' = r+q$. Then $p' \wedge r' = 1$. Now set $\alpha' := \alpha + \beta$ and $\gamma' = \gamma + \beta$. Then we want to count $(\alpha', \gamma') \in \ZZ^2$ such that
\begin{gather}\label{cond2}
\beta \leq \alpha' < k+\beta; \beta \leq \gamma' < l+\beta \hspace{1cm} \text{ and } \hspace{1cm} k \wedge l \wedge (p'\gamma' -r'\alpha') =1 
\end{gather}

Applying Lemma \ref{zmiakoulem}, the number of distinct $(\alpha', \gamma')$ that satisfy (\ref{cond2}) is given by
$kl \cdot\phi'(k\wedge l)$\end{proof}
We are now ready to prove Proposition  \ref{prop:typeIVcount}

\begin{proof}[Proof of Proposition \ref{prop:typeIVcount}]
First note that we need to count the parameters from Proposition \ref{prop:typeIVparam} under the conditions imposed by Lemma \ref{lem:allprimitivity}. Since Lemma \ref{typeIValphabeta} counts the contribution of the shear parameters, we must evaluate the sum
$$D(n) =  \sum_{\substack{p, q, r, k, l \in \NN \\ k < l, \\ (p+q)k + (r+q)l = n \\ (p+q) \wedge (q+r) = 1}} (k+l)kl \cdot\phi'(k\wedge l) + \sum_{\substack{p, q, r, k\in \NN \\ p <  r  \\ (p+r+2q)k = n \\ (p+q) \wedge (q+r) = 1}} 2k^2 \phi(k) $$
Let 
$$D_1(n) =   \sum_{\substack{p, q, r, k, l \in \NN \\ k < l, \\ (p+q)k + (r+q)l = n \\ (p+q) \wedge (q+r) = 1}} (k+l)kl \cdot\phi'(k\wedge l) \hspace{1cm} \text{and} \hspace{1cm} D_2(n) = \sum_{\substack{p, q, r, k\in \NN \\ p <  r  \\ (p+r+2q)k = n \\ (p+q) \wedge (q+r) = 1}} 2k^2 \phi(k) $$
The role of $k$ and $l$ are symmetric in $D_1(n)$ so that,
$$D_1(n) = \frac{1}{2}\sum_{\substack{p, q, r, k, l \in \NN \\ (p+q)k + (r+q)l = n \\ (p+q) \wedge (q+r) = 1}} (k+l)kl \cdot\phi'(k\wedge l) - \frac{1}{2}\sum_{\substack{p, q, r, k\in \NN  \\ (p+r+2q)k = n \\ (p+q) \wedge (q+r) = 1}} 2k^2 \phi(k) $$
Notice that 
$$D_2(n) = \frac{1}{2}\sum_{\substack{p, q, r, k\in \NN   \\ (p+r+2q)k = n \\ (p+q) \wedge (q+r) = 1}} 2k^2 \phi(k) - \frac{1}{2}\sum_{\substack{p, q, r, k\in \NN \\ p = r  \\ (p+r+2q)k = n \\ (p+q) \wedge (q+r) = 1}} 2k^2 \phi(k) = \sum_{\substack{p, q, r, k\in \NN   \\ (p+r+2q)k = n \\ (p+q) \wedge (q+r) = 1}} k^2 \phi(k)$$
Then, together,
$$D(n) = D_1(n) + D_2(n) = \frac{1}{2} \sum_{\substack{p, q, r, k, l \in \NN \\ (p+q)k + (r+q)l = n \\ (p+q) \wedge (q+r) = 1}} (k+l)kl \cdot\phi'(k\wedge l) $$
Rewriting the sum, $D(n)$ becomes,
\begin{align*}
\frac{1}{2}\sum_{\substack{p, q, r, k, l \in \NN \\ p>q, r> q\\ pk + rl = n \\ p \wedge r = 1}} (k+l)kl \cdot\phi'(k\wedge l)&= \frac{1}{2}\sum_{\substack{q, r, k, l \in \NN \\ r> q\\ qk + rl = n \\ q \wedge r = 1}} (k+l)kl \cdot\phi'(k\wedge l) (q-1) + \frac{1}{2}\sum_{\substack{q, p, k, l \in \NN \\ p> q\\ pk + ql = n \\ q \wedge p = 1}} (k+l)kl \cdot\phi'(k\wedge l) (q-1)\\
 &= \frac{1}{2}\cdot2 \sum_{\substack{q, p, k, l \in \NN \\ p> q\\ pk + ql = n \\ q \wedge p = 1}} (k+l)kl \cdot\phi'(k\wedge l) (q-1)\\
  &= \sum_{\substack{q, p, k, l \in \NN \\ p> q\\ pk + ql = n \\ q \wedge p = 1}} (k+l)kl \cdot\phi'(k\wedge l) q - \sum_{\substack{q, p, k, l \in \NN \\ p> q\\ pk + ql = n \\ q \wedge p = 1}} (k+l)kl \cdot\phi'(k\wedge l)
 \end{align*}
Further define 
$$D_{11}(n) = \sum_{\substack{q, p, k, l \in \NN \\ p> q\\ pk + ql = n \\ q \wedge p = 1}} (k+l)kl \cdot\phi'(k\wedge l) q \hspace{1cm}\text{ and }\hspace{1cm} D_{12}(n) = \sum_{\substack{q, p, k, l \in \NN \\ p> q\\ pk + ql = n \\ q \wedge p = 1}} (k+l)kl \cdot\phi'(k\wedge l)$$ 

We note that $D_{12}$ has already been simplified in Lemma \ref{lem:X}.
We simplify $D_{11}$ in Lemma \ref{lem:W}. Then finally, 
$$ D(n) = D_{11}(n) - D_{12}(n) = \left(\frac{1}{6}n^2 - \frac{1}{6}n\right)J_2(n) - \bigl((\mu \cdot \sigma_2) * (\sigma_1 \Delta \sigma_2)\bigr)(n)$$\end{proof}

\subsection{The total primitive count in $\calH(1,1)$}

We can now put together the counts for different cylinder diagrams and get the total count of primitive $n$-square surfaces in $\calH(1,1)$.

\begin{theorem}The number of primitive $n$-square surfaces in $\calH(1,1)$ is given by,
$$E(n):=\frac{1}{6}(n-2)(n-3)J_2(n)$$
\end{theorem}
\begin{proof}
 The total count is,
\begin{align*}
&A(n) + B(n) + C(n) + D(n)\\ &= \frac{1}{2}nJ_1(n) + \left(\frac{1}{24}n^2 - \frac{1}{4}n\right)J_2(n)  + \bigl((\mu \cdot \sigma_2) * (\sigma_1 \Delta \sigma_2)\bigr)(n) - \left(\frac{1}{12}n^2 + \frac{5}{24}n - \frac{3}{4}\right)J_2(n) - \frac{1}{2}nJ_1(n) \\&+  \frac{1}{24}(n-2)(n-3) J_2(n) + \left(\frac{1}{6}n^2 - \frac{1}{6}n\right)J_2(n) -  \bigl((\mu \cdot \sigma_2) * (\sigma_1 \Delta \sigma_2)\bigr)(n)(n) \\
&= \frac{1}{6} (n-2)(n-3)J_2(n)
\end{align*}\end{proof}

\section{Proportion of the different cylinder diagrams}\label{sec:proportions}

In this section we look at the proportion of surfaces with the different cylinder diagrams as $n \rightarrow \infty$.

We will need the following result of Ingham \cite{ingham}(see also \cite{lemkeolivershresthathorne} for a simpler proof and second order terms) to get the first order asymptotic for $\sigma_2 \Delta \sigma_1$. 
\begin{theorem} \label{thm:asymconv}
For any complex numbers $x$ and $y$ with positive real parts, 
$$(\sigma_x \Delta \sigma_y)(n) = \frac{\Gamma(x+1)\Gamma(y+1)}{\Gamma(x+y+2)}\frac{\zeta(x+1)\zeta(y+1)}{\zeta(x+y+2)}\sigma_{x+y+1}(n) + O(n^{x+y+\alpha})$$
where 
$$\alpha = \begin{cases} 0 & \text{ if } \Re(x)>1, \Re(y) > 1\\
1- x+ \epsilon & \text{ if } \Re(x) \leq 1, \Re(y) \geq 1\\
1-y + \epsilon & \text{ if } \Re(x) \geq 1, \Re(y) \leq 1\\
1 - \frac{xy}{x+y -xy} + \epsilon & \text{ if } \Re(x)<1, \Re(y) < 1
\end{cases}$$
$\Gamma$ is the Gamma function, and $\zeta$ is the Riemann-zeta function. 

\end{theorem}

Then, we finish the proof of the main result with the following theorem.
\begin{theorem}The asymptotic densities of the various cylinder diagrams is given by the following limits.
\begin{enumerate}
\item $\lim_{n\rightarrow \infty} \frac{A(n)}{E(n)} = \frac{1}{4}$
\item $\lim_{n\rightarrow \infty} \frac{B(n)}{E(n)} = \frac{\zeta(2)\zeta(3)}{2\zeta(5)} - \frac{1}{2} \approx 0.453$
\item $\lim_{n\rightarrow \infty} \frac{C(n)}{E(n)} = \frac{1}{4}$
\item $\lim_{n\rightarrow \infty} \frac{D(n)}{E(n)} = 1 - \frac{\zeta(2) \zeta(3)}{2 \zeta(5)} \approx 0.047 $
\end{enumerate}
\end{theorem}
\begin{proof}
\begin{enumerate}
\item We have,
$$\lim_{n\rightarrow \infty} \frac{A(n)}{E(n)} =\lim_{n\rightarrow \infty} \frac{\frac{1}{2}n J_1(n) + (\frac{1}{24}n^2 - \frac{1}{4}n) J_2(n)}{\frac{1}{6}(n^2 - 5n + 6)J_2(n)} = \lim_{n \rightarrow \infty} \left(\frac{\frac{J_1(n)}{2nJ_2(n)}}{\frac{1}{6} - \frac{5}{6n} + \frac{1}{n^2}} +\frac{ \frac{1}{24} - \frac{1}{4n}}{\frac{1}{6} - \frac{5}{6n} + \frac{1}{n^2}}\right)$$
But 
 $$\frac{J_1(n)}{nJ_2(n)} =  \frac{n \prod_{p|n}\left(1 - p^{-1}\right)}{ n^3 \prod_{p|n}\left(1 - p^{-2}\right)} < \frac{1}{n^2} \rightarrow 0$$ as $n \rightarrow \infty$. 

Hence, 
$$\lim_{n\rightarrow \infty} \frac{A(n)}{E(n)}  = \lim_{n \rightarrow \infty} \left(\frac{\frac{J_1(n)}{2nJ_2(n)}}{\frac{1}{6} - \frac{5}{6n} + \frac{1}{n^2}} +\frac{ \frac{1}{24} - \frac{1}{4n}}{\frac{1}{6} - \frac{5}{6n} + \frac{1}{n^2}}\right) = \frac{1}{4}.$$

\item Note that, by Theorem \ref{thm:asymconv}, we have,
$$(\sigma_1 \Delta \sigma_2)(n) = \frac{\Gamma(2)\Gamma(3)}{\Gamma(5)}\frac{\zeta(2)\zeta(3)}{\zeta(5)}\sigma_{4}(n) + O(n^{3+ \epsilon})= \frac{\zeta(2)\zeta(3)}{12\zeta(5)}\sigma_4(n)+O(n^{3+ \epsilon})$$

Hence,
\begin{align*}
((\Id_2 \cdot \mu) * \mu * (\sigma_1 \Delta \sigma_2))(n) &=  \frac{\zeta(2)\zeta(3)}{12\zeta(5)}((\Id_2 \cdot \mu) * \mu * \sigma_4)(n) + ((\Id_2 \cdot \mu) * \mu * f)(n) \\&= \frac{\zeta(2)\zeta(3)}{12\zeta(5)}(\Id_2 \cdot J_2)(n) + ((\sigma_2 \cdot \mu) * f)(n) 
\end{align*}
where $f \in O(n^{3+\epsilon})$.
Since $\sigma_2 \cdot \mu \in O(n^3)$, we have that 
$$|(\sigma_2\cdot \mu * f)(n)| \leq \sum_{d |n}|(\sigma_2 \cdot\mu)(d) f(n/d)| \leq \sum_{d|n} d^3 \frac{n^{3+\epsilon}}{d^{3+\epsilon}} \leq n^{3+\epsilon}$$

Moreover, 
$$ n^4 \frac{6}{\pi^2} = n^4 \prod_{p \text{ prime}}(1-p^{-2}) < n^4  \prod_{p|n}(1-p^{-2}) = n^2J_2(n) $$
so that
$$\left| \frac{((\sigma_2 \cdot \mu) * f)(n) }{n^2J_2(n)}\right|  \leq \frac{n^{3+\epsilon}}{\frac{6}{\pi^2}n^4}\rightarrow 0 \text{ as } n \rightarrow \infty$$
Hence, 
$$\lim_{n\rightarrow \infty} \frac{(\Id_2 \cdot \mu * \mu * \sigma_1 \Delta \sigma_2)(n)}{\frac{1}{6}n^2J_2(n)} = \lim_{n\rightarrow \infty} \left(\frac{\frac{\zeta(2)\zeta(3)}{12\zeta(5)}n^2J_2(n)}{\frac{1}{6}n^2J_2(n)} + \frac{((\sigma_2 \cdot \mu) * f)(n) }{\frac{1}{6}n^2J_2(n)}\right) =\frac{\zeta(2)\zeta(3)}{2\zeta(5)}$$

And finally, 
\begin{align*} \lim_{n \rightarrow \infty} \frac{B(n)}{E(n)} &=  \lim_{n \rightarrow \infty} \frac{(\Id_2 \cdot \mu * \mu * \sigma_1 \Delta \sigma_2)(n) - (\frac{1}{12}n^2 + \frac{5}{24}n - \frac{3}{4})J_2(n) - \frac{1}{2}nJ_1(n)}{\frac{1}{6}(n^2 - 5n + 6)J_2(n)}
\\&= \lim_{n\rightarrow \infty} \frac{(\mu \cdot \sigma_2 * \sigma_1 \Delta \sigma_2)(n)}{\frac{1}{6}n^2J_2(n)} - \frac{1}{2}\\
&= \frac{\zeta(2)\zeta(3)}{2\zeta(5)} - \frac{1}{2} \approx 0.453
\end{align*}

\item We have, $$\lim_{n\rightarrow \infty}\frac{C(n)}{E(n)} = \lim_{n\rightarrow \infty}\frac{\frac{1}{24}(n^2 - 5n + 6)J_2(n)}{\frac{1}{6}(n^2 - 5n + 6)J_2(n)} = \frac{1}{4}$$

\item Since we have already calculated the asymptotic proportions of the other three diagrams, we get,
$$ \lim_{n \rightarrow \infty} \frac{D(n)}{E(n)} = \lim_{n\rightarrow \infty}\frac{E(n) - A(n) - B(n) - C(n)}{E(n)} = 1 - \frac{1}{4} - \frac{\zeta(2)\zeta(3)}{2\zeta(5)} + \frac{1}{2} - \frac{1}{4} = 1 - \frac{\zeta(2)\zeta(3)}{2\zeta(5)} \approx 0.047$$

\end{enumerate}
\end{proof}

This completes the proof of our main result.

\bibliographystyle{abbrv}
\bibliography{references}


\newpage
\appendix

\section{Arithmetic Functions}\label{sec:arithmetic}

In this section we will recall some basic definition and facts about arithmetic functions that we use throughout our calculations. Most of the content of this section can be found in \cite[Chapters 16, 17]{hardywright} and \cite[Chapter 2]{apostol}

\begin{defn}An \emph{arithmetic function} is a function $f: \NN \rightarrow \CC$. 
\end{defn}
 We will use the following different operations on arithmetic functions. For all $n \geq 1$ and arithmetic functions $f$ and $g$.
 
 \begin{enumerate}[label=(\roman*)]
 \item $(f +g)(n) = f(n) + g(n)$ is the \emph{sum} of $f$ and $g$
 \item $(f \cdot g)(n) = f(n)g(n)$ is the \emph{product} of $f$ and $g$
 \item $(f/g)(n) = f(n)/g(n)$ is the \emph{quotient} of $f$ and $g$ \label{quotient}
 \item $(f * g)(n) = \sum_{d|n} f(d)g(n/d)$ is the \emph{Dirichlet convolution} of  $f$ and $g$\label{dirichletconv}
 \item $(f \Delta g)(n) = \sum_{k=1}^{n-1}f(k) g(n-k)$ is the \emph{additive convolution} of $f$ and $g$. 
 \end{enumerate}

Note that all of the above operations, except the quotient \ref{quotient}, are commutative. The Dirichlet convolution, \ref{dirichletconv} is associative. 
We will use the convention that $f \cdot g * h$ means $(f \cdot g) *h$.
\begin{defn}An arithmetic function $f$ is \emph{multiplicative} if $f(mn) = f(m)f(n)$ for all $m, n \in \NN$ such that $m \wedge n = 1$. 
\end{defn}
It follows that $f(1) = 1$ for all multiplicative functions $f$. 
\begin{defn}An arithmetic function $f$ is \emph{completely multiplicative} if $f(mn) = f(m)f(n)$ for all $m, n \in \NN$. 
\end{defn}
Completely multiplicative functions distribute over the Dirichlet product as stated by the following proposition. 
\begin{proposition}\label{prop:distributive}
Let $f, g, h$ be arithmetic functions and let $f$ be completely multiplicative. Then, 
$$f\cdot(g * h) \equiv (f\cdot g) * (f \cdot h)$$
\end{proposition}

The following are some well-studied multiplicative functions that we use in our computations:
\begin{enumerate}[label=(\roman*)]
\item $1 (n) = 1$ is the \emph{constant} 1 function. 
\item $\varepsilon(n) = \begin{cases} 1 & \text{ if } n = 1 \\ 0 & \text{ else}\end{cases}$ is the \emph{Dirichlet convolution identitity} since $f * \varepsilon \equiv \varepsilon * f \equiv f$ for any arithmetic function $f$ such that $f(1) =1$.
 \item $\Id_k(n) = n^k$ will denote the \emph{power} function of order $k$.
 \item $\mu(n) = \begin{cases} 0 & \text{ if } n \text{ is not square free}\\ 1 & \text{ if } n=1 \\ (-1)^k & \text{ if } k \text{is the number of distinct primes that divide } n   \end{cases}$ 
 
 is the \emph{M\"obius} function
\item $\sigma_k(n) = \sum_{d|n} d^k$ is the \emph{divisor function} of order k

\item $\phi(n) = n\prod_{\substack{p|n\\p \text{ prime }}}\left(1- p^{-1}\right)$ is the \emph{Euler totient} function
\item $J_k(n) = n^k\prod_{\substack{p|n\\p \text{ prime }}}\left(1- \frac{1}{p^k}\right)$  is the \emph{Jordan totient} function of order k

 \end{enumerate}

Next we state some identities that we use in our calculations. 

\begin{proposition}\label{prop:mobiusproduct}
Let $f$ be a multiplicative arithmetic function. Then, for any $n \in \NN$, 
$$ \sum_{d|n}\mu(d) f(d) = \prod_{p|n} (1 - f(p))$$
\end{proposition}
\begin{proposition} \label{prop:mobiussum}
For $n > 1$, $\sum_{d|n} \mu(d) = 0$, or equivalently, $\mu * 1 \equiv \varepsilon$.
\end{proposition}

\begin{proposition}[M\"obius inversion formula]\label{prop:mobiusinversion} 
$f$ and $g$ are arithmetic functions satisfying $ g \equiv 1 * f$ if and only if $ f \equiv \mu * g.$
\end{proposition}

\begin{proposition}\label{prop:mobiuscor}
Let $f, g, h$ be arithmetic functions such that $f \equiv g * h$. Then, $h \equiv (\mu \cdot g) * f$. 
\end{proposition}

The above stated propositions can be used to produce various identities relating the different multiplicative functions presented. We use the following identities in our calculations.
\begin{proposition} \label{prop:arithident}
For $k \geq 1$ we have the following relation between some of the arithmetic functions.

\begin{enumerate}[label=(\roman*)]

\item $\mu * \sigma_k \equiv \Id_k$ \label{prop:id}
\item $\mu * \Id_k \equiv J_k$ \label{prop:j}
\item $\phi * \Id_1 \cdot \mu \equiv \mu$ \label{prop:mu}
\end{enumerate}

\end{proposition}

Arithmetic functions can also be related to each other via additive convolution. The following proposition is due to S. Ramanujan. The proof, along with other similar identities involving $\sigma_3$ and $\sigma_5$ can be found in \cite{ramanujan}

\begin{proposition}\label{prop:sigma1convsigma1}
$$\sum_{\substack{p,q,k,l \in \NN \\ pq + kl = n}} kl = (\sigma_1 \Delta \sigma_1)(n) = \frac{5}{12} \sigma_3(n) + \frac{1}{12}\sigma_1(n) - \frac{1}{2}n \sigma_1(n)$$
\end{proposition}

Next, we recall Dirichlet series which are central objects in the proofs of Lemmas \ref{lem:U} and \ref{lem:W}.

\begin{defn} A Dirichlet series is a formal series of the form 
$$ F(s) = \sum_{n = 1}^\infty \frac{\alpha_n}{n^s}$$
where the $\{\alpha_n\}_{i=1}^\infty$ is a sequence of complex numbers and $s$ is a complex number.  
\end{defn}

Dirichlet series do not always converge. However, when they do, their coefficients satisfy the following uniqueness property:
\begin{proposition}[Uniqueness of Dirichlet Series]\label{prop:dirichletuniq}
 Let
$$ F(s) = \sum_{n > 0 } \frac{\alpha_n}{n^s} \hspace{1cm} \text{ and } \hspace{1cm} G(s) =  \sum_{n > 0} \frac{\beta_n}{n^s}$$ 
be Dirichlet series such that $F(s)$ and $G(s)$ converge for all $s > s_0$, $s_0 \neq \infty$ and that $F(s) = G(s)$ for all $s > s_0$. Then, $\alpha_n = \beta_n$ for all $n$. 
\end{proposition}

\newpage

\section{Enumeration of Some Intermediate Sums}\label{sec:intermediatesums}
In this section we enumerate some intermediate sums that appear during the enumeration of primitive square-tiled surfaces in each cylinder diagram.

The following sums appear during the enumeration of cylinder diagram B surfaces. 

\begin{lemma}\label{lem:X} 
$$X(n):=\sum_{\substack{q, p, k, l \in \NN \\ p> q\\ pk + ql = n \\ p \wedge q = 1}} (k+l)kl \cdot\phi'(k\wedge l) =( (\mu \cdot \sigma_2) * \sigma_1 \Delta \sigma_2)(n) - \frac{1}{12}n^2 J_2(n) $$
\end{lemma}
\begin{proof}

Dropping the GCD condition, define,
$$\tilde{X}(n) = \sum_{\substack{q, p, k, l \in \NN \\ p> q\\ pk + ql = n}} (k+l)kl \cdot\phi'(k\wedge l)$$

Then, writing the sum over the divisors of $n$, we get,

$$\tilde{X}(n) 
=  \sum_{d|n} \sum_{\substack{q, p, k, l \in \NN \\ p> q\\ pk + ql = n \\ p \wedge q = d}} (k+l)kl \cdot\phi'(k\wedge l)
=  \sum_{d|n} \sum_{\substack{q, p, k, l \in \NN \\ p> q\\ pk + ql = n/d \\ p \wedge q = 1}} (k+l)kl \cdot\phi'(k\wedge l)
=  \sum_{d|n} X\left(\frac{n}{d}\right)$$

By the M\"obius inversion formula, see Propoosition \ref{prop:mobiusinversion} from the Appendix, this gives us,
$$ X \equiv \mu * \tilde{X}$$

Now applying symmetry between $p$ and $q$, and between $k$ and $l$, we get,
\begin{align*}
\tilde{X}(n) &= \sum_{\substack{q, p, k, l \in \NN \\ p> q\\ pk + ql = n}} (k+l)kl \cdot\phi'(k\wedge l)\\
&= \frac{1}{2}\sum_{\substack{q, p, k, l \in \NN \\ pk + ql = n}} (k+l)kl \cdot\phi'(k\wedge l) - \frac{1}{2} \sum_{\substack{p, k, l \in \NN \\ p(k + l) = n}} (k+l)kl \cdot\phi'(k\wedge l) \\
&= \sum_{\substack{q, p, k, l \in \NN \\ pk + ql = n}} kl^2 \cdot\phi'(k\wedge l) - \sum_{\substack{p, k, l \in \NN \\ p(k + l) = n}} kl^2 \cdot\phi'(k\wedge l) \\
\end{align*}

Next define 
$$\tilde{X}_{1}(n):= \sum_{\substack{q, p, k, l \in \NN \\ pk + ql = n}} kl^2 \cdot\phi'(k\wedge l)  \hspace{1cm} \text{ and } \hspace{1cm}  \tilde{X}_{2}(n) := \sum_{\substack{p, k, l \in \NN \\ p(k + l) = n}} kl^2 \cdot\phi'(k\wedge l)$$

Then,
$$\tilde{X}_{1}(n) = \sum_{d|n} \sum_{\substack{q, p, k, l \in \NN \\ pk + ql = n\\ k \wedge l = d}} kl^2 \cdot\phi'(k\wedge l)
= \sum_{d|n} \frac{\phi(d)}{d}d^3\sum_{\substack{q, p, k, l \in \NN \\ pk + ql = n/d\\ k \wedge l = 1}} kl^2 
= \sum_{d|n} d^2\phi(d)\sum_{\substack{q, p, k, l \in \NN \\ pk + ql = n/d\\ k \wedge l = 1}} kl^2$$

Now, define $X_{11} (n) = \sum_{\substack{p, q, k, l,  \in \NN \\ pk+ql = n \\ k \wedge l = 1}} kl^2$, then we see that $\tilde{X}_{1} \equiv \Id_2 \cdot \phi * X_{11}$

Now, considering the sum $X_{11}$ without the gcd condition on $k$ and $l$

$$\tilde{X}_{11}(n) =  \sum_{\substack{p, q, k, l \in \NN \\  pk+ql = n }} kl^2 \\
= \sum_{d|n}d^3 \sum_{\substack{p, k, l \in \NN \\ pk+ql = n/d \\ k \wedge l = 1}} kl^2\\
= \sum_{d|n} d^3 X_{11}\left(\frac{n}{d}\right)$$

so that $\tilde{X}_{11} \equiv \Id_3 * X_{11} \implies  X_{11} \equiv \Id_3 \cdot \mu * \tilde{X}_{11}$ by Proposition \ref{prop:mobiuscor}. We also have,  
$$\tilde{X}_{11}(n) = \sum_{\substack{p, q, k, l \in \NN \\  pk+ql = n }} kl^2
= \sum_{\substack{\alpha, \beta \in \NN\\ \alpha + \beta = n}} \sum_{k | \alpha} k \sum_{l | \beta} l^2
= \sum_{\alpha = 1}^{n-1} \sigma_1(\alpha) \sigma_2(n - \alpha)
= (\sigma_1 \Delta \sigma_2)(n)$$

Similarly,
$$\tilde{X}_{2}(n) = \sum_{\substack{p, k, l \in \NN \\ p(k + l) = n}} kl^2 \cdot\phi'(k\wedge l)
= \sum_{d | n} \sum_{\substack{p, k, l \in \NN \\ p (k+l) = n \\ k \wedge l = d}} kl^2 \cdot\phi'(k\wedge l)
= \sum_{d|n} d^2 \phi(d) \sum_{\substack{p, k, l \in \NN \\ p (k+l) = n/d \\ k \wedge l = 1}} kl^2$$

Define $X_{21} (n) = \sum_{\substack{p, k, l \in \NN \\ p (k+l) = n \\ k \wedge l = 1}} kl^2$,  and $ \tilde{X}_{21} =  \sum_{\substack{p, k, l \in \NN \\ p (k+l) = n }} kl^2 $ so that $\tilde{X}_{2} \equiv \Id_2 \cdot \phi * X_{21}$. 

Then, $ \tilde{X}_{21} \equiv \Id_3 * X_{21}  \implies X_{21}\equiv \Id_3 \cdot \mu *\tilde{X}_{21}$ again by Proposition \ref{prop:mobiuscor}. 

Finally, we consider 
$$\tilde{X}_{21}(n) = \sum_{\substack{p, k, l \in \NN \\ p (k+l) = n }} kl^2 
= \sum_{\alpha | n} \sum_{ l=1}^{ \alpha - 1} l^2 (\alpha - l) 
= \sum_{\alpha | n}\left(\frac{\alpha^4}{12} - \frac{\alpha^2}{12}\right)
= \left(\frac{1}{12}\sigma_4 - \frac{1}{12} \sigma_2\right)(n)$$
Hence, 
$$X \equiv \mu * \tilde{X}
\equiv \mu *( \tilde{X}_1 - \tilde{X}_2)
\equiv \mu * \Id_2 \cdot \phi *( X_{11} - X_{21})$$
Replacing $X_{11}$ and $X_{21}$ and using Proposition \ref{prop:arithident} parts \ref{prop:mu}, \ref{prop:id}, \ref{prop:j},
\begin{align*}
X &\equiv \mu * \Id_2 \cdot \phi * ( \Id_3 \cdot \mu *\tilde{X}_{11}- \Id_3 \cdot \mu * \tilde{X}_{21})\\
&\equiv \mu * \Id_2 \cdot \phi * \Id_3 \cdot \mu * (\tilde{X}_{11} -  \tilde{X}_{21})\\
&\equiv \mu * \Id_2\cdot (\phi * \Id_1 \cdot \mu) * (  \sigma_1 \Delta \sigma_2 -  \frac{1}{12}\sigma_4 + \frac{1}{12} \sigma_2)\\
&\equiv \Id_2\cdot \mu *(\mu * (\sigma_1 \Delta \sigma_2 ) - \frac{1}{12}\Id_4 + \frac{1}{12} \Id_2)\\
&\equiv \Id_2 \cdot \mu * \mu * (\sigma_1 \Delta \sigma_2) - \frac{1}{12}\Id_2\cdot (\mu * \Id_2) + \frac{1}{12} \Id_2 \cdot (\mu * 1)\\
&\equiv (\mu \cdot \sigma_2) * (\sigma_1 \Delta \sigma_2) - \frac{1}{12}\Id_2\cdot J_2 
\end{align*}\end{proof}

\begin{lemma}\label{lem:Y}
$$Y(n) := \sum_{\substack{p, q, k, m \in \NN  \\ k > m \\ pk + qm = n \\ p \wedge q = 1}} km\cdot \phi'(k\wedge m) = \frac{5}{24} n J_2(n) + \frac{1}{2} n  J_1(n) -\frac{3}{4}J_2(n)$$
\end{lemma}

\begin{proof}
First define, dropping the gcd condition on $p$ and $q$, 
$$\tilde{Y}(n) := \sum_{\substack{p, q, k, m \in \NN  \\ k > m \\ pk + qm = n }} km\cdot \phi'(k\wedge m) 
= \sum_{d|n} \sum_{\substack{p, q, k, m \in \NN  \\ k > m \\ pk + qm = n\\ p \wedge q = d }} km\cdot \phi'(k\wedge m)
=  \sum_{d|n}\sum_{\substack{p, q, k, m \in \NN  \\ k > m \\ pk + qm = \frac{n}{d}\\ p \wedge q = 1}} km\cdot \phi'(k\wedge m)
=\sum_{d|n}Y\left(\frac{n}{d}\right)$$

Hence, by the Mobius inversion formula, $Y \equiv \mu * \tilde{Y}$.

Then,
$$\tilde{Y}(n)= \sum_{\substack{p, q, k, m \in \NN  \\ k > m \\ pk + qm = n }} km\cdot \phi'(k\wedge m)= \sum_{d|n} \sum_{\substack{p, q, k, m \in \NN  \\ k > m \\ pk + qm = n \\ k \wedge m = d }} km \frac{\phi(d)}{d}
= \sum_{d|n} d\phi(d)\sum_{\substack{p, q, k, m \in \NN  \\ k > m \\ pk + qm = n/d \\ k \wedge m = 1 }} km$$

So, $\tilde{Y} \equiv \Id_1 \cdot \phi * Y_1$
where $Y_1(n) = \sum_{\substack{p, q, k, m \in \NN  \\ k > m \\ pk + qm = n \\ k \wedge m = 1 }} km $.
Next define
$$\tilde{Y}_1(n) := \sum_{\substack{p, q, k, m \in \NN  \\ k > m \\ pk + qm = n }} km = \sum_{d|n} d^2\sum_{\substack{p, q, k, m \in \NN  \\ k > m \\ pk + qm = n/d \\ k \wedge m = 1}} km $$
so that $\tilde{Y}_1 \equiv \Id_2 * Y_1 \implies Y_1 \equiv \Id_2 \cdot \mu * \tilde{Y}_1$ by Proposition \ref{prop:mobiuscor}.
Then,
\begin{align*}
\tilde{Y}_1(n) &= \sum_{\substack{p, q, k, m \in \NN  \\ k > m \\ pk + qm = n }} km = \frac{1}{2}\sum_{\substack{p, q, k, m \in \NN  \\ pk + qm = n }} km - \frac{1}{2}\sum_{\substack{p, q, k \in \NN \\ (p+q)k = n}}k^2\\
&= \frac{1}{2} \sum_{\substack{\alpha, \beta \in \NN\\ \alpha + \beta = n}}\sum_{k | \alpha} k \sum_{m | \beta} m - \frac{1}{2} \sum_{k|n}\left(\frac{n}{k} -1\right)k^2   \\
&= \frac{1}{2} \left(\frac{5}{12} \sigma_3(n) + \frac{1}{12}\sigma_1(n) - \frac{1}{2}n \sigma_1(n) \right)-\frac{1}{2}\left(n \sigma_1(n) - \sigma_2(n)\right) \hspace{1cm}\text{ using Proposition \ref{prop:sigma1convsigma1}}.\\\\
&=\frac{5}{24} \sigma_3(n) + \frac{1}{2}\sigma_2(n) + \frac{1}{24}\sigma_1(n) - \frac{3}{4}n \sigma_1(n)
\end{align*}

Hence, 
$$Y \equiv \mu * \tilde{Y}
\equiv \mu * \Id_1\cdot \phi * Y_1
\equiv \mu * \Id_1 \cdot \phi * \Id_2 \cdot \mu * \tilde{Y}_1
\equiv \mu * \Id_1 \cdot \mu * \left(\frac{5}{12} \sigma_3 + \sigma_2+ \frac{1}{24}\sigma_1 - \frac{3}{4}\Id_1 \cdot \sigma_1\right) $$
using Proposition \ref{prop:arithident} \ref{prop:mu}. Then using \ref{prop:id} and \ref{prop:j} of the same proposition, we get,
$$Y\equiv \mu * \Id_1 \cdot \mu * \left(\frac{5}{12} \sigma_3 + \sigma_2+ \frac{1}{24}\sigma_1 - \frac{3}{4}\Id_1 \cdot \sigma_1\right)
\equiv\frac{5}{24} \Id_1\cdot J_2 + \frac{1}{2} \Id_1 \cdot J_1 -\frac{3}{4}J_2 $$\end{proof}

The following two lemmas are intermediate sums needed while enumerating cylinder diagram C surfaces.

\begin{lemma}\label{lem:U}
 $$U(n):= \sum_{\substack{p, q, k, l \in \NN  \\ k > l\\ pk + ql = n }} kl^2\cdot\phi'(k\wedge l) = \left((\Id_2 \cdot \mu) *  \left(\frac{1}{24}\sigma_4 + \frac{1}{2} \sigma_3 - \frac{1}{24}(12\Id_1+1)\sigma_2\right)\right)(n)  $$
\end{lemma}
\begin{proof}

We rewrite the sum over the divisors of $n$,
$$U(n)=\sum_{\substack{p, q, k, l \in \NN  \\ k > l\\ pk + ql = n }} kl^2\cdot\phi'(k\wedge l) = \sum_{d |n} \sum_{\substack{p, q, k, l \in \NN  \\ k > l\\ pk + ql = n \\ k \wedge l = d}} kl^2\cdot\phi'(k\wedge l) 
= \sum_{d|n}d^2\phi(d)\sum_{\substack{p, q, k, l \in \NN  \\ k > l\\ pk + ql = n/d \\ k \wedge l = 1 }} kl^2$$

Let $$U_1(n) = \sum_{\substack{p, q, k, l \in \NN  \\ k > l\\ pk + ql = n \\ k \wedge l = 1 }} kl^2 \hspace{1cm}\text{ and }\hspace{1cm} \tilde{U}_1(n) =  \sum_{\substack{p, q, k, l \in \NN  \\ k > l\\ pk + ql = n }} kl^2$$ 
Note that $\tilde{U}_1 = \Id_3 * U_1 \implies  U_1 = \Id_3 \cdot \mu * \tilde{U}_1 $ by Proposition \ref{prop:mobiuscor}.

To find $\tilde{U}_1$, consider first the series,
\begin{gather} \label{series2}
\sum_{p, q, k, l \in \NN} \frac{kl^2}{(pk+ql)^s} \end{gather}
for $s$ large enough that the series converges. 
We rewrite the series by breaking it into parts where $k =l >0$, $k> l$ and $l > k$ to get, 
\begin{gather}\label{diag1}
 \sum_{p, q, k, l \in \NN} \frac{kl^2}{(pk+ql)^s} = \sum_{k, p, q \in \NN} \frac{k^3}{((p+q)k)^s} + \sum_{p, q, k, l \in \NN} \frac{(k+l)l^2}{(p(k+l)+ql)^s} + \sum_{k,l,  p, q \in \NN} \frac{k(k+l)^2}{(pk+q(k+l))^s}
\end{gather}
We also break up (\ref{series2}) in another way by when $p=q$, $p>q$ and $p < q$ to get, 
\begin{gather}\label{diag2}
\sum_{p, q, k, l \in \NN} \frac{kl^2}{(pk+ql)^s} = \sum_{k, l, p \in \NN} \frac{kl^2}{(p(k+l))^s} + \sum_{k, l, p, q \in \NN} \frac{kl^2}{((p+q)k+ql)^s} + \sum_{p, q, k, l \in \NN}\frac{kl^s}{(pk+(p+q)l)^s}
\end{gather}
Expanding,
$$  \sum_{k,l,  p, q \in \NN} \frac{k(k+l)^2}{(pk+q(k+l)^s} =  \sum_{k,l,  p, q \in \NN} \frac{k^3}{(pk+q(k+l))^s}+ 2  \sum_{k,l,  p, q \in \NN} \frac{k^2l}{(pk+q(k+l))^s} +  \sum_{k,l,  p, q \in \NN} \frac{kl^2}{(pk+q(k+l))^s}$$
and note that
$$ \sum_{k,l,  p, q \in \NN} \frac{k^2l}{(pk+q(k+l))^s} = \sum_{p, q, k, l \in \NN}\frac{kl^2}{(pk+(p+q)l)^s}$$
Hence, equating $(\ref{diag1})$ and $(\ref{diag2})$ and simplifying, we obtain,
\begin{multline}\label{isol}
\sum_{p, q, k, l \in \NN} \frac{(k+l)l^2}{(p(k+l)+ql)^s}  =  \sum_{k, l, p \in \NN} \frac{kl^2}{(p(k+l))^s} - \sum_{k, p, q \in \NN} \frac{k^3}{((p+q)k)^s} - \sum_{k,l,  p, q \in \NN} \frac{k^3}{(pk+q(k+l))^s}\\ - \sum_{k,l,  p, q \in \NN} \frac{k^2l}{(pk+q(k+l))^s}
\end{multline}
Again,
$$  \sum_{k,l,  p, q \in \NN} \frac{k^3}{(pk+q(k+l))^s} + \sum_{k,l,  p, q \in \NN} \frac{k^2l}{(pk+q(k+l))^s} = \sum_{p, q, k, l \in \NN} \frac{(k+l)l^2}{(p(k+l)+ql)^s}  $$
so that (\ref{isol}) becomes,
\begin{gather}\label{final}
 2 \sum_{p, q, k, l \in \NN} \frac{(k+l)l^2}{(p(k+l)+ql)^s}  =  \sum_{k, l, p \in \NN} \frac{kl^2}{(p(k+l))^s} - \sum_{k, p, q \in \NN} \frac{k^3}{((p+q)k)^s}
\end{gather}
We simplify the first term on the right hand side as,
 $$\sum_{k, l, p \in \NN} \frac{kl^2}{(p(k+l))^s} = \sum_{n >0} \frac{1}{n^s}\sum_{d|n} \sum_{k=1}^d k^2(d-k) = \sum_{n>0}\frac{1}{n^s} \sum_{d|n} \left(\frac{d^4}{12} - \frac{d^2}{12}\right) = \sum_{n}\frac{\sigma_4(n) - \sigma_2(n)}{12n^s}$$
 Similarly,
 $$\sum_{k, p, q \in \NN} \frac{k^3}{((p+q)k)^s} = \sum_{n >0} \frac{1}{n^s}\sum_{d|n} d^3 \left(\frac{n}{d} -1\right) = \sum_{n>0} \frac{n\sigma_2(n) - \sigma_3(n)}{n^s}$$
 So, (\ref{final}) then becomes, 
 $$\sum_{p, q, k, l \in \NN} \frac{(k+l)l^2}{(p(k+l)+ql)^s} = \sum_{n>0}\frac{1}{n^s}\sum_{\substack{p,q,k,l \in \NN\\ p(k+l) +ql = n}}(k+l)l^2  = \frac{1}{2}\sum_{n}\frac{\sigma_4(n) - \sigma_2(n)}{12n^s} - \frac{1}{2}\sum_{n>0} \frac{n\sigma_2(n) - \sigma_3(n)}{n^s} $$
 By uniqueness of Dirichlet series (Proposition \ref{prop:dirichletuniq}), we see that 
 $$\tilde{U}_1(n) = \sum_{\substack{p, q, k, l \in \NN  \\ k > l\\ pk + ql = n }} kl^2 = \sum_{\substack{p,q,k,l \in \NN\\ p(k+l) +ql = n}}(k+l)l^2 = \frac{1}{24}\sigma_4(n) + \frac{1}{2} \sigma_3(n) - \frac{1}{24}(12n+1)\sigma_2(n)$$

Hence, using Proposition \ref{prop:arithident}\ref{prop:mu},
\begin{align*}
U &\equiv \Id_2 \cdot \phi * U_1 \equiv  (\Id_2 \cdot \phi) * ( \Id_3 \cdot \mu) * \left(\frac{1}{24}\sigma_4 + \frac{1}{2} \sigma_3 - \frac{1}{24}(12\Id_1+1)\sigma_2\right)\\
&\equiv \Id_2 (\phi * \Id_1 \cdot \mu) * \left(\frac{1}{24}\sigma_4 + \frac{1}{2} \sigma_3 - \frac{1}{24}(12\Id_1+1)\sigma_2\right)\\
& \equiv (\Id_2 \cdot \mu) *  \left(\frac{1}{24}\sigma_4 + \frac{1}{2} \sigma_3 - \frac{1}{24}(12\Id_1+1)\sigma_2\right) 
\end{align*}\end{proof}

\begin{lemma}\label{lem:V}
$$ V(n):=\sum_{\substack{p, q, k, l \in \NN  \\ pk + ql = n }} kl\cdot\phi'(k\wedge l)  = \left((\Id_1 \cdot \mu) * \left(\frac{5}{12} \sigma_3 + \frac{1}{12}\sigma_1 - \frac{1}{2}\Id_1 \sigma_1\right)\right)(n)$$
\end{lemma}

\begin{proof}
$$\sum_{\substack{p, q, k, l \in \NN  \\ pk + ql = n }} kl\cdot\phi'(k\wedge l)  = \sum_{d |n} \sum_{\substack{p, q, k, l \in \NN  \\ pk + ql = n \\ k \wedge l = d}} kl\cdot\phi'(k\wedge l)  
 = \sum_{d |n} d\phi(d) \sum_{\substack{p, q, k, l \in \NN  \\ pk + ql = n/d \\ k \wedge l = 1}} kl$$
 Let 
 $$ V_1(n) =  \sum_{\substack{p, q, k, l \in \NN  \\ pk + ql = n \\ k \wedge l = 1}} kl \text{   and   }\tilde{V}_1(n) = \sum_{\substack{p, q, k, l \in \NN  \\ pk + ql = n }} kl$$
 So, $V \equiv (\Id_1 \cdot \phi) * V_1$. Also, as in Proposition \ref{lem:U}, we have $ V_1 \equiv \Id_2 \cdot \mu * \tilde{V}_1$. 
 By Proposition \ref{prop:sigma1convsigma1} we know $\tilde{V}_1(n) =\frac{5}{12} \sigma_3(n) + \frac{1}{12}\sigma_1(n) - \frac{1}{2}n \sigma_1(n)$ so that, using Proposition \ref{prop:mobiuscor}\ref{prop:mu},
\begin{align*}
V &\equiv  \Id_1 \cdot \phi * V_1 \equiv (\Id_1 \cdot \phi )*(\Id_2 \cdot \mu) * \left(\frac{5}{12} \sigma_3 + \frac{1}{12}\sigma_1 - \frac{1}{2}\Id_1\cdot \sigma_1\right)\\&\equiv \Id_1 (\phi * \Id_1 \cdot \mu) * \left(\frac{5}{12} \sigma_3 + \frac{1}{12}\sigma_1 - \frac{1}{2}\Id_1 \cdot \sigma_1\right)
 \end{align*} \end{proof}

The following Lemma evaluates an intermediate sum that appears in our enumeration of cylinder diagram D surfaces:
\begin{lemma}\label{lem:W}
$$W(n) := \sum_{\substack{q, p, k, l \in \NN \\ p> q\\ pk + ql = n \\ q \wedge p = 1}} (k+l)kl \cdot\phi'(k\wedge l) q  = \left(\frac{1}{12}n^2 - \frac{1}{6}n\right)J_2(n)$$
\end{lemma}
\begin{proof}
Using Proposition \ref{prop:mobiuscor} and a similar technique implemented in Lemma \ref{lem:X}, we get $W \equiv \Id_1 \cdot \mu * \tilde{W}$ where
$$\tilde{W}(n) = \sum_{\substack{q, p, k, l \in \NN \\ p> q\\ pk + ql = n}} (k+l)kl \cdot\phi'(k\wedge l) q $$

Modifying the sum to incorporate the $p > q$ in a different way,
$$\tilde{W}(n) =\sum_{\substack{q, p, k, l \in \NN \\ (p+q)k + ql = n}} (k+l)kl \cdot\phi'(k\wedge l) q =   \sum_{\substack{q, p, k, l \in \NN \\ pk + q(k+l) = n}} (k+l)kl \cdot\phi'(k\wedge l) q $$
Again, using Proposition \ref{prop:mobiuscor} and technique implemented in Lemma \ref{lem:X}, we can get $\tilde{W} \equiv \Id_2 \cdot \phi * W_1$ where,
$$W_1(n) = \sum_{\substack{q, p, k, l \in \NN \\ pk + q(k+l) = n/d\\ k \wedge l = 1}} (k+l)kl q$$

Once more applying similar manipulation as above and using Proposition \ref{prop:mobiuscor}, we can get, $W_1 \equiv \Id_3 \cdot \mu *\tilde{W}_1$ where 
$$\tilde{W}_1(n) = \sum_{\substack{q, p, k, l \in \NN \\ pk + q(k+l) = n}} (k+l)kl q$$
Next we will calculate $\tilde{W}_1$. 
First consider the series,
\begin{gather}\label{series}
\sum_{\substack{k >0, l>0\\ q>0, p>0}} \frac{lk(l-k)q}{(pk+ql)^s}
\end{gather} for $s$ large enough so that the series converges. We can first break the sum into three parts according to when $k = l$, $k > l$ and $l > k$ as follows:
$$\sum_{\substack{k >0, l>0\\ q>0, p>0}} \frac{lk(l-k)q}{(pk+ql)^s} = \sum_{\substack{k >0\\ q>0, p>0}} \frac{k^2q \cdot 0}{(pk+qk)^s} - \sum_{\substack{k >0, l>0\\ q>0, p>0}} \frac{l(k+l)kq}{(p(k+l)+ql)^s} + \sum_{\substack{k >0, l>0\\ q>0, p>0}} \frac{(l+k)klq}{(pk+q(k+l))^s}$$
The first term is 0 and the third term can be written as,
$$ \sum_n \sum_{\substack{k >0, l>0\\ q>0, p>0\\ pk + q(k+l) = n}} \frac{(l+k)klq}{(pk+q(k+l))^s} = \sum_n \sum_{\substack{k >0, l>0\\ q>0, p>0\\ pk + q(k+l) = n}} \frac{(l+k)klq}{n^s} =  \sum_n\frac{\tilde{W}_1(n)}{n^s}$$ 

Another way we can split the sum in (\ref{series}) is to break it up intro three parts according to when $p = q$, $p > q$ and $p < q$ to obtain:
$$\sum_{\substack{k >0, l>0\\ q>0, p>0}} \frac{lk(l-k)q}{(pk+ql)^s} = \sum_{\substack{k >0, l>0\\p>0}} \frac{lk(l-k)p}{(p(k+l))^s} + \sum_{\substack{k >0, l>0\\ q>0, p>0}} \frac{lk(l-k)q}{((pk+q(k+l))^s} + \sum_{\substack{k >0, l>0\\ q>0, p>0}} \frac{lk(l-k)(p+q)}{(p(k+l)+ql)^s}$$

Note that the first term here is 0 since,
$$ \sum_{\substack{k >0, l>0\\p>0}} \frac{lk(l-k)p}{(p(k+l))^s} = \sum_{\substack{k >0, l>0\\p>0}} \frac{l^2kp}{(p(k+l))^s} - \sum_{\substack{k >0, l>0\\p>0}} \frac{k^2lp}{(p(k+l))^s} =   \sum_{\substack{k >0, l>0\\p>0}} \frac{l^2kp}{(p(k+l))^s} - \sum_{\substack{k >0, l>0\\p>0}} \frac{l^2kp}{(p(k+l))^s}$$
where the last equality comes from using the symmetry between $k$ and $l$.

Equating the two ways of breaking the series $(\ref{series})$, and rearranging, 
\begin{align*}
 \sum_n\frac{\tilde{W}_1(n)}{n^s} &= \sum_{\substack{k >0, l>0\\ q>0, p>0}} \frac{lk(l-k)q}{((pk+q(k+l))^s} + \sum_{\substack{k >0, l>0\\ q>0, p>0}} \frac{lk(l-k)p}{(p(k+l)+ql)^s} + \sum_{\substack{k >0, l>0\\ q>0, p>0}} \frac{lk(l-k)q}{(p(k+l)+ql)^s}  \\&+ \sum_{\substack{k >0, l>0\\ q>0, p>0}} \frac{l(k+l)kq}{(p(k+l)+ql)^s}
\end{align*}

Note now that the first term can be rewritten as,
$$ \sum_{\substack{k >0, l>0\\ q>0, p>0}} \frac{lk(l-k)q}{((pk+q(k+l))^s}=  \sum_{\substack{k >0, l>0\\ q>0, p>0}} \frac{lk(k-l)p}{((ql+p(k+l))^s} = -\sum_{\substack{k >0, l>0\\ q>0, p>0}} \frac{lk(l-k)p}{((ql+p(k+l))^s} $$ simply by renaming $k \leftrightarrow l$ and $p \leftrightarrow q$ from which we see that the first term cancels with the second term to give,
$$\sum_n\frac{\tilde{W}_1(n)}{n^s}  = \sum_{\substack{k >0, l>0\\ q>0, p>0}} \frac{lk(l-k)q}{(p(k+l)+ql)^s}  + \sum_{\substack{k >0, l>0\\ q>0, p>0}} \frac{l(k+l)kq}{(p(k+l)+ql)^s} = 2 \sum_{\substack{k >0, l>0\\ q>0, p>0}} \frac{kl^2q}{(p(k+l)+ql)^s}$$
Let 
$$ T(n):= \sum_{\substack{k >0, l>0\\ q>0, p>0 \\ p(k+l) +ql = n}} kl^2q$$
Then, $\sum_n \frac{\tilde{W}_1(n)}{n^s} = \sum_n 2\frac{T(n)}{n^s}$ which them implies, using Proposition \ref{prop:dirichletuniq} that 
\begin{gather}\label{FG}\sum_{\substack{k >0, l>0\\ q>0, p>0\\ pk + q(k+l) = n}} (l+k)klq = \tilde{W}_1(n) = 2T(n) =2\sum_{\substack{k >0, l>0\\ q>0, p>0 \\ p(k+l) +ql = n}} kl^2q \end{gather}
Now, 
\begin{align*}
T(n) = \sum_{\substack{k >0, l>0\\ q>0, p>0 \\ p(k+l) +ql = n}} kl (n - p(k+l)) =  \sum_{\substack{k >0, l>0\\ q>0, p>0 \\ p(k+l) +ql = n}} kln - \sum_{\substack{k >0, l>0\\ q>0, p>0 \\ p(k+l) +ql = n}} kl(k+l)p  
\end{align*}
Noting that $\sum_{\substack{k, l, p, q \in \NN \\ p(k+l) +ql = n}} kl(k+l)p  = \sum_{\substack{k, l, p, q \in \NN\\ q(k+l) +pl = n}} kl(k+l)q = \tilde{W}_1(n)$ we see that 
$$ T(n) = \sum_{\substack{k >0, l>0\\ q>0, p>0 \\ p(k+l) +ql = n}} kln - \tilde{W}_1(n)$$
Replacing $T(n)$ in \ref{FG} we get,
$$ \tilde{W}_1(n) = \frac{2}{3} n\sum_{\substack{k >0, l>0\\ q>0, p>0 \\ p(k+l) +ql = n}} kl $$
Now,
 $$\sum_{\substack{k,l, p, q \in \NN \\ p(k+l) +ql = n}} kl  = \sum_{\substack{k,l, p, q \in \NN \\ pk+ (p+q)l = n}} kl = \sum_{\substack{k,l, p, q \in \NN\\ q>p  \\ pk +ql = n}} kl =    \frac{1}{2}\sum_{\substack{k,l, p, q \in \NN \\ pk +ql = n}} kl - \frac{1}{2}\sum_{\substack{k,l, p \in \NN\\ p(k +l) = n}} kl$$
 and
 $$  \frac{1}{2}\sum_{\substack{k,l, p \in \NN\\ p(k +l) = n}} kl = \frac{1}{2}\sum_{\alpha | n} \sum_{k = 1}^{\alpha -1} k (\alpha - k) =\frac{1}{2}\sum_{\alpha |n} \left(\frac{\alpha^3}{6} - \frac{\alpha}{6} \right)= \frac{1}{12}\sigma_3(n) -\frac{1}{12}\sigma_1(n)$$
so that, using Proposition \ref{prop:sigma1convsigma1},
 \begin{align*}
\tilde{W}_1(n)&= \frac{2}{3}n\cdot \frac{1}{2}\sum_{\substack{k,l, p, q \in \NN \\ pk +ql = n}} kl - \frac{2}{3}n \cdot \frac{1}{2}\sum_{\substack{k,l, p \in \NN\\ p(k +l) = n}} kl\\
 &= \frac{2}{3}n\left(\frac{5}{24}\sigma_3(n) + \frac{1}{24}\sigma_1(n) - \frac{1}{4}n\sigma_1(n) -  \frac{1}{12}\sigma_3(n) +\frac{1}{12}\sigma_1(n)\right)\\
&= \frac{1}{12}n \sigma_3(n) + \frac{1}{12}n \sigma_1(n) - \frac{1}{6}n^2\sigma_1(n)
\end{align*} 
Finally,
$$W \equiv \Id_1\cdot \mu * \tilde{W} \equiv \Id_1 \cdot \mu * \Id_2 \cdot \phi * W_1 \equiv \Id_1 \cdot \mu * \Id_2 \cdot \phi * \Id_3 \cdot \mu * \tilde{W}_1 $$ so that, using Proposition \ref{prop:arithident}
\begin{align*}W &\equiv \Id_1 \cdot \mu * \Id_2 \cdot \phi * \Id_3 \cdot \mu * \left(\frac{1}{12}\Id_1 \cdot \sigma_3 + \frac{1}{12}\Id_1 \cdot \sigma_1 - \frac{1}{6}\Id_2\cdot \sigma_1\right)\\
&\equiv \Id_2 \cdot \phi * \Id_3 \cdot \mu * \Id_1 \cdot \left(\mu *\frac{1}{12}\sigma_3+\mu * \frac{1}{12} \cdot \sigma_1 - \mu*  \frac{1}{6}\Id_1\cdot \sigma_1\right)\\
&\equiv \Id_2 \cdot \mu *\left( \frac{1}{12}\Id_4 + \frac{1}{12}\Id_2 - \frac{1}{6}\Id_2\cdot\sigma_1 * \Id_1 \cdot \mu \right)\\
&\equiv \frac{1}{12} \Id_2 \cdot J_2 + \frac{1}{12} \Id_2 \cdot \epsilon - \frac{1}{6}\Id_2 \cdot (\sigma_1 * \mu) * \Id_1 \cdot \mu\\
&\equiv \frac{1}{12} \Id_2 \cdot J_2  - \frac{1}{6}\Id_3* \Id_1 \cdot \mu\\
&\equiv \frac{1}{12} \Id_2 \cdot J_2  - \frac{1}{6}\Id_1 \cdot J_2
\end{align*}\end{proof}

\newpage
\section{Classification of Cylinder Diagrams}\label{sec:cyldiagclass}
In this section we classify the distinct cylinder diagrams in $\calH(1,1)$. 

\begin{proposition}\label{prop:cyldiagtype}
There are 4 cylinder diagrams in $\calH(1,1)$. 
\end{proposition}

\begin{proof}
\textcolor{white}{blank}

\medskip
\noindent
\begin{minipage}{4.25in}
Any surface in $\calH(1,1)$, has two singularities with angle $4\pi$ each so that any ribbon graph associated to it has 2 vertices, each of valence 4. There are two outgoing edges in each vertex. Label the outgoing edges for one of the vertices as $a$ and $b$, and the other outgoing edges to be $c$ and $d$. Similarly, there are two incoming edges for each vertex, and we label these $1$, $2$ and $3$, $4$ as shown in the figure on the side.
\end{minipage}\hspace{0.3cm}
\begin{minipage}{2.5in}
\begin{center}

\includegraphics[scale=0.3]{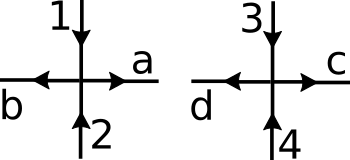}
\end{center}

\medskip
Two singular vertices and labels of the edges coming in and going out of the vertices.

\end{minipage}
\medskip

Define 
$$G = \{ \{(x_i,y_i)\}_{i=1}^4| \{x_i\}_{i=1}^4 = \{a, b, c, d\},  \{y_i\}_{i=1}^4 = \{1, 2, 3, 4\}\}$$ 
Each element of $G$ is a set of 4 tuples each of the form (letter, number), and each such element determines a graph on the two vertices. Note that $|G| = 24$.
Let $s: G \rightarrow G$ defined by sending $$a\rightarrow d, b \rightarrow c, c \rightarrow b, d \rightarrow a, 1 \rightarrow 4, 2 \rightarrow 3, 3 \rightarrow 2, 4 \rightarrow 1$$
This corresponds to a $\pi$ rotation of the graphs which is an order two action on $G$. Since rotating by $\pi$ does not change the surface obtained (as surfaces in $\calH(1,1)$ are fixed by such a rotation), we do not want to count graphs that differ by $s$.
Similarly, 
$t: G \rightarrow G$ defined by 
$$a\rightarrow c, b \rightarrow d, c \rightarrow a, d \rightarrow b, 1 \rightarrow 3, 2 \rightarrow 4, 3 \rightarrow 1, 4 \rightarrow 2$$
also is an order 2 action on $G$ which corresponds to swapping the vertices. Since we do not distinguish between the two vertices, we do not want to count graphs that differ by $t$. The group $\ideal{s,t}$ generated by these elements has order 4, and hence we see that there are at most 24/4 = 6 classes of graphs defined by elements in $G$ that potentially give different cylinder diagrams. Figure \ref{fig:6graphs} shows representatives of each class.

\begin{figure}[h!!]
   \centering
\begin{tabular}{ccccc}
\includegraphics[width=4.2cm]{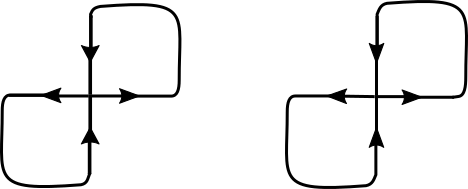}& \hspace{0.5cm} &
\includegraphics[width=4.2cm]{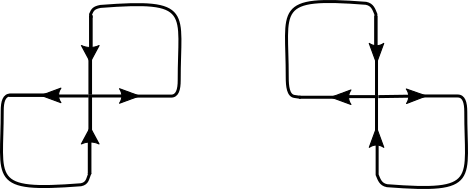}&\hspace{0.5cm} &
\includegraphics[width=4.2cm]{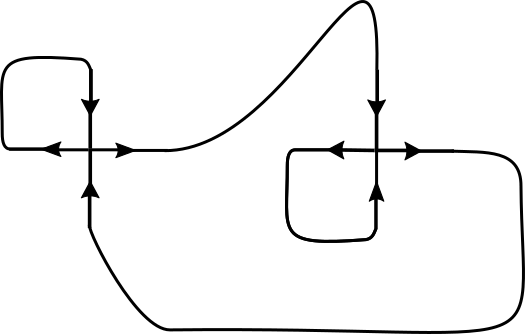}\\[10pt]
i & \hspace{0.5cm} & ii &  \hspace{0.5cm} &  iii\\ [10pt]
\includegraphics[width=4.2cm]{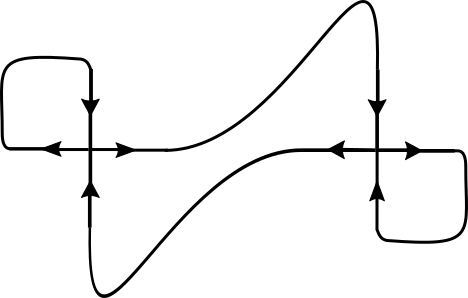}& \hspace{0.5cm} &
\includegraphics[width=4.2cm]{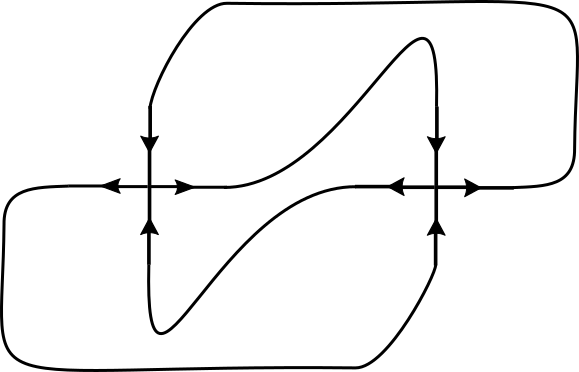}&\hspace{0.5cm} &
\includegraphics[width=4.2cm]{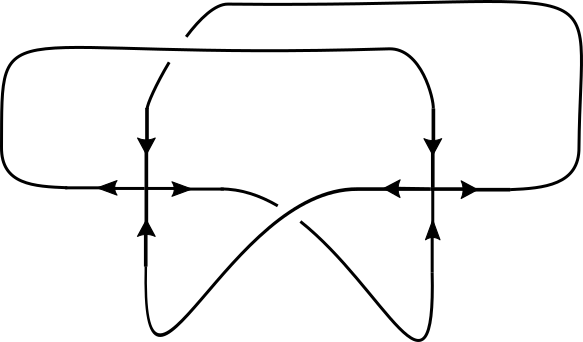}\\[10pt]
iv & \hspace{0.5cm} & v &  \hspace{0.5cm} &  vi\\[10pt] 
\end{tabular}
\caption{Directed graphs associated to the ribbon graphs of surfaces in $\calH(1,1)$}
\label{fig:6graphs}
\end{figure}

Among these, only graphs ii), iii), v) and vi) give cylinder diagrams. Consider the ribbon graph for i). After choosing an orientation on the ribbon graph, we note that the number of positively oriented boundary components is not equal to the number of negatively oriented boundary components. Hence, we cannot find a valid pairing to form a separatrix diagram. The same holds for iv). 

Finally one checks that each of these directed graphs and their corresponding ribbon graphs have exactly one pairing P each that result in a cylinder diagram. Hence, there are exactly 4 cylinder diagrams in $\calH(1,1)$. \end{proof}

We now present the proofs of the uniqueness of parameters determining $n$-square surfaces with diagrams A, B, C and D. In what follows, all shear parameters of cylinders will be assumed to be modulo the length of the corresponding cylinder. 
\begin{notn}
Given a tuple of parameters $v$ for a square-tiled surface of a certain diagram in $\calH(1,1)$, we will denote by $S(v)$ the associated square-tiled surface.
\end{notn}
\begin{proof}[Proof of Proposition \ref{prop:typeIparam}]
Given an $n$-square surface $S$ in $\calH(1,1)$ with cylinder diagram $A$ and parametrized by $(p, j, k, l, m, \alpha)$ not necessarily in $\Sigma_A$, we first note that we can find parameters in $\Sigma_A$ that realize $S$. Indeed, we can cut and paste so that the the shortest horizontal saddle connection appears on the bottom edge of the cylinder. Moreover,
\begin{itemize}
\item if $S$ has two shortest horizontal saddle connections of length $j = l$, but $k > m$. then, $S$ can be realized by $(p, l, m, j, k, \alpha + k-m) \in \Sigma_A$.
\item if $S$ has two shortest horizontal saddle connections of length $k = m$, and $j < l$. then, $S$ can be realized by $(p, m, j, k, l, \alpha + k+l) \in \Sigma_A$.
\item if $S$ has three horizontal saddle connections of equal length, for instance $k = l = m$,  then, $S$ can be realized by $(p, k, l, m, j, \alpha+j+k) \in \Sigma_A$.
\item if $S$ has all equal length horizontal saddle connections, but shear $n/(2p) \leq \alpha < n/p$, then, $S$ can be realized by $(p, j, k, l, m, \alpha - n/4) \in \Sigma_A$. 
\end{itemize}  
So, $\Sigma_A$ suffices to parametrize $n$-square surfaces with cylinder diagram A in $\calH(1,1)$. Next, we argue that distinct elements in $\Sigma_A$ parametrize distinct surfaces. Let $(p, j, k, l, m, \alpha)$ and $(p', j', k', l', m', \alpha') \in \Sigma_A$ parametrize the same surface $S$. Comparing the height of the cylinder in two parametrizations, we get $p = p'$. 
Note that in our parametrizations, $j$ and $j'$ will always be the length of the shortest horizontal saddle connections, no matter the number of such saddle connections. Hence, $j = j'$. Next, we will show $\alpha = \alpha'$. Assume without loss of generality that $\alpha' > \alpha$. Shear $S$ by $-\alpha$ and get a surface $S_0$ parametrized by $(p, j,k,l,m, 0)$ and $(p,j,k',l',m', \alpha' - \alpha)$. We note that $S_0$ has a vertical cylinder of length $p=p'$ and height $j=j'$. Call this cylinder $\calC$. 

If all horizontal saddle connections are equal in $S$, then $k = l = m = k' = l' = m'$. Then if $0 \neq \alpha - \alpha'  < n/2$, then, $S_0$ will not have a vertical cylinder of length $p=p'$ and height $j = j'$. Hence, $\alpha = \alpha'$ in this case.

If $S$ has three shortest horizontal saddle connections, then $k = l = k' = l'$. This then implies $m = m'$ since the total length of the cylinder must be equal in both parametrizations. Then for the existence of $\calC$,  it must be the case that $\alpha - \alpha'$ is either 0 or $n-(m-k)$. However, in the latter case $S_0$ admits 3 non-closed saddle connections of holonomy $(j, p)$ a contradiction since the shear 0 case only admits 2.

If $S$ has two shortest horizontal saddle connections, we first orient the horizontal saddle connections from left to right. Then, if the two shortest saddle connections do not both emanate from the same vertex, then they are parametrized by $j, k$ and $j', k'$ so that $k = k'$. Moreover, for $\calC$ to exist in $S_0$, either $\alpha - \alpha'=0$ or $\alpha' - \alpha = j+k$. In the latter case, there exists a closed saddle connection of holonomy $(j+k, p)$ which is not present in the 0 shear case, unless $m = 3k$. But in that case, the non-trivially sheared surface has a closed saddle connection of holonomy (4k, p). For the 0-shear case to have this curve, $l$ needs to be equal to $k$ a contradiction. So that $\alpha' = \alpha$. To show $l = l'$ and $m = m'$, we consider the consider 5 shortest simple closed curves in $S_0$ that have vertical holonomy $p$. In one parametrization the holonomy vectors of these simple closed curves lie in the set
$$\{ (0, k), (j+m, p), (l+m, p), (-j-k, p), (m-k, p)\}$$ and in another parametrization, the corresponding set of holonomy vectors is
$$ \{(0, p), (j+m', p), (l'+m', p), (-j-k', p), (m'-k', p)\}$$
These sets must be equal, as they parametrize the same objects. Hence, we conclude $m=m'$ and $l = l'$.

If the two shortest saddle connections both emanate from the same vertex, we know they are parametrized by $j, l$ and $j',l'$ respectively, so that $l = l'$ since $j = j'$ has been established. Also, in this case, $k \leq m$ and $k'\leq m'$, as per our parameter set. But $\{k, l, m\} = \{k', l', m'\}$ so that $k= k'$ and $m= m'$. We argue $\alpha = \alpha'$ similarly as before.

 If $S$ has a unique shortest horizontal saddle connection, for $\calC$ to exist in $S_0$, it must be that $\alpha' - \alpha = 0 \implies \alpha = \alpha'$. Then comparing the parametrizations of the holonomy vectors of the 5 shortest simple closed curves in $S_0$ as above, we can conclude that $- j - k = - j' = k' \implies k = k'$, $m - k = m' - k' \implies m = m'$ and finally $l = l'$.
\end{proof}

Next we prove the uniqueness of parameters for cylinder diagram B.

\begin{proof}[Proof of Proposition \ref{prop:typeIIparam}]
Let $(p, q, k, l, m, \alpha, \beta)$ and $(p', q', k', l', m', \alpha', \beta') \in \Sigma_B$ parametrize the same surface. Since the horizontal singular closed curve that contains only one singularity has length $m$ in one parametrization and $m'$ in another parametrization, we get that $m = m'$.
Considering the length of the horizontal core curve of the longer cylinder in both parametrizations, we get that $k'+l'+m' = k+l+m \implies k'+l' = k+l$.
Similarly, considering the heights of the cylinders in both parametrizations gives us $p = p'$ and $q = q'$ since two cylinders are of different lengths.
Assume $\beta' > \beta$. Now, define $S_0$ by shearing the shorter cylinder of $S$ by $-\beta$. The parameters for $S_0$ then are $(p, q, k, l, m, \alpha, 0)$ and $(p', q', k', l', m', \alpha', \beta' - \beta)= (p, q, k', l', m, \alpha', \beta' - \beta)$. Note that there exists exactly one vertical saddle connection in $S_0$ transverse to the core curve of the shorter cylinder. Since both the boundary components of this cylinder have only one singularity, the only way for there to be a vertical saddle connection is if the horizontal shear is zero. Hence, $\beta' - \beta = 0 \implies \beta = \beta'$.
Now, assume without loss of generality $\alpha' > \alpha$, and obtain $S_{00}$ from $S_0$ by shearing the longer cylinder by $-\alpha$.  Then, the parameters for $S_{00}$ will be $(p,q,k,l,m, 0, 0)$ and $(p,q,k',l',m,\alpha' - \alpha, 0)$. Note that $S_{00}$ has a vertical simple closed curve of length $p+q$, crossing the core curves of both the horizontal cylinders exactly once each without any cone points. From figure \ref{fig:typeIIshearequality} the shear in the longer cylinder must be zero, so that $\alpha' - \alpha = 0 \implies \alpha = \alpha'$. Now consider the length of the horizontal saddle connections we get that either i) $k = k'$ and $l =l'$ or ii) $k = l'$ and $l = k'$. We will show that the case ii) leads to a special case of i).

\begin{figure}[h!!]
\begin{tabular}{c}
\includegraphics[scale=0.18]{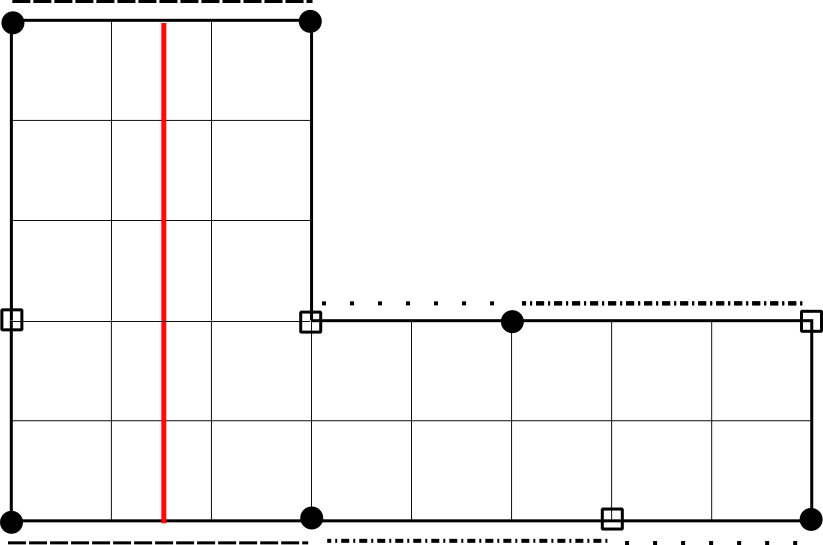} \\
 \includegraphics[scale=0.18]{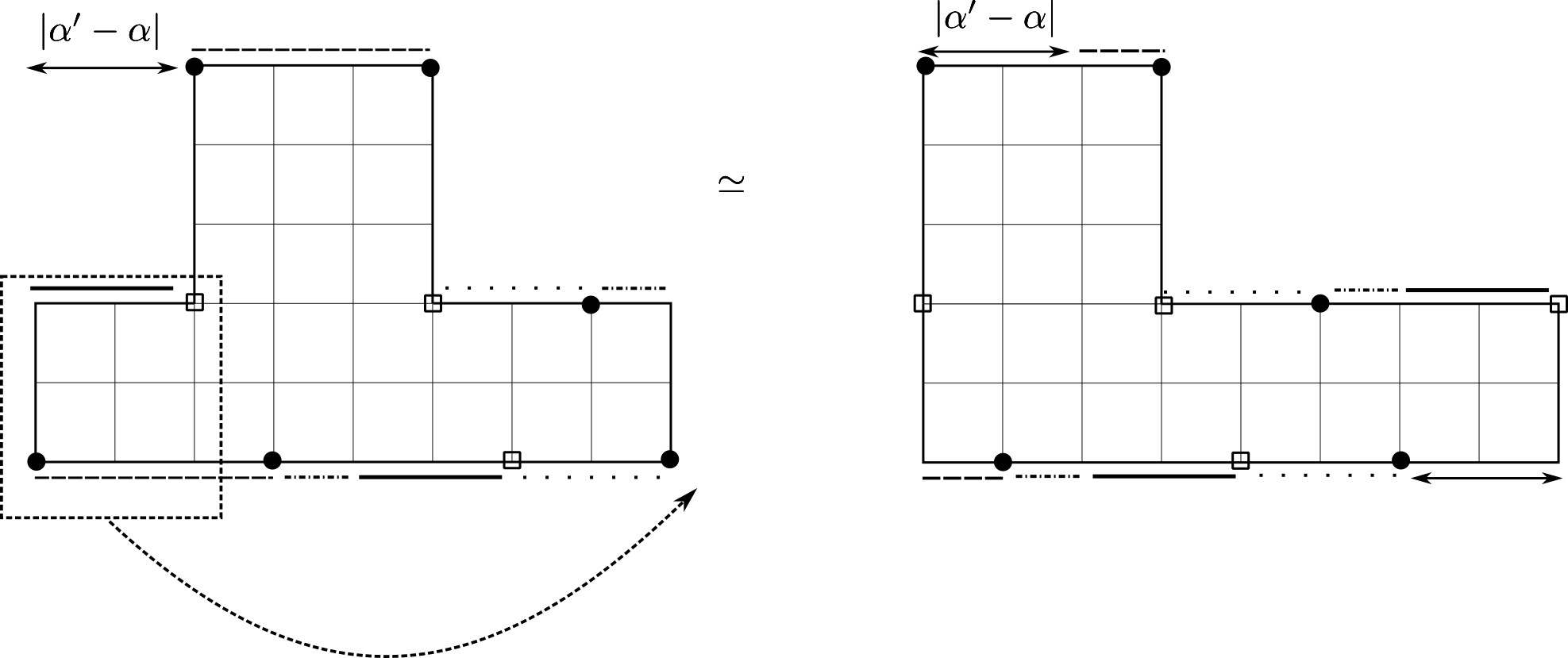}
\end{tabular}
\caption{Diagram B argument for uniqueness of shear parameters. Top: $S_{00}$ with a non-singular vertical simple closed curve of length $p+q$, crossing the core curves of both the horizontal cylinders exactly once. Bottom: $S_{00}$ with a nontrivial $\alpha - \alpha'$ shear on the longer cylinder.}
\label{fig:typeIIshearequality}
\end{figure}
\medskip
\noindent
\begin{minipage}{4.5in}

Consider the set of the 4 shortest non-closed saddle connections in $S_{00}$ with vertical holonomy $-p$, transverse to the core curve of the longer cylinder and non-zero horizontal holonomy. These saddle connections have holonomy vectors in the set 
$$\{(m, -p), (k+l, -p), (-k,-l), (l-k, -p)\} $$ 
according to one parametrization, and 
$$\{(m, -p), (k'+l', -p), (-k'-l', (l'-k', -p)\} $$
according to the other parametrization. Assume $k = l'$ and $l = k'$. Then, equality of these sets implies that 
$$l' - k' = l - k \implies k - k' = k' - k \implies k = k'.$$ See Figure on the side for these curves.
\end{minipage}\hspace{0.3cm}
\begin{minipage}{2.2in}
\begin{center}
\includegraphics[scale=0.2]{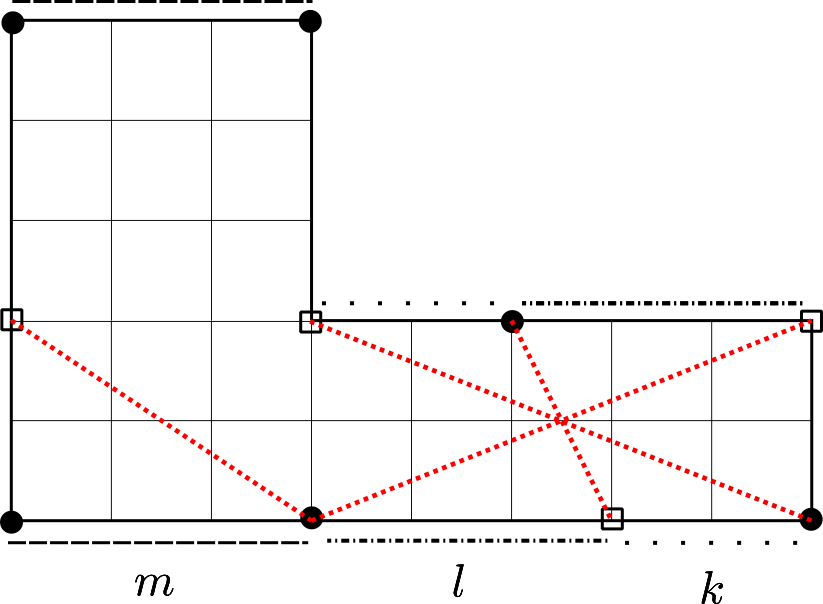}
\end{center}
\medskip
Special short curves with vertical holonomy $-p$ in $S_{00}$
\end{minipage}
\end{proof}

Next we prove uniqueness of parameters for cylinder diagram C.
\begin{proof}[Proof of Proposition \ref{prop:typeIIIparam}]
Given an $n$-square surface $S$ in $\calH(1,1)$ with cylinder diagram C and parametrized by $(p,q, k, l, m, \alpha, \beta)$ (not necessarily in $\Sigma_C$), we first note that we can find parameters in $\Sigma_C$ that realize $S$.  Indeed, if $k > m$, then the parameters $(q,p,m,l,k,\beta, \alpha) \in \Sigma_C$ give $S$ as shown in Figure \ref{fig:typeiiicutandpasteex}. 
\begin{figure}[h!]
\includegraphics[scale=0.3]{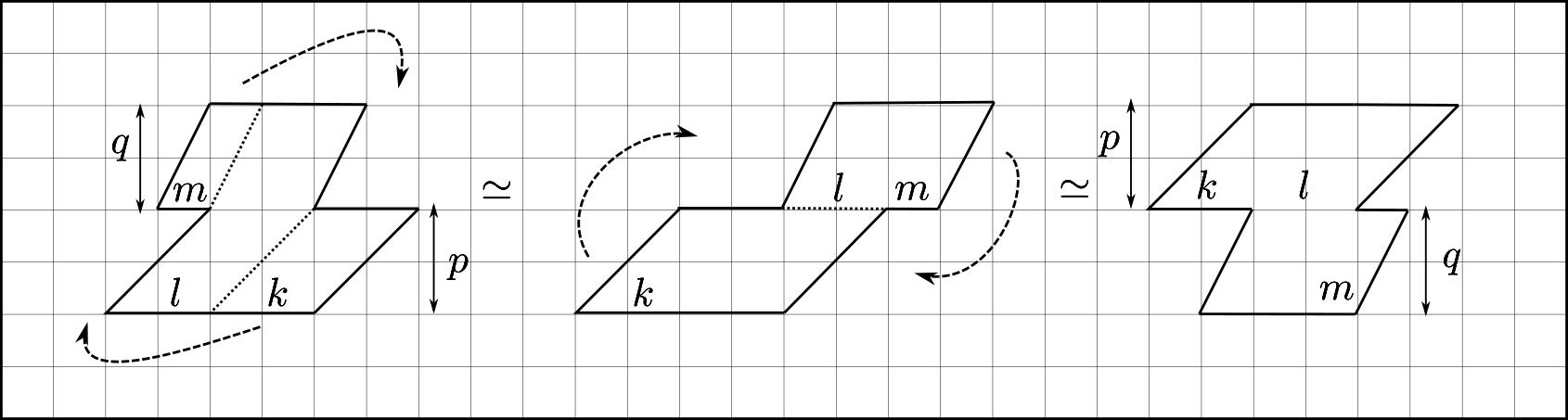}
\caption{Cut and paste moves showing $S(p,q,k,l,m, \alpha, \beta) \simeq S(q,p,m,l,k,\beta, \alpha)$ for diagram C}
\label{fig:typeiiicutandpasteex}
\end{figure}
Similarly, given $S(p,q, k, l, k, \alpha, \beta)$ with $p > q$, then we note that $S(q,p,k,l,k, \beta, \alpha)$  $\simeq$ $S(p,q, k, l, k, \alpha, \beta)$ are the same surface, and finally, given $S(p,p, k, l, k, \alpha, \beta)$ with $\alpha > \beta$ we see that $S(p,p,k,l,k,\alpha, \beta) \simeq S(p, p, k, l, k, \beta, \alpha)$.

Now we show that distinct parameters will correspond to distinct surfaces in $\Sigma_C$. Let $v:= (p, q, k, l, m, \alpha, \beta)$ and $v':=(p', q', k', l',m', \alpha', \beta') \in \Sigma_C$ parametrize the same surface $S$. We will show $v = v'$.
Noting that $l$ and $l'$ parametrizes the length of the horizontal saddle connections shared by the two cylinders, we conclude $l = l'$. 

When the length of the two cylinders is different, comparing their parametrizations gives $k +l = k'+l' \implies k = k'$ and $m+l = m'+l' \implies m= m'$. Moreover, $p$ and $p'$ both parametrize the height of the same cylinder of length $k+l$. Hence, $p = p'$. Similarly, $q = q'$. Next, assume $\alpha' \geq \alpha$, and obtain $S_0$ from $S$ by shearing the shorter cylinder by $-\alpha$. Then, the parameters for $S_0$ are $(p, q, k, l, m, 0, \beta)$ and $(p, q, k, l, m, \alpha' - \alpha, \beta')$. Observe that $S_0$ has a vertical simple closed curve passing through a singularity of length $p$, transverse to the core curve of the shorter cylinder. This only happens when $ \alpha'-\alpha =0$ since given a singularity, the shorter cylinder has exactly one singularity in each of the boundary curves. Hence $\alpha = \alpha'$. Using a similar argument, one concludes $\beta = \beta'$.

When the length of the two cylinders is equal, and heights are different, then comparing the parametrizations of the length gives $k = k' = m' = m$ while comparing the one for heights of the cylinders gives $p = p'$ and $q = q'$. Now, consider the closed saddle connections of vertical holonomy $p$, intersecting the core curve of the width $p$ cylinder exactly once. There are two such saddle connections, and the holonomy vectors are $(\alpha, p), (\alpha - (k+l), p)$ in the first set of parameters and $(\alpha', p), (\alpha' - (k+l), p)$ in the second set of parameters.The only way for these sets to be equal is to have $\alpha = \alpha'$. A similar argument will show that $\beta = \beta'$. 

When the length and height of cylinders are equal, again, we conclude $k= k'=m'=m$ and $p = p' = q' = q$. consider the shortest closed saddle connection on the surface with positive vertical holonomy. In one parametrization, this saddle connection has holonomy vectors $(\alpha, p)$ (since $\alpha \leq \beta$.). In the other parametrization we get $(\alpha', p)$ (again since $\alpha' \leq \beta'$). Hence, $\alpha = \alpha'$. Doing a similar argument in the more sheared cylinder, we get that $\beta = \beta'$.\end{proof}

Finally, we will prove uniqueness of parameters for cylinder diagram D.

\begin{proof}[Proof of Proposition \ref{prop:typeIVparam}]
Given an $n$-square surface $S$ in $\calH(1,1)$ with cylinder diagram D and parametrized by $(p,q,r, k, l, \alpha, \beta, \gamma)$ (not necessarily in $\Sigma_D$), we first note that we can find parameters in $\Sigma_D$ that realize $S$.  
Indeed, if $k > l$, then $S(p,q,r, k,l, \alpha, \beta, \gamma) \simeq S(r, q, p, l, k, \gamma, \beta, \alpha)$.
Similarly, if $p >r$, then, $S(p, q, r, k, k, \alpha, \beta, \gamma) \simeq S(r, q, p, k, k, \gamma, \beta, \alpha)$. Finally if $\alpha > \gamma$, then $S(p,q, p, k, k, \alpha, \beta, \gamma) \simeq S(p, q, p, k, k, \gamma, \beta, \alpha)$. 

We will now show that distinct tuple of parameters in $\Sigma_D$ correspond to distinct surfaces. Let $v:= (p, q, k, l, m, \alpha, \beta)$ and $v':=(p', q', k', l',m', \alpha', \beta') \in \Sigma_D$ parametrize the same surface $S$. We will show $v = v'$. Comparing the two parametrizations of the heights of the longest cylinder, we get $q =q'$. 

If the lengths of the 3 cylinders are all distinct, comparing the different parametrizations of these lenghts give us $k= k'$ and $l =l'$. Similarly, comparing the parametrizations of the cylinder heights then gives $p = p'$, $r = r'$.  Now obtain a surface $S_0$ from $S$ by shearing the shortest cylinder by $-\alpha$. The new parameters will be $(p, q, r, k, l, 0, \beta, \gamma)$ and $(p, q, r, k, l, |\alpha' - \alpha|, \beta', \gamma')$.  $S_0$ has a vertical saddle connection of length $p$ that is transverse to the core curve of the shortest cylinder. Since there is only one singularity in each of the boundaries of the shortest cylinder, there can be no nontrivial shear for this saddle connection to exist. Hence, $\alpha - \alpha' = 0 \implies \alpha = \alpha'$. A similar technique yields $\gamma = \gamma'$. To prove $\beta = \beta'$ we obtain $S_{00}$ from $S_0$ by shearing the longest cylinder by $-\beta$, to get parameters $(p, q, r, k, l, 0, 0, \gamma)$ and $(p, q, r, k, l, 0, \beta' - \beta, \gamma)$.$S_{00}$ has exactly two vertical saddle connections that are of length $q$ and are transverse to the core curve of the longest cylinder. For both saddle connections to exist, $\beta' - \beta = 0 \implies \beta = \beta'$.  See Figure \ref{fig:typeivsingledouble}
\begin{figure}
\begin{tabular}{ccc}
\includegraphics[scale=0.22]{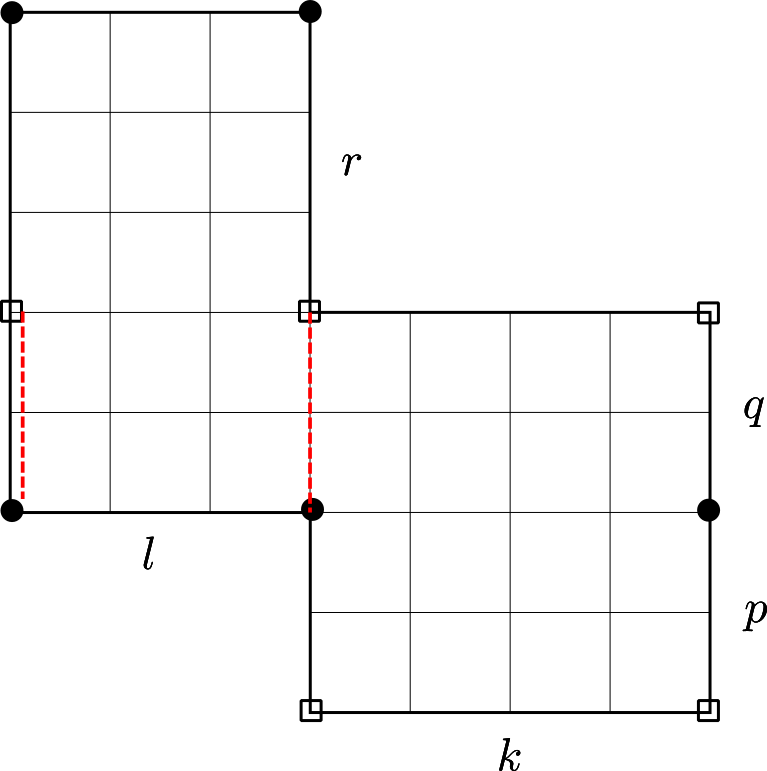} & \hspace{1cm} & \includegraphics[scale=0.22]{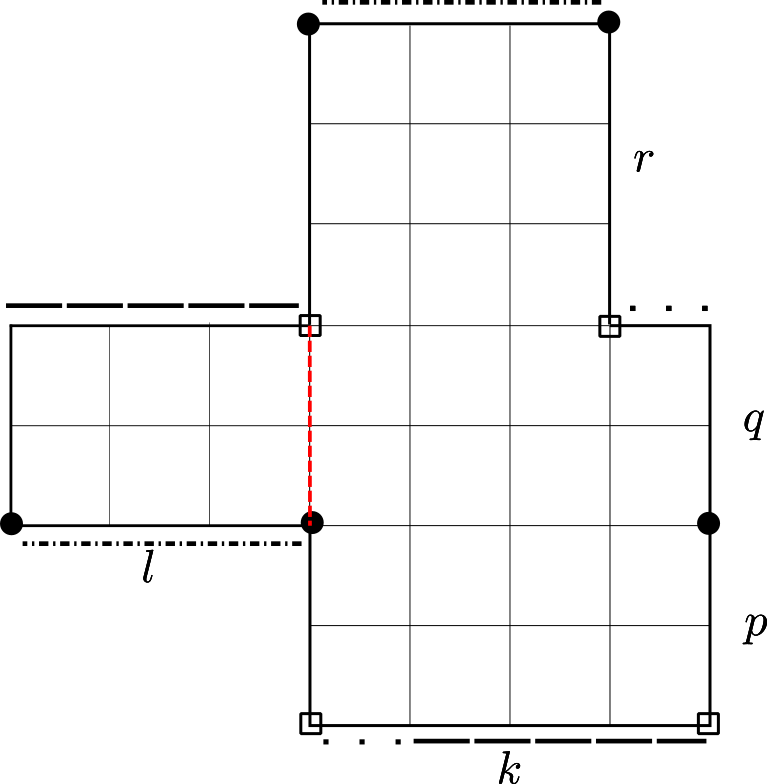}
\end{tabular}
\caption{Diagram D argument for uniqueness of the shear parameter of the longest cylinder. Left: $S_{00}$ with two special vertical saddle connections. Right: If shear is $|\beta-\beta'| = l$, then there is only one special vertical saddle connection.}
\label{fig:typeivsingledouble}
\end{figure}

If the length of the two shortest cylinders is the same, then $k = k' = l' = l$ is this length. Comparing the parametrization of heights of these cylinders we get $\{p, r\} = \{p', r'\}$, but noting that $p \leq r$ and $p' \leq r'$ for this case gives us $p = p'$ and $r = r'$. If $p < r$, similar arguments to above can be used to prove $\alpha = \alpha'$ and $\beta = \beta'$. To then see that $\beta = \beta'$, we obtain a surface $S_0$ from $S$ by shearing the smaller cylinders $-\alpha$ and $-\gamma$ respectively and the longer cylinder $-\beta$. Then, $S_0$ has parameters $(p, q, r, k, l, 0, 0, 0)$ and $(p, q, r, k, l, 0, 0, \beta' - \beta)$, so that $S_0$ has two special vertical saddle connections of length $q$ that are transverse to the core curve of the longest cylinder.

\medskip
\noindent
\begin{minipage}{3.5in}
Existence of these saddle connections asserts that either $\beta' - \beta = 0$ or $\beta' - \beta = k$. If $\beta'-\beta = k$, the surface formed by those parameters has a single vertical cylinder, a contradiction. See Figure on the side. The directed vertical curve is the core curve of the vertical cylinder. If $p = r = r' = p'$, consider the shortest saddle connection with vertical holonomy $p$ that is transverse to the core curve of one of the short cylinders in $S$. Since $\alpha \leq \gamma$ and $\alpha' \leq \gamma'$, this saddle connection will have holonomy vector $(\alpha, p)$ in one parametrization, and $(\alpha', p)$ in the other parametrization. Hence, we conclude $\alpha = \alpha'$. Arguments similar to ones that have been earlier presented in this proof for the other shear parameters will yield $\gamma = \gamma'$ and $\beta = \beta'$. 
\end{minipage}\hspace{0.3cm}
\begin{minipage}{3in}
\begin{center}
\includegraphics[scale=0.2]{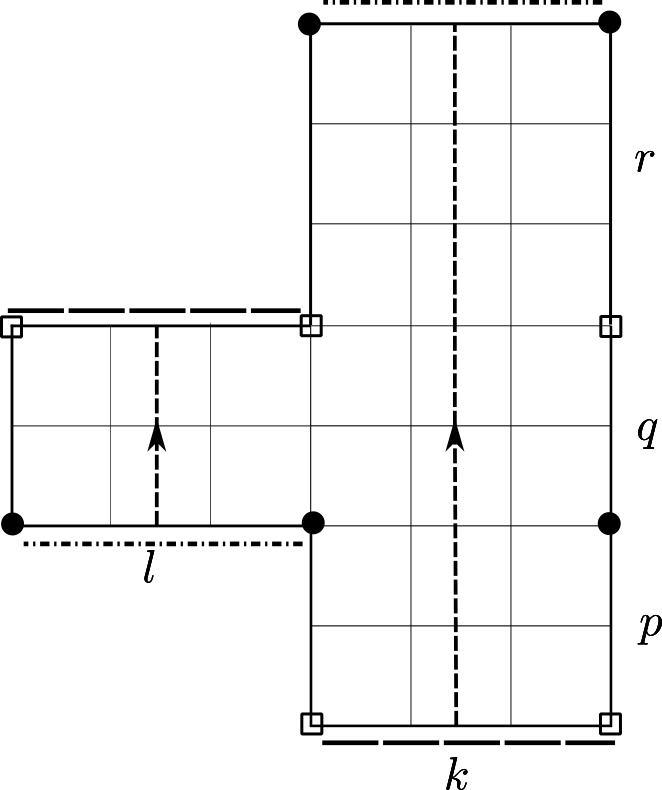}
\end{center}

\medskip
The case when $\beta'-\beta = k$ and the surface decomposes into a single cylinder in the vertical direction. \end{minipage}

\medskip

\end{proof}

\newpage
\section{Primitivity criteria by cylinder diagram}\label{sec:primcriterion}

In this section we provide proofs for the primitivity criteria stated in Lemma \ref{lem:allprimitivity} for cylinder diagrams A and C, and note that the proof for the primitivity conditions for cylinder diagrams B and D are similar to that of C. 
\begin{lemma}[Primitivity in Cylinder Diagram A]\label{prop:typeIprimitivity}
A square-tiled surface in $\calH(1,1)$ with cylinder diagram A and parametrized by $(p, j, k, l, m, \alpha)$ is primitive if and only if $p=1$ and $(j+k) \wedge (k+l) \wedge n = 1$.
\end{lemma}
\begin{proof}
Give square-tiled surface with parameters $(p, j, k, l, m, \alpha)$, it suffices to show that the statement is true for $\alpha = 0$ since primitivity of a square-tiled surface is an $SL(2, \ZZ)$ orbit invariant. Hence let $S:= S(p,j,k,l,m,0)$ be given.

$(\Rightarrow)$

Let $S$ be primitive. From Proposition \ref{prop:primsurfprimgroup} we know that $G:=\Mon(S)$ is primitive. Fix a labelling of $S$, and let $\sigma$ be the right permutation and $\tau$ be the top permutation associated to the labelled $S$. Since $S$ is made up of a maximal single horizontal cylinder, we know $\sigma$ is a product of $k$ disjoint cycles of equal length so that

$$ \sigma = u_1 \dots u_k$$ with $u_1 = (x_1 \dots x_a  z_1 \dots z_b y_1 \dots y_c t_1 \dots t_d)$ where $x_1, y_1, z_1$ and $t_1$ are so that $[\sigma, \tau] = (x_1 y_1)(z_1t_1)$.

\begin{figure}[!ht!]
\includegraphics[scale=0.25]{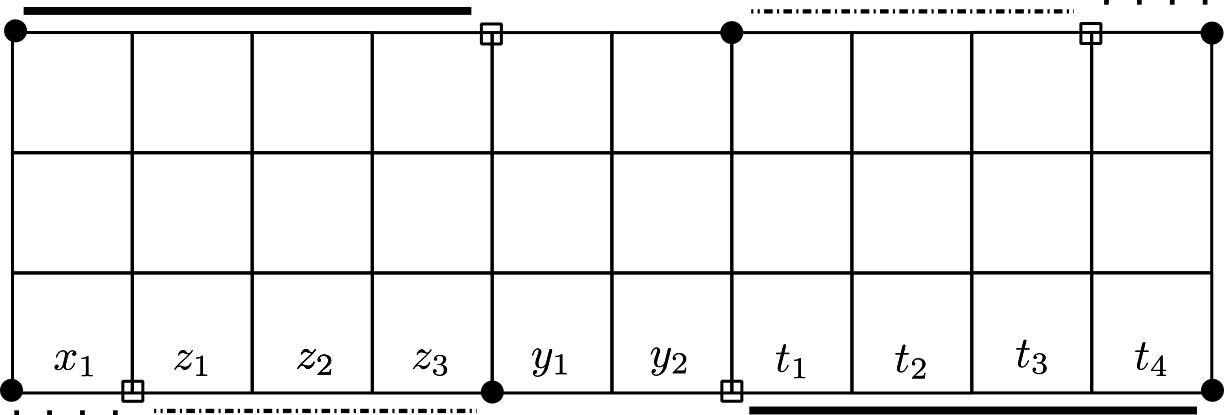}
\caption{A 1-cylinder surface in $\calH(1,1)$ with $p = 3, j= 1, k=2, l=2, m=4, \alpha = 9$}
\end{figure}

Define $\Delta = \supp(u_1)$. We will show that $\Delta$ is a block for $G$. 

Let $w \in G$. Since $\sigma$ and $\tau$ generate $G$, $w$ is written as a product of powers of $\sigma$ and $\tau$. Note that $\sigma (\supp(u_i)) = \supp (u_i)$ and $\tau(\supp(u_i)) = \supp(u_{i+1\modd k})$. But $\supp(u_i) \cap \supp(u_j) = \emptyset$ when $i \neq j$. So, either $w(\Delta) = \Delta$ or $w(\Delta) \cap \Delta = \emptyset$. As $w \in G$ was arbitrary, this implies that $\Delta$ is a block for $G$. As $\Delta$ contains more than one element (for instance $x_1$ and $z_1$), and as $G$ is primitive, we get that $\Delta = \{1, \dots, n\}$. Therefore, $p = 1$.

Now, assume $D = (j+k) \wedge (k+l) \wedge n$. We will show that $D = 1$. 

We rename $(x_1 \dots x_j z_1 \dots z_k y_1 \dots y_l t_1 \dots t_m)$ as $(1,2,,\dots, n)$. Note that $D|n$.  Next, consider the set

$$ S_0 : =\{ D, 2D, 3D, \dots, mD\}$$ where $mD =  n$. Similarly, define 
$$ S_i = \{D+i, 2D+i, 3D+i, \dots, mD +(i \modd n)\}$$

Note that for $a \in S_i$, $a+ lD \in S_i$ for all $i$ and $l$.

We then have, $\sigma(S_i) = S_{i+1 \modd D}$. Moreovoer, we claim that $\tau (S_i) = S_i$. Indeed, for $i \in S_j$, 

$$ \tau(i) = \begin{cases} i & \text{for }  1\leq i \leq j \\
 i + k+l & \text{for } j < i \leq j+m \\
 i + k-m  & \text{for } j+m < i \leq j+m+l \\
 i - (l+m)  & \text{for } j + m + l < i \leq n\\
 \end{cases}$$ 
 As $k+l$, $l+m$, and $k-m$ are all divisible by $D$, we see that in each case $\tau(i) \in S_j$. 
 \begin{figure}[!ht!]
 \includegraphics[scale=0.3]{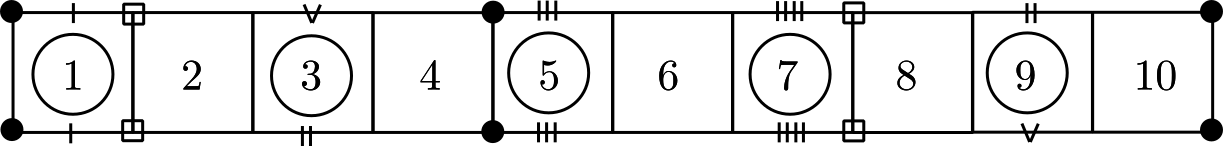}
 \caption{1-cylinder surface in $\calH(1,1)$ with $j=1, k=3, l=3, m=3$. The circled numbers indicate the set $S_1$ as defined in the proof. The identifications indicate the image of $S_1$ under $\tau$.}
 \end{figure}
 Hence, for $w \in G$, either $w(S_i) = S_i$ or $w(S_i) \cap S_i = \emptyset$ which implies that $S_i$ is a block for $G$ for all $i$. As $G$ is primitive, $S_i = \{1, \dots, n\}$ and $D = 1$.

$(\Leftarrow)$ 

Let $S$ be a one-cylinder square-tiled surface in $\calH(1,1)$ parametrized by $(1, j, k, l, m, 0)$. Assume that $(j+k) \wedge (k+l) \wedge n =1 = (j+k) \wedge (k+l) \wedge (l+m) \wedge (m+j)$

Note that $(j+k), (k+l), (l+m), (m+j)$ are lengths of horizontal simple closed curves curves based at singularities. Since the height of the cylinder is 1, there exists a vertical simple closed curve of length 1, based at a singularity. Hence, $\AbsPer(S)$ contains the vectors $(j+k, 0)$, $(l+m, 0)$, $(l+k, 0)$, $(m+j, 0)$ and $(0,1)$ as these are absolute periods of the curves just mentioned. Hence, since  $(1,0) \in \AbsPer(S)$. 
As $(1,0)$, $(0,1) \in \AbsPer(S)$, $\AbsPer(S) = \ZZ^2$. By Lemma \ref{lem:latticeprim}, $S$ is primitive. \end{proof}

The general strategy for cylinder diagrams $B$, $C$ and $D$ is the same. To prove that the number theoretic conditions are necessary for primitivity, we first assume that they are not fulfilled for a given surface $S$. Then, using Lemma \ref{lem:lattice1}we conclude that the $\AbsPer(S)$ is a sublattice of $\ZZ^2$. We next label $S$ and consider the set of labels of squares whose bottom left corner is in $\AbsPer(S)$. We show that this special set of labels forms a block for $\Mon(S)$, proving that $\Mon(S)$ is non-primitive. As $S$ is connected, this then implies, using Proposition \ref{prop:primsurfprimgroup} that $S$ is not primitive. 

To prove that the number theoretic conditions are sufficient, we simply use Lemma \ref{lem:latticeprim}.

We next present the above outline carried out in detail for surfaces with cylinder diagram $C$. 

\begin{lemma}[Primitivity in Cylinder Diagram C]\label{typeIIIprim}
A square-tiled surface in $\calH(1,1)$ with cylinder diagram C and parametrized by $(p, q, k, l, m, \alpha, \beta)$ is primitive if and only if 
$$ p \wedge q =1 \hspace{1cm} \text{ and } \hspace{1cm} (k +l) \wedge (l+ m) \wedge (p\beta - q \alpha) = 1$$
\end{lemma}
\begin{proof}

$(\Rightarrow)$ Let $S = S(p,q,k,l,m,\alpha, \beta)$ be a surface of diagram C in $\calH(2)$.  Assume $p \wedge q = 1 $ and $(k +l) \wedge (l+ m) \wedge (p\beta - q \alpha) = 1$ are not satisfied. By Lemma \ref{lem:lattice1} this means that $(\alpha, p), (\beta, q), (k+l, 0),(l+m,0)$ generate some lattice $L$ that does not coincide with $\ZZ^2$. We will then show that $G = \ideal{\sigma, \tau}$ is not primitive where $\sigma$ and $\tau$ are the right and top permutations associated to $S$.

\begin{figure}[htbp]
   \begin{center}
   \begin{tikzpicture}[scale=.4, >=latex,shorten >=0pt,shorten <=0pt,line width=1pt,bend angle=20, 
   		circ/.style={circle,inner sep=2pt,draw=black!40,rounded corners, text centered},
   		tblue/.style={dashed,blue,rounded corners}]
	
      	\draw[help lines] (-15,-8) grid (15,6);
      	\node[left] at (-8.8,-6.3){\scriptsize $(0,0)$};
      
	\begin{scope}
      	\clip (-10,-6) -- ++(2,6) -- ++(-4,0) -- ++(4,4) -- ++(12,0) -- ++(-4,-4) -- ++(8,0) -- ++(-2,-6) --++ (-16,0);
	\foreach \x in {-8,...,8}
		\foreach \y in {-8,...,8}
			\fill[green!60!white] (\x*4 + \y*2, \y*2) rectangle +(1,1);
	\end{scope}

      	\clip (-15.5,-8.5) rectangle (15.5,7.5);
	\foreach \x in {-8,...,8}
		\foreach \y in {-4,...,8}
      			\fill[blue] (\x*4 + \y*2,\y*2) circle (5pt);

      	\path[->] 	(-10,-6) 	edge node[below,yshift=1pt]{\scriptsize $\overrightarrow{(k+l,0)}$} +(16,0)
      					edge node[above,xshift=-12pt,yshift=-6pt]{\scriptsize $\overrightarrow{(\alpha,p)}$} +(2,6)
			(-12,0) 	edge node[above,xshift=-10pt]{\scriptsize $\overrightarrow{(\beta,q)}$} +(4,4)
			(-8,4) 	edge node[above,xshift=0pt,yshift=-1pt]{\scriptsize $\overrightarrow{(m+l,0)}$} +(12,0);
			
	\draw (4,4) -- ++(-4,-4) -- ++(8,0) -- ++(-2,-6);
	\draw (-12,0) -- ++(4,0);
   \end{tikzpicture}
   \captionwidth=14cm 
   \caption{The marked points belong to the lattice $L$ generated by the vectors with coordinates $(\alpha,p)$, $(\beta,q)$, $(k+l,0)$ and $(m+l,0)$. The colored squares of the surface form a nontrivial block for the group $G$.}
   \label{fig:twocylinderlattice}
   \end{center} 
\end{figure}
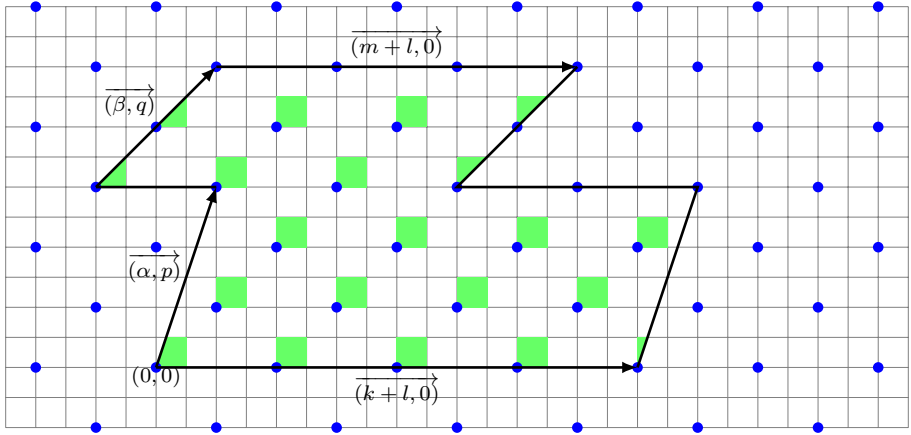

We take an unfolded representation of $S$, and we color the squares of the square-tiled surface which have their lower left vertex contained in the lattice $L$, and will show that the set of labels of the colored squares is a block for $G = \ideal{\sigma, \tau}$. We call this set $\Delta$.  

Consider an arbitrary square of the surface with lower left vertex $(s, t)$. Under $\sigma$ and $\tau$ this vertex is sent to:
\begin{gather}
(s+1, t) - \epsilon_1(l+m,0)-\epsilon_2(k+l, 0)\\
(s, t+1) - \epsilon_3(\alpha+\beta, p+q) - \epsilon_4(\alpha, p) - \epsilon_5(\beta, q)
\end{gather}
with $\epsilon_i \in \{0,1\}$

Now, we consider the action of an element $ g \in  G = \ideal{\sigma, \tau}$ on the lower left vertices of the colored squares. Given an element of $G$ note that up to a linear combination of $\epsilon_1(l+m,0),\epsilon_2(k+l, 0),\epsilon_3(\alpha+\beta, p+q),\epsilon_4(\alpha , p) , \epsilon_5(\beta, q)$, all the points in the lattice are translated by the same vector $v$. But, $\epsilon_1(l+m,0),\epsilon_2(k+l, 0),\epsilon_3(\alpha+\beta, p+q),\epsilon_4(\alpha , p) , \epsilon_5(\beta, q) \in L$. 

Hence, either $g(\Delta) = \Delta$ or $g(\Delta )\cap \Delta = \emptyset$. This implies $\Delta$ is a block for $G$.\\
$(\Leftarrow)$ Assume now that 
$$ p \wedge q =1 \hspace{1cm} \text{ and } \hspace{1cm} (k +l) \wedge (l+ m) \wedge (p\beta - q \alpha) = 1$$
By Lemma \ref{lem:lattice1} this means that  $\ideal{(\alpha, p), (\beta, q), (k+l, 0),(l+m,0)} = \ZZ^2$. But $(\alpha, p)$, $(\beta, q)$, $(k+l, 0),(l+m,0) \in \AbsPer(s).$ Hence, $\AbsPer(S) =  \ZZ^2$. By Lemma \ref{lem:latticeprim} we then get that $S$ is primitive. \end{proof}

Next, we provide a proof for the alternate parametrization of primitive $n$-square surfaces with cylinder diagram $A$. There exists a bijection between the set of parameters that encodes the lengths of the horizontal saddle connections and heights and shears of the cylinders with the parametrization that encodes the labels of the squares with cone points in their bottom left corners. 
\begin{proof}[Proof of Proposition \ref{prop:typeIreparam}] The goal is to show that the set of primitive $n$-square surfaces of $\calH(1,1)$ with cylinder diagram A is parametrized uniquely by the set
$$ \Omega := \{ (x, y, z, t) \in \NN^4| 1 \leq x < y < z < t \leq n, (z - x) \wedge (t-y) \wedge n = 1\}$$
Note that Propositions \ref{prop:typeIparam} and \ref{lem:allprimitivity}, give a unique parametrization of primitive $n$-square surfaces in $\calH(1,1)$ with cylinder diagram A. Let $\Sigma_A$ be as in Proposition \ref{prop:typeIparam} and let $\Sigma_A^P$ be the subset of $\Sigma_A$ that satisfies primitivity conditions imposed by Lemma \ref{lem:allprimitivity}. Then,  
We will define a function $f: \Sigma_A^P \rightarrow \Omega$, and prove that this function is bijective. Geometrically, $f$ is defined as follows: Given $(1,j, k, l, m, \alpha) \in \Sigma_A^P$, we draw the associated surface $S$ as in Figure \ref{fig:typeAalternateparam}.
\begin{figure}[h!!]
\includegraphics[scale=0.3]{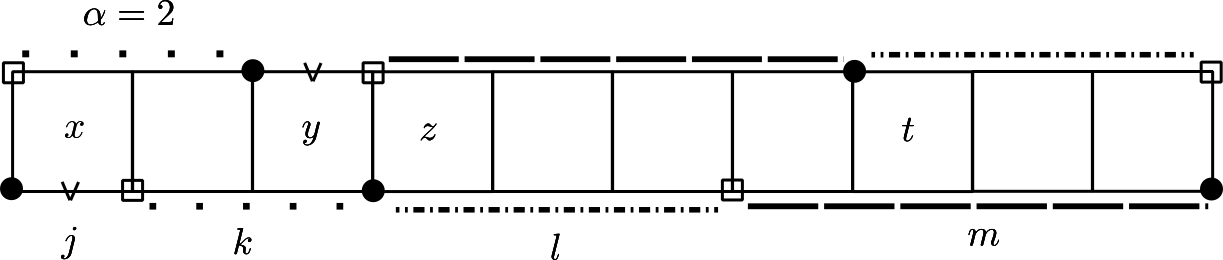}
\caption{A surface with cylinder diagram A given by parameters (1, 1, 2, 3, 4, 2). Note that the shear $(\alpha) = 2$ so that $x = 1$,$y=3$, $z = 4$ and $t = 8$.}
\label{fig:typeAalternateparam}
\end{figure}

We label the squares 1 through $n$ from left to right, and then set $1 \leq x < y < z < t \leq n$ to be the label of the squares with top corners as singularities.
Formally, given $(1, j, k, l, m, \alpha) \in \Sigma_A^P$, $f(1, j, k, l, m, \alpha) = (x,y,z,t)$ where $x < y < z < t \in \{\alpha+1, \alpha+1+j, \alpha+1+j+m, \alpha+1+j+m+l\}$ (all taken$\mod n$) and depending on the value of $\alpha$, the ordering on the set varies. However, for all possible values of $\alpha$, once checks that we will have $(z-x) \wedge (t-y) \wedge n= 1$ so that $f$ is well defined.

We will now describe a map $g$ from $\Omega$ to $\Sigma_A^P$. Given $(x,y,z,t) \in \Omega$, let $I:= \{y-x, z-y, t-z, n-t+x\}$ with a cyclic ordering by $(y-x, z-y, t-z, n-t+x)$. Let $J := \{ j, k, l, m\}$ be with a cyclic ordering $(j, m, l, k)$ too. Then, $I$ can have 1, 2 or 3 smallest elements. If $I$ has a unique smallest element, we set $j$ to be that element, and $m, l, k$ to be the rest of $I$ as per the cylic ordering on $I$ and $J$. In each of the cases, once one element of $J$ is set to be an element of $I$, the rest are determined by the cyclic ordering. If $I$ has 2 consecutive smallest elements $(a, b)$, then set $(k, j) = (a, b)$. If $I$ has 2 non-consecutive smallest elements, then $I$ also has a unique largest element, otherwise the GCD condition in $\Omega$ will not be satisfied. Set $m$ to be the largest element. If $I$ has 3 smallest elements, then $I$ has a unique largest element. Again set $m$ to be this element. We now pick $\alpha$. When $y-x = j, m, l, k$, set and solve $\alpha +1 = x, x-j, x-j-m, y$ respectively to get $\alpha$ in each case. 

One then checks that the map just described is an inverse to $f$. For instance, assume $(x, y, z, t) \in \Omega$, and assume we are in the case that $I$ has two smallest non-consecutive elements $(y-x, t-z)$ and that $z-y$ is the largest. Then, $g(x, y, z, t) = (1,y-x,n-t+x ,t-z, z-y, \alpha)$ with $\alpha = x-1$. Then 
$\alpha+ 1 < \alpha+1+y-x <  \alpha + 1+ y-x+ z-y < \alpha+1+y-x+z-y+t-z$ so that
\begin{align*}
f(g(x, y, z,t)) &= f(1,y-x,n-t+x ,t-z, z-y, x-1)\\ &= (\alpha + 1, \alpha+1+y-x, \alpha + 1+ y-x+ z-y, \alpha+1+y-x+z-y+t-z) = (x, y, z, t)
\end{align*}
Hence, $f$ is a bijection.
\end{proof}

\end{document}